\documentclass[11pt,reqno]{amsart}
\usepackage[margin=2.8cm]{geometry}
\pdfoutput=1

\usepackage[english]{babel}
\usepackage[utf8]{inputenc}

\usepackage{graphicx}				
\usepackage{amsfonts, amsmath, amsthm, amssymb}
\usepackage{mathtools, mathrsfs}
\usepackage[shortlabels]{enumitem}
\usepackage{bigints}
\usepackage[final]{pdfpages}
\usepackage{framed}
\usepackage[11pt]{moresize}
\usepackage[font=footnotesize,labelfont=bf]{caption}
\usepackage{subcaption}
\usepackage{tikz}
\usepackage{cite}
\usepackage{wrapfig}
\usepackage{bigints}
\usepackage{afterpage}
\usepackage{todonotes}
\usepackage{setspace}
\newcommand{\e}{\varepsilon}
\newcommand{\f}{\frac}
\newcommand{\im}{\implies}
\newcommand{\ra}{\rightarrow}

\newcommand{\la}{\lambda}
\newcommand{\al}{\alpha}
\newcommand{\be}{\beta}

\newcommand{\p}{\partial}
\newcommand{\lra}{\longrightarrow}

\newcommand{\w}{\omega}

\newcommand{\bu}{\mathbf{u}}

\newcommand{\bv}{\mathbf v}

\newcommand{\C}{\mathbb{C}}

\newcommand{\R}{\mathbb{R}}

\newcommand{\Z}{\mathbb{Z}}
\newcommand{\N}{\mathbb{N}}

\newcommand{\cL}{\mathcal{L}}

\newcommand{\cK}{\mathcal{K}}
\newcommand{\cD}{\mathcal{D}}

\newcommand{\cT}{\mathcal{T}}

\newcommand{\cZ}{\mathcal{Z}}

\newcommand{\sgn}{\operatorname{sign}}

\newcommand{\ran}{\operatorname{ran}}
\newcommand{\gker}{\operatorname{gker}}
\newcommand{\imp}{\operatorname{Im}}

\newcommand{\spec}{\operatorname{Spec}}

\newcommand{\dom}{\operatorname{dom}}
\newcommand{\spn}{\operatorname{Span}}

\DeclareMathOperator{\Mas}{Mas}

\DeclareMathOperator{\Tr}{Tr}

\newcommand{\graph}{\operatorname{graph}}

\newcommand{\pde}[2]{\frac{\partial #1}{\partial #2 }} 
\newcommand{\pdes}[2]{\frac{\partial^2 #1}{\partial #2^2 }} 
\newcommand{\de}[2]{\frac{d #1}{d #2 }} 
\newcommand{\des}[2]{\frac{d^2 #1}{d #2^2 }} 

\newtheorem{lemma}{Lemma}[section]

\usepackage[colorlinks=true,backref=page]{hyperref}
\usepackage[capitalise]{cleveref}

\newtheorem{prop}[lemma]{Proposition}
\newtheorem{theorem}[lemma]{Theorem}
\newtheorem{cor}[lemma]{Corollary}
\newtheorem{hypo}[lemma]{Hypothesis}
\newtheorem{define}[lemma]{Definition}

\theoremstyle{definition}

\newtheorem{ex}[lemma]{Example}
\newtheorem{rem}[lemma]{Remark}

\newlist{enumhypo}{enumerate}{1}
\setlist[enumhypo]{label=(\roman*),	ref=\thehypo(\roman*),partopsep=0pt, topsep=0pt, parsep=0pt}

\crefname{rem}{Remark}{Remarks}
\crefname{conj}{Conjecture}{Conjectures}
\crefname{hypo}{Hypothesis}{Hypotheses}
\crefname{enumhypoi}{Hypothesis}{Hypotheses}
\crefname{theorem}{Theorem}{Theorems}
\crefname{cor}{Corollary}{Corollaries}
\crefname{prop}{Proposition}{Propositions}
\crefname{define}{Definition}{Definitions}

\definecolor{skyblue}{rgb}{0.85,0.85,1}

\begin{document}
	\title[Hamiltonian spectral flows and the Maslov index]{Hamiltonian spectral flows, the Maslov index, and the stability of standing waves in the nonlinear Schr\"odinger equation}
	\author{G. Cox, M. Curran, Y. Latushkin, R. Marangell}
	\date{\today}
	\begin{abstract}
		We use the Maslov index to study the spectrum of a class of linear Hamiltonian differential operators. We provide a lower bound on the number of positive real eigenvalues, which includes a contribution to the Maslov index from a non-regular crossing. A close study of the eigenvalue curves, which represent the evolution of the eigenvalues as the domain is shrunk or expanded, yields formulas for their concavity at the non-regular crossing in terms of the corresponding Jordan chains. This enables the computation of the Maslov index at such a crossing via a homotopy argument. We apply our theory to study the spectral (in)stability of standing waves in the nonlinear Schr\"odinger equation on a compact interval. We derive stability results in the spirit of the Jones--Grillakis instability theorem and the Vakhitov--Kolokolov criterion, both originally formulated on the real line. A fundamental difference upon passing from the real line to the compact interval is the loss of translational invariance, in which case the zero eigenvalue of the linearised operator is (typically) geometrically simple. Consequently, the stability results differ depending on the boundary conditions satisfied by the wave. We compare our lower bound to existing results involving constrained eigenvalue counts, finding a direct relationship between the correction factors found therein and the objects of our analysis, including the second-order Maslov crossing form.
		
	\end{abstract}

	\maketitle
	\parskip=0em
	{\tableofcontents}
	\numberwithin{equation}{section}
	\allowdisplaybreaks

	\section{Introduction}
	\parskip=1em
	
	We use the Maslov index to study the real spectrum of Hamiltonian differential operators of the form
	\begin{equation*}
		N = \begin{pmatrix}
			0 & -L_- \\ L_+ & 0
		\end{pmatrix},
	\end{equation*}
	where $L_\pm$ are scalar-valued Schr\"odinger operators with arbitrary $C^2$ potentials on a compact interval $[0,\ell]$. In particular, we provide a lower bound on the number of positive real eigenvalues of the operator $N$ (\cref{thm:N_bound}). 
	
	Our approach is to restrict $N$ to a subinterval $[0,s\ell]$, $s\in(0,1]$, and, rescaling back to $[0,\ell]$, study the $s$-dependent spectrum of the one-parameter family of operators in the spatial parameter $s$. We are thus led to a characterisation of the eigenvalues of the rescaled operators as a locus of points in the $\lambda s$-plane (with $\la$ the spectral parameter), which we refer to as {\em eigenvalue curves}. We interpret the eigenvalue curves as loci of intersections, or \textit{crossings}, of a path in the manifold of Lagrangian planes with a certain codimension one subvariety. This affords the use of the Maslov index, a signed count of such crossings. Formulas for the concavity of the eigenvalue curves are given (\cref{thm:sL,thm:sL:general,thm:touching}), and are used to compute a correction term appearing in the lower bound in \cref{thm:N_bound}.

	Operators of the form of $N$ arise in the linearisation about a standing wave solution $\widehat\psi(x,t) = e^{i\be t} \phi(x)$ of the nonlinear Schr\"odinger (NLS) equation
	\begin{equation}\label{nls_eqn}
		i\psi_t=\psi_{xx} + f\left (|\psi|^2\right )\psi,
	\end{equation}	
	where $\psi: [0,\ell]\times [0,\infty) \lra \C$, the nonlinearity $f:\R^+\lra \R$  is a $C^3$ function and $\be\in\R$ is the temporal frequency. The wave around which we linearise is said to be \textit{spectrally unstable} if there exists spectrum of $N$ in the open right half plane, and \emph{spectrally stable} otherwise.   By applying \cref{thm:N_bound}, we establish stability criteria for standing waves in the NLS equation on a compact interval subject to perturbations satisfying Dirichlet boundary conditions. Namely, we derive analogues of the \textit{Jones--Grillakis instability theorem} (\cref{cor:JonesGrillakis}) and the \textit{Vakhitov--Kolokolov (VK) criterion} (\cref{thm:VK}). While \cref{cor:JonesGrillakis} is also a consequence of the abstract result of \cite[Theorem 3.2]{KP12art}, \cref{thm:VK}, which makes use of the concavity formulas of \cref{thm:sL}, appears to be new for the case of the compact interval.   These two stability results actually remain valid for a spatially dependent nonlinearity $f(x,|\psi|^2)$; see \cref{rem:nonautonomous}.

	Along the way, we find \emph{Hadamard-type} formulas for the slope of the eigenvalue curves as the ratio of certain quadratic forms, called \textit{crossing forms}, whose signatures locally determine the Maslov index (\cref{prop:M,cor:Hadamard}). Variational formulas for the eigenvalues of boundary value problems with respect to perturbation of the domain are classical and go back to the work of Hadamard \cite{hadamard1908}, Rayleigh \cite{Rayleigh45} and Rellich \cite{Rellich69}; see also \cite{DanHenry,Grinfeld10} and \cite[\S VII.6.5]{Kato}. Recently such formulas have been given in terms of the (Maslov) crossing form for families of Schr\"odinger \cite{LSHad17, LS20metric} and abstract selfadjoint operators \cite{LS20first}. Our formulas agree with and build on those found therein. 
	
	We also encounter a \emph{non-regular} crossing when $\la=0$, corresponding to a degeneracy of the associated crossing form and points of zero slope for the eigenvalue curves. Geometrically, this corresponds to the Lagrangian path tangentially intersecting the relevant codimension one subvariety. Some care is then required in order to compute the Maslov index, and it is a key feature of the current work that we are able to do so (\cref{thm:compute_c}). In particular, it is sufficient to know the concavity of the eigenvalue curve through the non-regular crossing, as well as whether or not the operators $L_+$ and $L_-$ have a nontrivial kernel. To the best of our knowledge, no such computation has previously been made in the literature. Analysing the non-regular crossing in the context of the NLS equation leads to stability criteria that resemble the VK criterion in certain cases, furnishing an interesting connection between the concavity of the eigenvalue curve at the non-regular crossing, the Maslov index there, and the classical VK result; see \cref{sec:applications}.
	
	In the case when the spatial domain is the entire real line, if zero is a hyperbolic fixed point of the standing wave equation
	\begin{align}\label{standwave1}
		\phi_{xx} + f(\phi^{2})\phi + \beta \phi = 0
	\end{align} 
	and there exists an orbit that is homoclinic to it in the phase plane, a localised solution to \eqref{nls_eqn} exists and belongs to $L^{2}(\R)$ for all time.  In this case $L_+$ and $L_-$, which are unbounded operators on $L^2(\R)$, both have a nontrivial kernel. Indeed, the stationary state $\phi$ and its derivative $\phi_x$ satisfy $L_-\phi=0$ (the stationary equation \eqref{standwave1}) and $L_+\phi_x=0$ (the associated variational equation) respectively, and decay exponentially as $x\ra \pm \infty$. By the results of Jones \cite{J88} and Grillakis \cite{Grill88}, one then has that if $P-Q \neq 0 ,1$, where $P$ and $Q$ are the numbers of negative eigenvalues (or \textit{Morse indices}) of $L_+$ and $L_-$, then $N$ has at least one positive real eigenvalue, and hence the standing wave solution to \eqref{nls_eqn} is unstable. In the edge case when $P=1$ and $Q=0$, the results of Vakhitov and Kolokolov \cite{VK73} and Grillakis, Shatah and Strauss \cite{GSS87,GSS90} dictate that the wave is spectrally (and orbitally) stable if the $\be$-derivative of the mass of the wave
	\begin{equation}
		\label{intro:VK}
		\pde{}{\beta} \int_{-\infty}^{\infty} \phi^2 \,dx,
	\end{equation}
	is negative, and spectrally unstable if \eqref{intro:VK} is positive (see \cite[Theorem 4.4, p.215]{pelinovsky}).
	
	One of the key differences upon passing from the real line to the compact interval is that, generically, the operators $L_+$ and $L_-$ (equipped with Dirichlet boundary conditions) do not simultaneously have a nontrivial kernel. Depending on the boundary conditions satisfied by the wave profile $\phi$, typically zero will lie in the spectrum of either $L_+$ or $L_-$ (or neither). A physical reason for this is the loss of translational invariance, which manifests in the failure of the relevant boundary conditions of arbitrary translates of $\phi$. As a consequence, our stability results (\cref{cor:JonesGrillakis,thm:VK}) will differ depending on which of the operators $L_\pm$ has a nontrivial kernel. In the case that $L_-$ has a nontrivial kernel, we can recover the integral expression \eqref{intro:VK} appearing in the classical VK criterion. Such a recovery is not possible when $L_{+}$ has a nontrivial kernel; for details, see the discussion in \cref{sec:VKrecover}.
	
	There is a large body of work relating the Morse index of a selfadjoint operator and its number of conjugate points (which was later interpreted as the Maslov index of an associated Lagrangian path), going back to the middle of last century \cite{A67,A85,Bott56,Duistermaat76,Edwards64,Smale65}. Most of these theorems can be viewed as generalisations of the classical Sturmian theory, and indeed in \cite{Bott56,Edwards64,Smale65} they are framed as such, where the nodal count of an eigenfunction indicates where in the sequence of eigenvalues the corresponding eigenvalue sits. Following on from Jones' seminal work \cite{J88}, the idea of using the Maslov index for spatially Hamiltonian systems to extrapolate temporal spectral information has proven quite fruitful in the ensuing years (see, for example, \cite{JLM13, CJLS14, CJM15, HS16,HLS18, LS18} and the references therein for a partial list of results).

	In more recent times, Deng and Jones in \cite{DJ11} (see also \cite{CJLS14,CJM15}), used the Maslov index to analyse second-order elliptic eigenvalue problems on bounded domains. An important feature of this analysis, as well as that of \cite{BCLJMS18,HS16,HLS18,HS22,HJK18}, is monotonicity of the Maslov index in the spectral parameter. Monotonicity also holds in the spatial parameter under certain boundary conditions \cite{CJLS14,HLS17,JLM13}. This property is convenient since it enables an \emph{equality} of the Morse index with the Maslov index of the Lagrangian path corresponding to $\la = 0$. Importantly, as in \cite{J88}, we do \emph{not} have monotonicity in either the spatial or the spectral parameter. However, the signature of crossings in the $s$-direction when $\la=0$ can always be accounted for, and, consequently, a nonzero Maslov index can nonetheless be used to detect a real, unstable eigenvalue, just as in \cite{JMS10,JMS12,JMS14,MRS20}. This lack of monotonicity thus leads to the \emph{inequality} in \cref{thm:N_bound}.

	Another feature in the aforementioned references, as well as in \cite{BJ95,CH07, CH14, CDB09part1, CDB09, CDB11part2,Corn19,CJ18, CJ20, Howard21} is a dynamical systems approach to eigenvalue problems. In these works, the eigenvalue equations associated with the linearised operators are Hamiltonian, or can be made Hamiltonian under a suitable change of variables. The critical feature of such systems is that they induce a symplectically invariant flow and hence preserve the manifold of Lagrangian planes, which affords the application of the Maslov index. For recent works where the Hamiltonian requirement is relaxed, see \cite{Corn19,CJ18,CJ20}. In \cite{CJ18,CJ20}, a change of variables is used to recover the Hamiltonian structure, and in \cite{Corn19} the system, while not Hamiltonian, still preserves the space of Lagrangian planes. For an example of where the Hamiltonian requirement is dropped altogether, see \cite{BCCJM21}.

	Existing results on the stability of standing wave solutions of \eqref{nls_eqn} on a compact spatial interval have been given for periodic solutions of \eqref{standwave1}, with (quasi)periodic perturbations, and predominantly for cubic focusing ($f(\phi^2) =\phi^2$) or defocusing ($f(\phi^2) =-\phi^2$) NLS. Rowlands in \cite{Rowlands74} studied the spectral stability of spatially periodic elliptic solutions to the cubic NLS, subject to long wavelength disturbances. Pava \cite{Pava07} showed that the Jacobi dnoidal solutions to cubic focusing NLS were orbitally stable with respect to co-periodic perturbations. In \cite{GH07_small}, Gallay and H\v{a}r\v{a}gus showed the orbital stability of spatially periodic and quasiperiodic travelling waves with complex-valued profile for small amplitude solutions in both the focusing and defocusing case. They extended this result to waves of arbitrary amplitude in \cite{GH07_orbital}. For the real-valued (cnoidal) waves, their orbital stability result is restricted to perturbations that are anti-periodic on a half period. This latter condition was done away with in \cite{IL08}, wherein Ivey and Lafortune undertook a spectral stability analysis of the cnoidal travelling wave solutions of the focusing NLS, showing stability with respect to co-periodic perturbations. In \cite{BDN11,GP15_I} the authors extend the orbital stability results for both real- and complex-valued wave profiles to the class of subharmonic perturbations (i.e. perturbations with period an integer multiple of the period of the wave profile) in the defocusing case. In \cite{DS17,DU20} the authors  examine the spectral stability of the elliptic solutions with respect to subharmonic perturbations in the focusing case. Unlike the above works, we are interested in the spectral stability of real-valued solutions of \eqref{standwave1}, for an arbitrary $C^{3}$ nonlinearity $f$, that are subject to perturbations satisfying Dirichlet boundary conditions. Moreover, as previously stated, many of our results hold for a spatially dependent $f$.

	Our theory can be extended in several possible directions. In particular, our theory should hold for the case of quasi-periodic boundary conditions on the perturbations, which is natural to consider given that many of the solutions $\phi$ to \eqref{standwave1} that satisfy Dirichlet boundary conditions are periodic. The Maslov index has already been used to develop eigenvalue counts for selfadjoint matrix-valued Schr\"odinger operators with such boundary conditions in \cite{JLM13,JLS17}. Our theory should also hold when the Schr\"odinger operators $L_\pm$ are selfadjoint and matrix-valued, and indeed in \cref{sec:proof1,sec:symplectic_view} many of our results are stated for the operator $N$ with an $n$-dimensional kernel to accommodate this scenario. Finally, while the analysis is significantly more involved, it should be possible to extend to the case where the spatial domain is multidimensional, as in \cite{CJM15,CJLS14,CoxMarz19}.

	The paper is organised as follows. In \cref{sec:setup} we set up the eigenvalue problem and state the main results. In \cref{sec:symplectic_view} we provide background material on the Maslov index, interpret the (real) eigenvalue problem symplectically and prove \cref{thm:N_bound}. In \cref{sec:proof1} we analyse the eigenvalue curves. After computing formulas for their derivatives and relating these to the Maslov crossing forms (\cref{prop:M,cor:Hadamard}), we compute their concavities at the zero eigenvalue (\cref{thm:sL:general,thm:touching}), facilitating the computation of the Maslov index at the non-regular crossing (\cref{thm:compute_c}).  {We conclude the section by confirming that the signature of the \emph{second}-order Maslov crossing form provides the correct contribution to the Maslov index at this crossing, which is consistent with \cite{DJ11}}.   In \cref{sec:applications} we provide some applications of \cref{thm:N_bound,thm:sL}. In particular, we prove \cref{cor:JonesGrillakis,cor:exact_count_Q=0} and \cref{thm:VK}. We also compute expressions for the concavity (at $s=1$) of the eigenvalue curve passing through $(\la,s)=(0,1)$ for linearised NLS, in each of the cases when $L_+$ and $L_-$ has a nontrivial kernel (\cref{prop:VK_compact,prop:concave_up}). In the latter case, we recover a compact-interval analogue of the classical VK criterion.  {We conclude the paper with a comparison of the lower bound in \cref{thm:N_bound} with existing results which make use of constrained eigenvalue counts. We find that the ``correction" terms appearing in our lower bound and others in the literature are equivalent (\cref{prop:KreinMaslov}), applying our formulas to provide new versions of the Hamiltonian--Krein index theorem in terms of the Maslov index (\cref{prop:HKM_fmls}).}  
	
	\textbf{Notation}: We let $I_n$ and $0_n$ denote the $n\times n$ identity and zero matrices respectively. We denote the canonical $2n\times 2n$ symplectic matrix and the first Pauli matrix by
	\begin{equation}\label{def:JSsig}
		J= \begin{pmatrix}
			0_n & -I_n \\ I_n & 0_n
		\end{pmatrix},
		\qquad 
		S= \begin{pmatrix}
			0 & 1 \\ 1 & 0
		\end{pmatrix},
	\end{equation}
	respectively. We let $\langle\cdot,\cdot\rangle$ and $\|\cdot\|$ denote the $L^2$ inner product and norm, respectively. Subscripts $s$ or $\la$ will indicate dependence of a quantity on these parameters (not derivatives). The spectrum of a linear operator $T$ will be denoted by $\spec(T)$, and its kernel by $\ker(T)$.

	\section{Set-up and statement of main results}\label{sec:setup}
	
	The basic set-up is an eigenvalue problem of the form
	\begin{equation}\label{eq:N_EVP_full}
		N\begin{pmatrix} u \\ v \end{pmatrix}  = \la \begin{pmatrix} u \\ v \end{pmatrix} , \qquad \begin{pmatrix} u(0)  \\ v(0) \end{pmatrix} = \begin{pmatrix} u(\ell) \\ v(\ell) \end{pmatrix} = \begin{pmatrix} 0 \\ 0 \end{pmatrix} ,
	\end{equation} 
	where $N$ is given by 
	\begin{equation}
		\label{Hamiltonian_N}
		N \coloneqq \begin{pmatrix}
			0 & -L_- \\ L_+ & 0
		\end{pmatrix}
	\end{equation}
	and $L_\pm$ are the Schr\"odinger operators 
	\begin{equation}\label{LplusLminus}
		L_+=-\partial_{xx}-g(x), \qquad
		L_-=-\partial_{xx}-h(x),  
	\end{equation}
	with $g$ and $h$ arbitrary functions in $C^2([0,\ell], \R)$. To be precise, we consider $N$ as an unbounded operator in $L^2(0,\ell)\times L^2(0,\ell)$ with dense domain 
	\begin{align}
		\dom(N) &=  \left ( H^2(0,\ell) \cap H^1_0(0,\ell) \right ) \times \left ( H^2(0,\ell) \cap H^1_0(0,\ell) \right )\subset L^2(0,\ell) \times L^2(0,\ell). \label{domN}
	\end{align}
	Hereafter, we drop the product notation on the relevant spaces; it will be clear from the context whether the functions are scalar- or vector-valued. An {\em eigenvalue} of $N$ is thus a value of $\la\in\C$ for which there exists a nontrivial solution $\mathbf{u}\coloneqq(u,v)^{\top}$ to the boundary value problem \eqref{eq:N_EVP_full}. Eigenvalues for the unbounded operators $L_\pm$, with dense domains
	\begin{align}
		\dom(L_\pm) &= H^2(0,\ell) \cap H^1_0(0,\ell) \subset  L^2(0,\ell), \label{domLpm}
	\end{align}
	are similarly defined. Note that the unbounded operators $L_\pm = L_\pm^*$ with domain \eqref{domLpm} are selfadjoint, while $N$ is not. 
	
	\begin{rem} \label{rem:operators} 
		Notationally, we will not distinguish between the formal differential expressions $N$ and $L_\pm$ and the unbounded operators with domains \eqref{domN} and \eqref{domLpm} whose spectra we wish to study. It will be clear from the context in what sense we refer to these objects. 
	\end{rem} 
	
	While it is possible for $N$ to have complex eigenvalues, we will restrict our analysis of \eqref{eq:N_EVP_full} to the case when $\la$ is real and positive. The existence of such an eigenvalue implies instability. On the other hand, there are cases where the spectrum of $N$ lies entirely on the real and imaginary axes, in which case the absence of a real positive eigenvalue implies stability; see \cref{thm:VK} for an example.
	
	Our first result is a lower bound for the number of positive real eigenvalues of $N$. It follows from an application of the Maslov index. The idea is to study the spectral problem in \eqref{eq:N_EVP_full} via a rescaling of the domain. We restrict \eqref{eq:N_EVP_full} to a family of subdomains $[0,s\ell]$ using a parameter $s\in(0,1]$,
	\begin{equation}\label{eq:N_EVP_restricted}
		N\bu = \lambda \bu, \quad \bu(0) = \bu(s\ell) = 0,
	\end{equation}
	and define a \textit{conjugate point} to be a value of $s$ for which there exists a nontrivial solution to \eqref{eq:N_EVP_restricted} with $\la=0$. We then deduce the existence of unstable eigenvalues of \eqref{eq:N_EVP_full} by counting conjugate points (via the Maslov index) as $s$ varies from 0 to 1. Defining the quantities
	\begin{align*}
		P &\coloneqq \# \{\text{negative eigenvalues of } L_+\}, \\
		Q &\coloneqq \# \{\text{negative eigenvalues of } L_-\}, \\
		n_+(N) &\coloneqq   \#\{\text{positive real eigenvalues of } N\},
	\end{align*}
	we have: 
	\begin{theorem}\label{thm:N_bound}
		Let $N$ be an operator as in \eqref{Hamiltonian_N}--\eqref{LplusLminus}. The number of positive real eigenvalues of $N$ satisfies
		\begin{equation}\label{eq:bound_positive_evals}
			n_+(N) \geq |P-Q-\mathfrak{c}|,
		\end{equation}
		where $\mathfrak{c}$ (given in \cref{defn:c}) is the total contribution to the Maslov index in the $s$ and $\la$ directions from the conjugate point at $s=1$. (If there is no such conjugate point, $\mathfrak{c}=0$.)
	\end{theorem}
	\begin{rem}
		One of the main results of this paper is that we are able to give explicit formulas for this so-called ``corner term" $\mathfrak{c}$ which has the property that $\mathfrak{c}\in\{-1,0,1\}$. The name derives from the location of the associated crossing in terms of the so-called \emph{Maslov box}. For precise statements see \cref{sec:symplectic_view,sec:proof1}, in particular \cref{thm:compute_c}.
	\end{rem}

	\begin{rem}
		In \eqref{eq:N_EVP_restricted} the symbol $N$ denotes a differential expression. For the associated unbounded operator we define
		\begin{equation}\label{}
			N|_{[0,s\ell]}\mathbf{u}\coloneqq N\mathbf{u}, \qquad \mathbf{u}\in\dom(N|_{[0,s\ell]}) =  H^2(0,s\ell) \cap H^1_0(0,s\ell) \subset L^2(0,s\ell),
		\end{equation}
		so that $\lambda \in \spec (N|_{[0,s\ell]})$ if and only if  \eqref{eq:N_EVP_restricted} has a non-trivial solution.
	\end{rem}
	
	 {
		
		\Cref{thm:N_bound} (the proof of which is given in \cref{subsec:proof_thm22}) is in the spirit of a number of lower bounds in the literature.  In contrast to \cite[Assumption 2.1(b)]{HK08}, we do not assume that the operators $L_\pm$ are invertible. If both $L_+$ and $L_-$ are invertible, it will follow that there is no conjugate point at $s=1$, and therefore $\mathfrak{c}=0$. In this case we recover the inequality in \cite[Theorem 2.25]{HK08}. The lower bound for $n_+(N)$ in the case when one or both of $L_+$ and $L_-$ has a nontrivial kernel has been studied in \cite[Thm 3.2]{KP12art}, \cite[Thm 5.6]{KM14}, \cite[Thm 2.3]{LZ22} and \cite[Thm 1.2]{Grill88}, to name a few; see also \cite[\S 7.1.3]{KapProm}. In these works, the authors typically project off the kernels of $L_+$ and $L_-$, and give the lower bound in terms of the associated constrained eigenvalue counts for $L_+$ and $L_-$. By contrast, we require no such projections. The constrained counts for $L_+$ and $L_-$ (given in the current work in \eqref{constrained}) involve the number of negative eigenvalues of certain matrices denoted $D_\pm$.  In \cref{subsec:comparison}, we will show that our ``correction" factor -- given by the corner term $\mathfrak{c}$ -- is equivalent to the ``correction" factor in \cite[Theorem 7.1.16]{KapProm}, given by the difference $n_-(D_+) - n_-(D_-)$ of negative indices of $D_+$ and $D_-$ (see \cref{prop:KreinMaslov}). Thus, \cref{thm:N_bound} together with \cref{prop:KreinMaslov} recovers \cite[Theorem 7.1.16]{KapProm}. The Maslov index interpretation afforded by $\mathfrak{c}$ is convenient because it provides a way of \emph{computing} the difference $n_-(D_+) - n_-(D_-)$. Namely, \eqref{eq:computingD+-} shows that the signs of $D_\pm$ (which in our set-up are scalars) are given by the signs of the concavities of the eigenvalue curves at $(\la,s) = (0,1)$.
	}

	Our main application will be to the linearisation of \eqref{nls_eqn} about a {\em standing wave} solution. This is a solution to \eqref{nls_eqn} of the form $\widehat\psi(x,t)=e^{i\beta t}\phi(x)$ for some $\beta\in\R$, where the real-valued wave profile or \textit{stationary state} $\phi:[0,\ell] \ra \R$ solves the time-independent equation
	\begin{equation}\label{standwave}
		\phi_{xx}+ f(\phi^2)\phi +\be \phi=0. 
	\end{equation}
	The results of this paper hold under fairly general boundary conditions on $\phi$. Two examples that we will often focus on are Dirichlet conditions
	\begin{equation}\label{eq:dirichletBC}
		\phi(0)=\phi(\ell)=0,
	\end{equation}
	or Neumann conditions 
	\begin{equation}\label{eq:NeumannBC}
		\phi'(0)=\phi'(\ell)=0.
	\end{equation}
	In these cases, one possible choice for the interval length $\ell$ is to fix a $T$-periodic solution to \eqref{standwave}, and to set $\ell = kT/2$ for some $k\in\N$. Some example phase portraits for \eqref{standwave} featuring periodic orbits are given in \cref{fig:phase_plane}. As an aside, note that the homoclinic orbits in \cref{fig:phaseA} correspond to strictly positive or negative localised solutions on $\R$.
	
	A natural question to ask is whether the standing wave $\widehat\psi$ is stable in time with respect to small perturbations in $\phi$. Substituting the perturbative solution
	\[
	\psi(x,t) =e^{i\be t} \left [\phi(x) + \e e^{\la t}( u(x)+i v(x))\right ]
	\]
	into \eqref{nls_eqn} and collecting $O(\e)$ terms, we arrive at the differential equations in \eqref{eq:N_EVP_full}, where
	\begin{gather}
		\label{NLS_potentials}
		\begin{aligned}
			g(x)&= 2f'(\phi^2(x))\phi^2(x)+f(\phi^2(x))+\be, \\
			h(x)&= f(\phi^2(x))+\be.
		\end{aligned}
	\end{gather}
	Then, subject to the class of perturbations $\mathbf{u}=(u,v)^\top$ that vanish at both endpoints, the standing wave $\widehat\psi$ is spectrally stable if the spectrum of the linearised operator $N$ is contained in the imaginary axis, since the eigenvalues of $N$ are symmetric with respect to the real and imaginary axes.

	When $\la=0$ the differential equations in \eqref{eq:N_EVP_full} decouple into two independent equations: $N\mathbf{u} = 0$ if and only if $L_+u=0$ and $L_-v=0$. Thus $\ker(N) = \ker(L_+)\oplus \ker(L_-)$, and $0\in\spec(N)$ if and only if $0\in\spec(L_+) \cup \spec(L_-)$. Furthermore, because the eigenvalues of the Sturm-Liouville operators $L_\pm$ are simple,
	\begin{gather}\label{eq:decouple_facts}
		\begin{aligned}
			\dim\ker(N) = 1 &\iff 0\in\spec(L_-)\triangle\spec(L_+), \\
			\dim\ker(N)=2  &\iff 0\in\spec(L_-)\cap\spec(L_+),
		\end{aligned}
	\end{gather}
	where $A \triangle B \coloneqq A\cup B  \setminus A \cap B$ denotes the symmetric difference. In our application to the stability of standing waves of \eqref{nls_eqn}, note that \eqref{standwave} is equivalent to $L_-\phi=0$, while autonomy of this equation yields $L_+\phi' = 0$. The boundary conditions satisfied by $\phi$ therefore influence whether $0\in\spec(L_\pm)$. For instance, if $\phi$ satisfies the Dirichlet conditions \eqref{eq:dirichletBC}, then $0\in\spec(L_-)$ with eigenfunction $\phi$, whereas if $\phi$ satisfies the Neumann conditions \eqref{eq:NeumannBC}, then $0\in\spec(L_+)$ with eigenfunction $\phi'$, provided $\phi$ is nonconstant. It is also possible that $0\notin \spec(L_+)\cup \spec(L_-)$ if, for example, more general Robin boundary conditions are imposed on $\phi$. 
	
	In any of these cases, that $L_+$ and $L_-$ have nontrivial kernel simultaneously is nongeneric, and so we make this an assumption when studying the stability of NLS standing waves.  {We stress that the general set-up of the paper is given by \eqref{eq:N_EVP_full}--\eqref{LplusLminus}, and the following hypothesis is \emph{not} assumed throughout; we will explicitly state whenever we make use of it.}

	\begin{hypo} \mbox{} \vspace{-0.6em}
		\label{hypo:NLS} 
		$N$ is of the form \eqref{Hamiltonian_N}--\eqref{LplusLminus}, where
		\begin{enumerate}[(i)]
			\item the potentials $g$ and $h$ come from the linearisation of the NLS equation \eqref{nls_eqn} about a standing wave $\widehat\psi$ (and hence are given by \eqref{NLS_potentials}), and
			\item $0\notin\spec(L_-)\cap \spec(L_+)$.
		\end{enumerate}
	\end{hypo}

	\begin{figure}[]
		\centering
		\hspace*{\fill}
		\subcaptionbox{ \label{fig:phaseA} }
		{\includegraphics[width=0.3\textwidth]{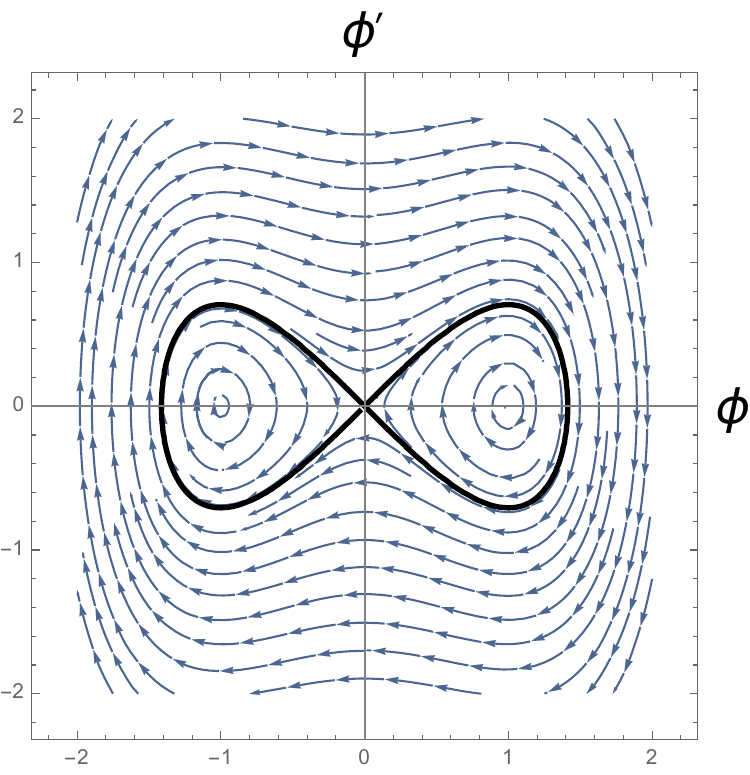}}\hfill 
		\subcaptionbox{ \label{fig:phaseB} }
		{\includegraphics[width=0.3\textwidth]{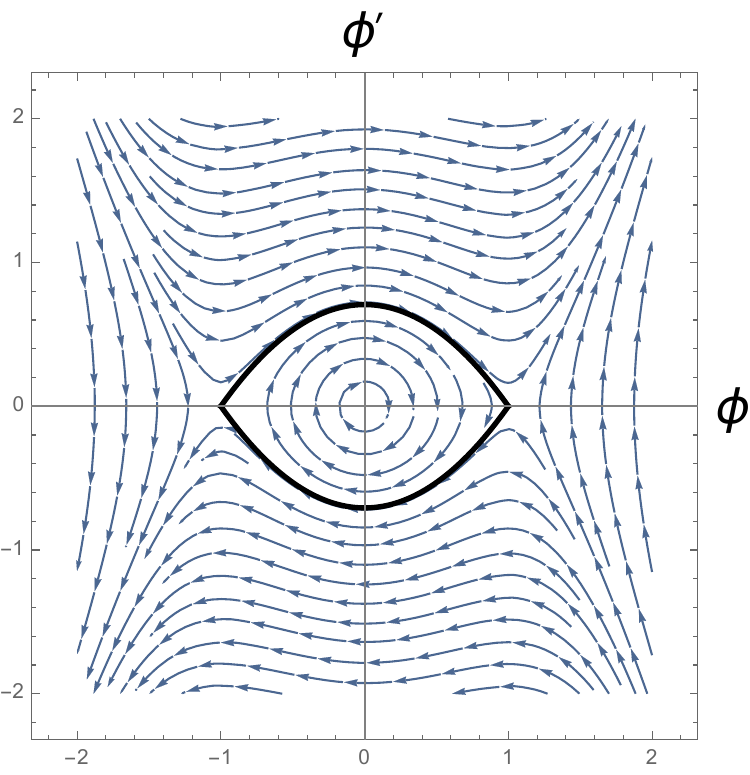}} \hfill
		\subcaptionbox{ \label{fig:phaseC} }
		{\includegraphics[width=0.3\textwidth]{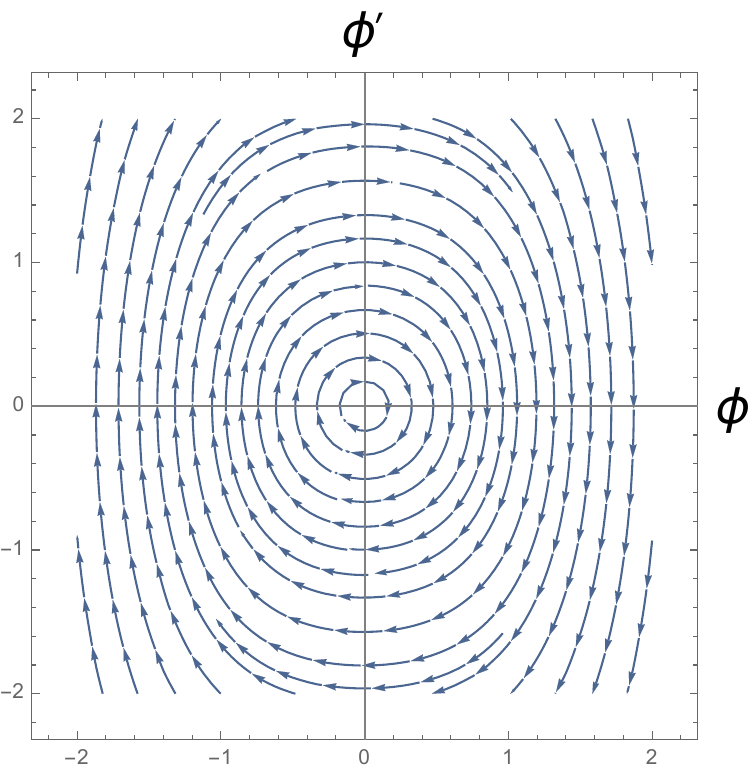}} 
		\hspace*{\fill}
		\caption{Examples of phase portraits for equation \eqref{standwave}. In (a) we have cubic focusing nonlinearity $f(\phi^2) = \phi^2$ and $\be<0$. The homoclinic orbits in black, representing localised solutions on $\R$, separate those inside (nonzero Jacobi dnoidal functions) and those outside (Jacobi cnoidal functions that oscillate evenly about $\phi=0$). In (b) we have cubic defocusing nonlinearity $f(\phi^2) = -\phi^2$ and $\be>0$, with periodic orbits existing only inside the heteroclinic cycle in black. In (c) we have $f(\phi^2) = \phi^2$ and $\be>0$.}
		\label{fig:phase_plane}
	\end{figure}
	
	\begin{rem}\label{rem:nonautonomous}
		With $g$ and $h$ arbitrary functions of $x$ in general, the results of this paper concerning the stability of NLS standing waves are valid for a spatially dependent nonlinearity $f(x,|\psi|^2)$ as appearing in, for example, \cite{J88,Grill88}. In this case, the loss of autonomy in the standing wave equation \eqref{standwave} means that $L_+\phi'\neq0$; thus, only the results which rely on $\phi'$ being an eigenfunction for $L_+$ (\cref{cor:exact_count_Q=0}, \cref{prop:concave_up,cor:concaveup}) do not generalise to the non-autonomous case. 
	\end{rem}
	
	Under the assumptions of \cref{hypo:NLS}, our analogue of the Jones--Grillakis instability theorem will follow from both \cref{thm:N_bound} and a computation of the values of $\mathfrak{c}$ given in  \cref{thm:compute_c}.
	
	\begin{cor}\label{cor:JonesGrillakis}
		Let $N$ be an operator as in \eqref{Hamiltonian_N}--\eqref{LplusLminus}. If $0\in\spec(L_+)\setminus \spec(L_-)$ and $P-Q\neq -1,0$, or $0\in\spec(L_-)\setminus \spec(L_+)$ and $P-Q\neq 0,1$, then $n_+(N)\geq 1$. Under \cref{hypo:NLS}, $\widehat\psi$ is spectrally unstable in these cases. 
	\end{cor}
	\vspace{-1mm}
	(The proof is given in \cref{subsec:51}.) This criterion leads to the following instability result. The waves described correspond, for example, to the periodic orbits represented by the phase curves that are contained inside either of the orbits homoclinic to $(0,0)$ in \cref{fig:phaseA}.
	
	\begin{cor}\label{cor:exact_count_Q=0}
		Assume \cref{hypo:NLS}. Standing waves satisfying the Neumann boundary conditions \eqref{eq:NeumannBC} that are nonconstant and nonvanishing over $[0,\ell]$, and have one or more critical points in $(0,\ell)$, are unstable.
	\end{cor}
	\vspace{-1mm}
	(The proof is given in \cref{subsec:51}.) To effectively use \cref{thm:N_bound}, we need to understand the quantity $\mathfrak c$ appearing in \eqref{eq:bound_positive_evals}. Its definition involves the Maslov index at a potentially degenerate crossing, and hence requires some work to calculate. We do this by analysing the curves in the $\la s$-plane that describe the evolution of the real eigenvalues $\la$ of the restricted problem \eqref{eq:N_EVP_restricted} as $s$ is varied. As will be seen in \cref{thm:compute_c}, $\mathfrak c$ is determined by the concavity of these curves. Below, dot denotes $d/d\la$. The proof of the following theorem is given in \cref{subsec:Lyap_Schmidt}.
	
	\begin{theorem}
		\label{thm:sL}
		Let $N$ be an operator as in \eqref{Hamiltonian_N}--\eqref{LplusLminus}. If $\dim\ker (N) = 1$, then there exists a smooth function $s(\lambda)$, defined for $|\lambda| \ll 1$, such that $s(0) = 1$ and $\lambda$ is an eigenvalue of \eqref{eq:N_EVP_restricted} on $[0,s(\lambda)\ell]$. Moreover, $\dot s(0) = 0$  and the concavity of $s(\lambda)$ can be determined as follows:
		\begin{enumerate}
			\item If $0 \in \spec(L_-) \setminus\spec(L_+)$ with eigenfunction $v\in\ker(L_-)$, then
			\begin{equation}
				\label{s''L-}
				\ddot s(0) = \f{2}{\ell}\f{\langle \widehat{u},v \rangle}{ \left( v'(\ell)  \right)^2 }
			\end{equation}
			where $\widehat{u}\in H^2(0,\ell) \cap H^1_0(0,\ell)$ is the unique solution to $L_+ \widehat{u} = v$.
			\item If $0 \in \spec(L_+) \setminus \spec(L_-)$ with eigenfunction $u\in\ker(L_+)$, then
			\begin{equation}
				\label{s''L+}
				\ddot s(0) = -\f{2}{\ell} \f{\langle \widehat{v}, \,u \rangle}{\left ( u'(\ell)  \right )^2 }
			\end{equation}
			where $\widehat{v} \in H^2(0,\ell) \cap H^1_0(0,\ell)$ is the unique solution to $-L_- \widehat{v}= u$.
		\end{enumerate}
	\end{theorem}
	\begin{rem}\label{rem:positivity}
		In applications, we will primarily be interested in the \textit{sign} of $\ddot s(0)$, for which \eqref{s''L-} and \eqref{s''L+} give
		\begin{equation}\label{integrals}
			\sgn \ddot s(0) = \sgn \int_0^{\ell} \widehat{u} \,v \,dx\qquad \text{and}\qquad  \sgn \ddot s(0) = -\sgn \int_0^{\ell} \widehat{v} \,u \,dx , \,\,\,\,
		\end{equation}
		respectively. The integrals in \eqref{integrals} can be rewritten as
		\begin{equation}
			\,\, \,\,\,\,\,\int_0^{\ell} \widehat{u} \,v \,dx = \int_0^{\ell} \widehat{u} \,\left (L_+ \widehat{u} \right ) \,dx\qquad \text{and} \qquad \int_0^{\ell} \widehat{v} \,u \,dx  = \int_0^{\ell} \widehat{v}\, \left (L_- \widehat{v} \right ) \,dx.
		\end{equation}
		Consequently, $\ddot s(0) > 0$ if $0 \in \spec(L_-)$ and $L_+$ is a strictly positive operator, or if $0 \in \spec(L_+)$ and $L_-$ is strictly positive.
	\end{rem}
	
	In \cref{sec:proof1} we will prove a more general version of \cref{thm:sL}; see \cref{thm:sL:general}. An analogous result for the case when $\dim\ker(N) = 2$ is given in \cref{thm:touching}. Using these results, we give a computation of the Maslov index at the non-regular crossing in \cref{thm:compute_c}. 
	
	As an application of our theory, working under \cref{hypo:NLS}, we provide a new formula for the sign of $\ddot s(0)$ by evaluating the integral expression in \eqref{s''L+} for stationary states satisfying \eqref{eq:NeumannBC}; see \cref{prop:concave_up}. In the edge cases when $P-Q=1$ and $0\in\spec(L_-)\setminus\spec(L_+)$, or $P-Q=-1$ and { $0\in\spec(L_+)\setminus\spec(L_-)$}, we show (see \cref{thm:VK}) that spectral stability of the standing wave $\widehat\psi$ is determined by the sign of $\ddot s(0) $. This suggests that on a bounded interval, the integrals $\langle\cdot,\cdot \rangle$ in \eqref{s''L-} and \eqref{s''L+} play the same role that \eqref{intro:VK} plays in the well known VK criterion on the real line. We thus refer to the two integral expressions in \eqref{integrals} as \textit{VK-type integrals}. In \cref{sec:VKrecover} we show that it is possible to recover the classical VK criterion on a compact interval using the numerator in \eqref{s''L-} (but \textit{not} \eqref{s''L+}). 
	
	\begin{theorem}\label{thm:VK}
		Let $N$ be an operator as in \eqref{Hamiltonian_N}--\eqref{LplusLminus}. Consider the case when $P=1$, $Q=0$, and $0\in\spec(L_-)\setminus \spec(L_+)$. If the associated VK-type integral in \eqref{s''L-} is positive, then { $n_+(N) =  1$}, while if the integral is negative, then $\spec(N)\subset i\R$. In particular, under \cref{hypo:NLS},  $\widehat\psi$ is spectrally unstable if \eqref{s''L-} is positive, and spectrally stable if \eqref{s''L-} is negative.
		
		Similarly, consider the case when $Q=1$, $P=0$, and $0\in\spec(L_+)\setminus \spec(L_-)$. If the VK-type integral in \eqref{s''L+} is { negative}, then { $n_+(N) =  1$}, while if the integral is { positive}, then $\spec(N)\subset i\R$. In particular, under \cref{hypo:NLS}, $\widehat\psi$ is spectrally unstable if \eqref{s''L+}  is positive, and spectrally stable if \eqref{s''L+} is negative. 
	\end{theorem}
	
	  (The proof is given in \cref{subsec:52}.) The proofs that { $n_+(N)=1$} rely on an argument that allows the replacement of the inequality in \eqref{eq:bound_positive_evals} with an equality, as well as a computation of $\mathfrak{c}$  that yields 1 on the right hand side of \eqref{eq:bound_positive_evals}. The former comes from the fact that the Maslov index is monotone in $\la$ \emph{provided either $P$ or $Q$ is zero} (see \cref{lemma:sign_definite}). On the other hand, to prove $\spec(N)\subset i\R$ in the cases described in \cref{thm:VK}, it will be shown (see \cref{lemma:imag_eigvals}) that the nonnegativity of $L_+$ or $L_-$ forces the spectrum of $N$ to be confined to the real and imaginary axes. It will then follow from monotonicity in $\la$ (i.e. \cref{lemma:sign_definite}) that $n_+(N)=0$ (and therefore that $\spec(N)\subset i\R$).

	\begin{rem}
		In \cref{thm:VK} we recover the equality in \cite[Theorem 2.25]{HK08} without the assumption that the operators $L_\pm$ are invertible (albeit in the case when $P=0$ or $Q=0$). Recovering the equality (when $L_+$ and $L_-$ are invertible) in cases when both $P$ and $Q$ are nonzero via our Maslov index calculations remains an open question. 
	\end{rem}

	\section{A symplectic approach to  the eigenvalue problem}\label{sec:symplectic_view}
	
	In this section we review the definition of the Maslov index and give a symplectic formulation of the eigenvalue problem \eqref{eq:N_EVP_full}, culminating in the proof of \cref{thm:N_bound}.
	
	\subsection{The Maslov index}\label{subsec:maslov_index}
	We begin with some background material on the Maslov index \cite{Maslov65}. We follow the definition given by Robbin and Salamon \cite{RS93}, wherein the Maslov index is first defined for regular paths, and then extended to arbitrary continuous paths by a homotopy argument. For more on the topological properties of the spaces discussed, see \cite{A67}. For a systematic and unified treatement of the Maslov index, featuring an axiomatic description and four equivalent definitions, see \cite{CLM94}.

	The starting point is $\R^{2n}$ equipped with the nondegenerate, skew-symmetric bilinear form 
	\begin{equation}\label{omega}
		\w: \R^{2n} \times \R^{2n} \lra \R, \qquad \w(x,y) = Jx \cdot y 
	\end{equation}
	called a \textit{symplectic form}, where ``$\cdot$'' is the dot product in $\R^{2n}$ and $J$ is given in \eqref{def:JSsig}. A \textit{Lagrangian subspace} or \textit{plane} $\Lambda$ of $\R^{2n}$ is an $n$-dimensional subspace on which the symplectic form vanishes. The \textit{Lagrangian Grassmannian} is the set of all Lagrangian subspaces, $\cL(n) = \left \{ \Lambda \subset \R^{2n} :\, \dim(\Lambda) =n,\quad \w(x,y) = 0, \,\forall \,\,x,y\in\Lambda \right \}$. This space has infinite cyclic fundamental group, i.e. $\pi_1(\cL(n))  = \Z$. A notion of winding therefore exists for paths in $\cL(n)$; this is the Maslov index. Namely, the Maslov index of a loop in $\cL(n)$ is its equivalence class in the fundamental group. Poincar\'e duality \cite[\S 3.3]{hatcher} affords an interpretation of this winding number as the (signed) number of intersections with a distinguished codimension one submanifold, and this allows one to extend the definition to \textit{any} path in $\cL(n)$. This is the approach of Arnol'd, which we briefly review.
	
	Fix a reference plane $\Lambda_0\in\cL(n)$. The distinguished codimension one submanifold of $\cL(n)$ is given by the top stratum $\cT_1(\Lambda_0)$ of the \textit{train} of $\Lambda_0$, 
	\[
	\cT(\Lambda_0) = \big \{ \Lambda \in \cL(n) : \Lambda \cap \Lambda_0 \neq \{0\} \big \} = \bigcup_{k=1}^n \cT_k(\Lambda_0),
	\]
	where $\cT_k(\Lambda_0) = \left \{ \Lambda \in \cL(n) : \dim(\Lambda \cap \Lambda_0 )= k\right \}$. As the fundamental lemma of \cite{A67} states, $\cT_1(\Lambda_0)$ is two sidedly imbedded in $\cL(n)$. This means there exists a continuous vector field transverse to $\cT_1(\Lambda_0)$ and tangent to $\cL(n)$. One can therefore assign a signature to each transverse intersection of a path in $\cL(n)$ with $\cT_1(\Lambda_0)$. Any Lagrangian path with endpoints not in $\cT(\Lambda_0)$ can be perturbed to one that only intersects the top stratum $\cT_1(\Lambda_0)$ of the train, and only does so transversally; the Maslov index is then defined to be the sum of the signatures of all such intersections. 
	
	We next recall the approach of Robbin and Salamon \cite{RS93}, which requires additional regularity but applies to paths whose endpoints are in the train, and also allows for intersections with $\cT_k(\Lambda_0)$ when $k>1$. This approach, while less geometric than the above interpretation of the Maslov index as an intersection number, is more suited to practical computations.
	
	Given a smooth path $\Lambda: [a,b] \lra \cL(n)$, a \textit{crossing} is a point $t=t_0$ where $\Lambda(t_0) \in \cT(\Lambda_0)$. Let $\Lambda_0^\perp\subset \R^{2n}$ be a subspace transverse to $\Lambda_0$. Then $\Lambda(t)$ is transverse to $\Lambda_0$ for all $t$ in a small neighbourhood of $t_0$, so there exists a smooth family of matrices $R_t: \Lambda_0 \ra \Lambda_0^\perp$ such that $R_{t_0}=0_{2n}$ and
	\begin{equation}\label{eq:graph_of_R}
		\Lambda(t) = \graph(R_t) = \{ q + R_t q : q\in\Lambda(t_0) \}
	\end{equation}
	for $|t-t_0| \ll 1$. At a crossing $t_0$, the \textit{crossing form} is the quadratic form
	\begin{equation}\label{maslovcrossform}
		\mathfrak{m}_{t_0}(q) = \de{}{t} \w(q,q+R_tq)\Big|_{t=t_0} = \w(q,\dot R_{t_0}q),\qquad q\in \Lambda(t_0) \cap\Lambda_0,
	\end{equation}
	on the intersection $\Lambda(t_0) \cap \Lambda_0$. The full symmetric bilinear form associated with the quadratic form \eqref{maslovcrossform} may be recovered using the polarisation identity; see, for example, the proof of \cref{cor:higherdims}. A crossing is called \textit{regular} if the form \eqref{maslovcrossform} is nondegenerate, and \textit{simple} if $\Lambda(t_0) \in \cT_1(\Lambda_0)$. Since $\mathfrak{m}_{t_0}$ is quadratic, it may be diagonalised; we let $n_+(\mathfrak{m}_{t_0})$ and $n_-(\mathfrak{m}_{t_0})$ be the number of positive and negative squares obtained in so doing. The signature of $\mathfrak{m}_{t_0}$ is the integer $\sgn(\mathfrak{m}_{t_0}) = n_+(\mathfrak{m}_{t_0})-n_-(\mathfrak{m}_{t_0})$. We then define the Maslov index as follows. 
	\begin{define}\label{defn:maslov}
		The Maslov index for a path $\Lambda:[a,b] \lra \cL(n)$ having only regular crossings is given by
		\begin{equation}\label{eq:maslov_defn}
			\Mas(\Lambda(t),\Lambda_0; [a,b]) \coloneqq -n_-(\mathfrak{m}_{a}) +\sum_{a<t_0<b} \sgn(\mathfrak{m}_{t_0}) + n_+(\mathfrak{m}_b),
		\end{equation}
		where the sum is taken over all crossings $t_0\in(a,b)$. 
	\end{define}
	\vspace{-2mm}
	One can show that regular crossings are isolated and therefore the sum is well-defined. Note the convention at the endpoints: at $t=a$ only the negative squares contribute to the Maslov index, while at $t=b$ only the positive squares contribute. Other conventions are possible, see e.g. \cite[\S 2]{RS93}, but we choose the above in order to ensure the Maslov index is an integer. 
	
	The Maslov index of an arbitrary continuous path $\Lambda_1 :[a,b] \lra \cL(n)$ is then defined to be $\Mas(\Lambda_2(t),\Lambda_0; [a,b])$, where $\Lambda_2$ is any path that is homotopic (with fixed endpoints) to $\Lambda_1$ and has only regular crossings. It is guaranteed by \cite[Lemmas 2.1 and 2.2]{RS93} that such a path exists, and any two such paths have the same index, so the Maslov index of $\Lambda_1$ is well defined.

	The essential properties of the Maslov index that we will use are given in the following proposition, see \cite[Theorem 2.3]{RS93}.
	\begin{prop}\label{prop:enjoys}
		The Maslov index enjoys 
		\begin{enumerate}
			\item Homotopy invariance: if two paths $\Lambda_1, \Lambda_2 : [a,b] \lra \cL(n)$ are homotopic with fixed endpoints, then
			\begin{equation}\label{homotopy_property}
				\Mas(\Lambda_1(t),\Lambda_0; [a,b]) = \Mas(\Lambda_2(t),\Lambda_0; [a,b]).
			\end{equation}
			\item Additivity under concatenation: for $\Lambda(t):[a,c] \lra \cL(n)$ and $a<b<c$, 
			\begin{equation}\label{}
				\Mas(\Lambda(t),\Lambda_0; [a,c]) = \Mas(\Lambda(t),\Lambda_0; [a,b]) + \Mas(\Lambda(t),\Lambda_0; [b,c]).
			\end{equation}
		\end{enumerate}
	\end{prop}
	
	To conclude our discussion of the Maslov index, we expound the notion of a non-regular crossing, that is, a crossing with degenerate crossing form. Consider a Lagrangian path $\Lambda: [a,b] \lra \cL(n)$ with a non-regular crossing $t=t_0$. In the case that $\mathfrak{m}_{t_0}$ is identically zero, in \cite[Proposition 3.10]{DJ11} the authors state that the contribution to the Maslov index is determined by the second-order crossing form
	\begin{equation}\label{secondorderform}
		\mathfrak{m}^{(2)}_{t_0}(q) \coloneqq \des{}{t} \w(q,q+R_tq)\Big|_{t=t_0} = \w(q,\ddot R_{t_0}q),\qquad q\in \Lambda(t_0) \cap\Lambda_0,
	\end{equation}
	provided it is nondegenerate. Such a crossing can only contribute to the Maslov index if it occurs at one of the endpoints: if $t_0=a$ then it contributes $-n_-(\mathfrak{m}^{(2)}_a)$, and if $t_0=b$ then it contributes $n_+(\mathfrak{m}^{(2)}_b)$.
	
	As an example, consider the case of a simple crossing with $\mathfrak{m}_{t_0}=0$ but $\mathfrak{m}^{(2)}_{t_0} \neq 0$. In the Lagrangian Grassmannian, this corresponds to our path $\Lambda$ tangentially intersecting the train $\cT(\Lambda_0)$ of the fixed reference plane to quadratic order; i.e. $\Lambda$ ``bounces off" the train as $t$ passes through $t_0$. Provided $t_0$ lies in the interior of $[a,b]$, the contribution to the Maslov index will be zero: clearly the path can locally be homotoped to one with no crossings at all. If $t_0=a$, the contribution is $-1$ provided the path leaves in the negative direction (and zero otherwise), while if $t_0=b$, the contribution is $+1$ provided the path arrives in the positive direction (and zero otherwise). If the second order form is degenerate, i.e. $\mathfrak{m}^{(2)}_{t_0} = 0$, higher order derivatives are needed in order to determine the local behaviour of the path $\Lambda$.

	In the present setting, with the spectral parameter $\la$ acting as the independent variable, we will observe that a non-regular crossing occurs at $\la=0$. To determine the contribution to the Maslov index of this non-regular crossing, we use a homotopy argument, made possible by our analysis of the local behaviour of the eigenvalue curves in \cref{subsec:degenerate}. We confirm that our computation agrees with the number of negative squares of the second order form \eqref{secondorderform} used in \cite{DJ11}. For a further discussion of non-regular crossings, see \cite{GPP04,GPP_full}.

	\subsection{Spatial rescaling and construction of the Lagrangian path}\label{subsec:lagrangian}

	We now view the problem through the lens of the Lagrangian formalism by interpreting eigenvalues as nontrivial intersections of Lagrangian planes. Following the approach of \cite{DJ11}, we restrict the eigenvalue problem to a family of subintervals $[0,s\ell]$ for $s\in(0,1]$. Rescaling the equations to the full domain $[0,\ell]$, we construct a two-parameter family of Lagrangian subspaces in $s$ and $\la$ via rescaled boundary traces of solutions to the system of differential equations without any boundary conditions at all. An eigenvalue is produced when this family of subspaces nontrivially intersects a fixed reference plane that encodes Dirichlet boundary conditions. Identifying a Lagrangian structure boils down to a judicious choice of both the symplectic form and the definition of the trace map: if we employ the standard symplectic form $\w$ in \eqref{omega}, then we need to carefully define the trace map \eqref{Trace} such that the space of boundary traces is Lagrangian with respect to $\w$. We begin by introducing some notation. 
	
	We let 
	\begin{equation}\label{}
		N = D+B(x), \qquad 	D\coloneqq \begin{pmatrix}
			0  & \partial_{xx} \\ - \partial_{xx} & 0
		\end{pmatrix}, \qquad B(x)\coloneqq\begin{pmatrix}
			0 & h(x)\\ -g(x) & 0
		\end{pmatrix},
	\end{equation}
	and introduce the $s$-dependent operators acting on functions on $[0,\ell]$,
	\begin{align} \label{rescaled_operators}
		B_s(x) \coloneqq s^2 B(sx), \qquad N_s \coloneqq \begin{pmatrix} 0 & -L_-^s  \\ L_+^s & 0 \end{pmatrix}, \qquad \begin{cases}
			L_+^s \coloneqq -\p_{xx} - s^2 g(sx) \\
			L_-^s \coloneqq -\p_{xx} - s^2 h(sx) 
		\end{cases}
	\end{align}
	so that $N_s = D+B_s(x)$. We define the \textit{rescaled trace} of $\mathbf{u}=(u,v)^{\top} \in H^2(0,\ell)$  as the vector
	\begin{gather} \label{Trace}
		\begin{aligned}
			\Tr_s{\mathbf{u}} &\coloneqq  \left( u(0), v(0), u(\ell), v(\ell), -\f{1}{s}u'(0), \f{1}{s}v'(0), \f{1}{s}u'(\ell), -\f{1}{s}v'(\ell)) \right)^\top \in\R^8,
		\end{aligned}
	\end{gather}
	and denote the vertical subspace of $\R^8$ by $\cD\coloneqq \{0\}\times \R^4$. Using the above notation, we may rewrite the restricted problem \eqref{eq:N_EVP_restricted} as a boundary value problem on $[0,\ell]$. Indeed, if $\mathbf{u}(x) \in H^2(0,s\ell)\cap H^1_0(0,s\ell)$ then $\mathbf{u}_s(x)\coloneqq\mathbf{u}(sx) \in H^2(0,\ell) \cap H^1_0(0,\ell)$. It follows from \eqref{Trace} that $\bu(0) = \bu(s\ell) = 0$ if and only if $\Tr_s{\mathbf{u}_s} \in \cD$. Thus, rescaled to $[0,\ell]$, \eqref{eq:N_EVP_restricted} reads
	\begin{align}\label{eq:N_EVP_rescaled}
		N_s\mathbf{u}_s = s^2\la \mathbf{u}_s, \quad \Tr_s{\mathbf{u}_s} \in \cD.
	\end{align}
	Note that the solution spaces of the boundary value problems \eqref{eq:N_EVP_restricted} and \eqref{eq:N_EVP_rescaled} are isomorphic: $\mathbf{u}=(u,v)^\top \in \dom(N|_{[0,s\ell]})$ solves \eqref{eq:N_EVP_restricted} if and only if $\mathbf{u}_s=(u_s,v_s)^\top\in \dom(N_s)$ solves \eqref{eq:N_EVP_rescaled}. Consequently, $\la$ is an eigenvalue of $N|_{[0,s\ell]}$ if and only if $s^2\la$ is an eigenvalue of $N_s$.

	\begin{rem}
		The rescaled problem \eqref{eq:N_EVP_rescaled} is well-defined for $s>1$ provided the potentials $g$ and $h$ are defined for $x>\ell$. In this case the ``restricted" eigenvalue problem \eqref{eq:N_EVP_restricted} corresponds to a \textit{stretching} of the domain. 
	\end{rem}

	\begin{rem}
		As per \cref{rem:operators}, notationally we will not distinguish between $N_s$ and $L_\pm^s$ as differential expressions and as unbounded operators with dense domains given by \eqref{domN} and \eqref{domLpm}, respectively. Thus, when we write $s^2\la\in\spec(N_s)$ or $\mathbf{u}_s \in \ker(N_s-s^2\la)$, we mean that \eqref{eq:N_EVP_rescaled} is solved for some eigenfunction $\mathbf{u}_s$; similar statements hold  when $\la\in\spec(L^s_\pm) $.
	\end{rem}
	
	That the formulation \eqref{eq:N_EVP_rescaled} lends itself to a symplectic interpretation can be seen via the following modified version of Green's second identity. Using our definition of the rescaled trace map \eqref{Trace} and the symplectic form \eqref{omega}, one can verify that for each $s \in (0,1]$ and all $\mathbf{u},\mathbf{v} \in H^2(0,\ell)$, 
	\begin{equation}\label{greens_rescaled_mod}
		\langle S(N_s-s^2\la) \mathbf{u} ,\mathbf{v} \rangle - \langle \mathbf{u}, S (N_s-s^2\la)\mathbf{v} \rangle = s\w(\Tr_s \mathbf{u}, \Tr_s \mathbf{v} ),
	\end{equation}
	where $S$ is defined in \eqref{def:JSsig}. Now define the space
	\begin{equation}\label{}
		\mathcal{K}_{\la,s}\coloneqq  \left \{ \mathbf{u}\in H^2(0,\ell) : ( N_s - s^2\la )\mathbf{u}=0 \,\,\,\text{in}\,\,\,L^2(0,\ell)\right \}
	\end{equation}
	of all solutions to the homogeneous differential equation $N_s \mathbf{u}= s^2\la\mathbf{u}$ \textit{without} any reference to the boundary conditions, so that $\ker(N_s-s^2\la) = \mathcal{K}_{\la,s} \cap H^1_0(0,\ell)$.
	
	\begin{rem}\label{rem:injective}
		The trace map is an injective linear operator on the space $\mathcal{K}_{\la,s}$. If $\mathbf{u}_s \in\cK_{\la_0,s}$, then $\Tr_s\mathbf{u}_s  = 0$ implies $\mathbf{u}_s=0$, since $\mathbf{u}_s$ solves a system of second order equations. 
	\end{rem}
	
	Taking the (rescaled) boundary trace leads to the desired family of Lagrangian subspaces, with respect to the form $\w$ in \eqref{omega}.
	\begin{lemma}
		The space
		\begin{equation}\label{defn:Lambda}
			\Lambda(\la,s) \coloneqq \Tr_s(\cK_{\la,s}) = \{ \Tr_s(\mathbf{u}) : \mathbf{u}\in \cK_{\la,s} \}
		\end{equation} 
		is a Lagrangian subspace of $\R^8$  for all $s\in (0,1]$ and all $\la\in\R$. 
	\end{lemma}
	
	\begin{proof}
		Fix $\la\in\R$ and $s\in(0,1]$. From \eqref{greens_rescaled_mod}, for $\mathbf{u}, \mathbf{v}\in \mathcal{K}_{\la,s} $ we have $ \omega(\Tr_s\mathbf{u}, \Tr_s \mathbf{v})=0 $. Since $\cK_{\la,s}$ is the space of solutions to a system of two second-order differential equations, $\dim \cK_{\la,s}=4$. Hence $\dim \Tr_s(\cK_{\la,s}) = 4$, and $\Tr_s(\cK_{\la,s}) \in \cL(4) $ is Lagrangian. 
	\end{proof}
	We now have the desired interpretation of eigenvalues as nontrivial intersections of Lagrangian subspaces. 
	
	\begin{prop}\label{prop:Lag_interp}
		$s^2\la \in \spec (N_{s})$ if and only if $\Lambda(\la,s)\cap \cD \neq\{0\}$. Moreover, the geometric multiplicity of the eigenvalue is equal to the dimension of the Lagrangian intersection, 
		\begin{equation}\label{eq:dimensions}
			\dim \ker (N_{s} -s^2\la ) = \dim\Lambda(\la,s)\cap \cD.
		\end{equation}
	\end{prop}
	\begin{proof}
		The first statement follows from the definition of $\Lambda$. Equality \eqref{eq:dimensions} follows from the injectivity (and thus bijectivity) of the trace map acting between the finite dimensional spaces  $\ker (N_{s_0} -s_0^2\la_0 ) = \cK_{\la_0,s_0} \cap H^1_0(0,\ell)$ and $\Tr_{s_0}(\cK_{\la_0,s_0}\cap H^1_0(0,\ell)) =  \Lambda(\la_0,s_0)\cap \cD$. 
	\end{proof}
	
	Hereafter, a \emph{crossing} refers to a pair $(\la,s)=(\la_0,s_0)$ such that $\Lambda(\la_0,s_0)\cap \cD \neq\{0\}$, while a \emph{conjugate point} refers to a crossing for which $\la_0=0$. It follows from \cref{prop:Lag_interp} that crossings where $s_0=1$ correspond to eigenvalues of the operator $N$ on $[0,\ell]$.
	
	To prove \cref{thm:N_bound}, our goal then is to bound from below the number of crossings for which $s_0=1,\la_0>0$. To do so we use a homotopy argument that involves appropriately counting conjugate points.  In order to set this argument up, we introduce in \cref{fig:MaslovBox} the so-called \textit{Maslov box}, given by the boundary $\Gamma$ of the rectangle $ [ 0,\la_{\infty} ]\times [\tau,1]$ in the $\la s$-plane, where $\tau>0$ is small and $\la_\infty>0$ is large. 
	
	Since $\Lambda : [ 0,\la_{\infty} ]\times [\tau,1]\lra \cL(4) $ is a continuous map, the image $\Lambda(\Gamma)$ of the Maslov box is null homotopic, and so
	\begin{align}\label{mas0}
		\Mas(\Lambda,\cD;\Gamma) =0.
	\end{align}
	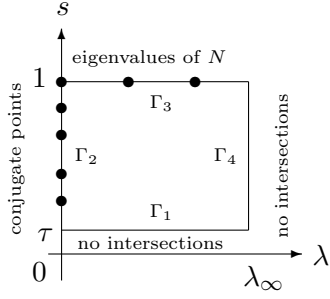
\begin{figure}
		\centering
		\begin{picture}(100,115)(-20,0)
			\put(-1,83){$1$}
			\put(-1,10){$0$}
			\put(10,10){\vector(0,1){95}}
			\put(5,20){\vector(1,0){95}}
			\put(10,85){\line(1,0){70}}
			\put(10,29){\line(1,0){70}}
			\put(80,29){\line(0,1){56}}
			\put(67,55){\text{\tiny $\Gamma_4$}}
			\put(15,55){\text{\tiny$\Gamma_2$}}
			\put(43,75){\text{\tiny$\Gamma_3$}}
			\put(43,34){\text{\tiny$\Gamma_1$}}
			\put(104,16){$\lambda$}
			\put(8,110){$s$}
			\put(1,25){$\tau$}
			\put(78,8){$\lambda_\infty$}
			\put(10,85){\circle*{4}}
			\put(10,65){\circle*{4}}
			\put(10,40){\circle*{4}}
			\put(10,50){\circle*{4}}
			\put(10,75){\circle*{4}}
			\put(35,85){\circle*{4}}
			\put(60,85){\circle*{4}}
			\put(14,92){{\tiny \text{eigenvalues of $N$}}}
			\put(12,22){{\tiny \,\,\,\text{no intersections}}}  
			\put(-10,27){\rotatebox{90}{{\tiny conjugate points}}}
			\put(90,26){\rotatebox{90}{{\tiny no intersections}}}
		\end{picture}
		\caption{Maslov box in the $\la s$-plane. }
		\label{fig:MaslovBox}
	\end{figure}
	We partition $\Gamma$ into its constituent sides such that $\Gamma = \Gamma_1 \cup \Gamma_2 \cup \Gamma_3 \cup \Gamma_4$, where 
	\begin{equation}\label{}
		\begin{aligned}
			& \Gamma_1:  s = \tau, \quad 0\leq \la \leq \la_\infty \qquad  &\Gamma_3:  s = 1, \quad 0\leq \la \leq \la_\infty  \\
			& \Gamma_2: \la = 0,\quad  \tau\leq s \leq 1 \qquad  &\Gamma_4: \la = \la_\infty, \quad \tau\leq s \leq 1
		\end{aligned}
	\end{equation}
	(see \cref{fig:MaslovBox}) and assign a direction to each of these intervals such that the entirety of the Maslov box is oriented in a clockwise fashion. We then appeal to the concatenation property in \cref{prop:enjoys} to rewrite \eqref{mas0} as
	\begin{align}\label{mas_concat}
		\Mas(\Lambda,\cD;\Gamma_1) + \Mas(\Lambda,\cD;\Gamma_2) + \Mas(\Lambda,\cD;\Gamma_3) + \Mas(\Lambda,\cD;\Gamma_4) =0.
	\end{align} 
	Taking $\la=\la_\infty$ large enough and $s=\tau$ small enough, it will follow (see \cref{lem:maslov_right_bottom_zero}) that there are no crossings along $\Gamma_1$ and $\Gamma_4$, and therefore that the Maslov indices of these pieces are zero. The crossing forms needed to analyse $\Mas(\Lambda,\cD;\Gamma_2)$ and $\Mas(\Lambda,\cD;\Gamma_3)$ are given in the next section.

	\subsection{Crossing forms}\label{sec:crossingforms}
	
	Our next task is the calculation of the crossing forms \eqref{maslovcrossform} associated with the trajectories through the crossing $(\la_0,s_0)$ where $\la=\la_0$ is held constant and $s$ increases, and vice versa. The key ingredient will be the Green's-type identity \eqref{greens_rescaled_mod}. The approach is inspired by Lemma 4.18 and the proof of Theorem 4.19 in \cite{LS20first}, as well as the crossing form calculation in \cite[Lemma 5.2]{CJLS14}. Before proceeding, we set some notation that will be useful in this section and throughout the rest of the paper.
	
	\begin{rem}\label{rem:notation}
		We denote by $\bu_{s_0}$ any eigenfunction $\bu_{s_0}\in\ker(N_{s_0} - s_0^2\la_0)$, and when $s_0=1$ we drop the subscript. If $\dim\ker(N_{s_0} - s_0^2\la_0)=n$, we denote a basis for this space by $\big \{\bu_{s_0}^{(1)}, \dots, \bu_{s_0}^{(n)} \big\}$, where $ \mathbf{u}_{s_0}^{(i)} = \big (u_{s_0}^{(i)}, v_{s_0}^{(i)}\big )^\top$.  The set $\big \{S\bu_{s_0}^{(1)}, \dots, S\bu_{s_0}^{(n)} \big\}$ is then a basis for the kernel of the adjoint operator, $\ker(N_{s_0}^*-s_0^2\la_0)$, since $\la_0$ is real. Note that $S$ (given in \eqref{def:JSsig}) merely swaps the entries of the vector it acts on. When $s_0=1$ we denote:
		\begin{equation}\label{notations0=1}
			\mathbf{u}_i \coloneqq \mathbf{u}_{1}^{(i)}, \qquad u_i\coloneqq u_{1}^{(i)},\,\, v_i\coloneqq v_{1}^{(i)}.
		\end{equation}
		Because $\ker(N_{s_0}) = \ker(L_+^{s_0}) \oplus \ker(L_-^{s_0})$, when $\la_0=0$ and $\dim\ker(N_{s_0})=1$ we have
		\begin{equation}\label{eq:simple_eigfns}
			\mathbf{u}_{s_0} = \begin{cases}
				(u_{s_0},0)^\top,  &0\in\spec(L_+^{s_0})\setminus \spec(L_-^{s_0}),\quad  \ker(L_+^{s_0})=\spn\{u_{s_0}\},\\
				(0, v_{s_0})^\top,  &0\in\spec(L_-^{s_0})\setminus \spec(L_+^{s_0}), \quad \ker(L_-^{s_0})=\spn\{v_{s_0}\}.
			\end{cases}
		\end{equation}
		When $\la_0=0$ and $\dim\ker(N_{s_0})=2$, we denote
		\begin{equation}
			\label{eq:doubleeigfns}
			\bu^{(1)}_{s_0} = \begin{pmatrix} u^{(1)}_{s_0} \\ 0 \end{pmatrix}, \qquad  \bu^{(2)}_{s_0} = \begin{pmatrix} 0 \\ v^{(2)}_{s_0} \end{pmatrix},
		\end{equation}
		where $\ker(L_+^{s_0})=\spn\{u_{s_0}^{(1)}\}$ and $\ker(L_-^{s_0})=\spn\{v_{s_0}^{(2)}\}$. 
		
		In the current paper where the potentials $g$ and $h$ from \eqref{LplusLminus} are scalar-valued, we will always have $n\leq 2$. However, if $g$ and $h$ are matrix-valued (and symmetric), so that $L_\pm$ are systems of selfadjoint Schr\"odinger operators, or if the operator $N$ acts on functions on a multidimensional domain, then we may have $n>2$. The results in this section and \cref{sec:proof1} have been stated for a general $n$ to indicate how the theory extends to these cases. 
	\end{rem}
	
	Returning to our computation of crossing forms, we first compute the crossing form \eqref{maslovcrossform} for the path of Lagrangian planes $s\mapsto \Lambda(\la_0,s)$, holding $\la=\la_0$ fixed. Recall that $N_s = D+B_s$, as in \eqref{rescaled_operators}, and that $S=S^T$.
	\begin{lemma}
		\label{lemma:scrossingform}
		Let $(\la_0,s_0)$ be a crossing and fix any nonzero $q\in\Lambda(\la_0,s_0) \cap\cD$. Then there exists a unique
		$\mathbf{u}_{s_0}\in\cK_{\la_0,s_0}$ such that $q=\Tr_{s_0}{\mathbf{u}_{s_0}}$, and the crossing form for the Lagrangian path $s\mapsto \Lambda(\la_0,s)$ at $s=s_0$ is given by 
		\begin{equation} 
			\label{eq:maslov_vertical}
			\mathfrak{m}_{s_0}(q)
			=  \f{1}{s_0} \left  \langle  \big(\p_sB_{s_0} - 2 s_0\lambda_0\big) \bu_{s_0}, S\bu_{s_0} \right \rangle,
		\end{equation}
		where $\p_sB_s=2sB(sx) + s^2 B'(sx)x$. In particular, along $\Gamma_2$ where $\la_0=0$, we have
		\begin{equation}\label{maslov_form_s_lambda=0}
			\mathfrak{m}_{s_0}(q) =  \f{\ell}{s_0^2} \left[ - \left ( u_{s_0}'(\ell)\right )^2 +  \left ( v_{s_0}'(\ell)\right )^2\right].
		\end{equation}
		In this case, if the crossing $(0,s_0)$ is simple, then the form \eqref{maslov_form_s_lambda=0} is non-degenerate.
	\end{lemma}
	
	\begin{proof}
		Consider a $C^1$ family of vectors $s\mapsto \mathbf{w}_s\in \cK_{\la_0,s}$ satisfying
		\begin{subequations}
			\begin{align}
				N_s \mathbf{w}_s&= s^2\la_0 \mathbf{w}_s,  \qquad  x\in[0,\ell], \ \  s\in (s_0-\e,s_0+\e), \label{eq:K_eigv_eq}	\\
				\Tr_s\mathbf{w}_s &= \Tr_{s_0}\mathbf{u} _{s_0} + R_s\Tr_{s_0}\mathbf{u}_{s_0},  \qquad\mathbf{w}_{s_0} = \mathbf{u}_{s_0},
			\end{align}
		\end{subequations}
		where $R_s:\Lambda(\la_0,s_0) \ra \Lambda(\la_0,s_0)^\perp$ is the smooth family of matrices such that $\Lambda(\la_0,s) = \graph(R_s)$, cf. \eqref{eq:graph_of_R}. To prove the existence of such a family $s\mapsto \mathbf{w}_s$, consider the smooth family of vectors $h_s \coloneqq q + R_s q \in \Lambda(\la_0,s)$, where $h_{s_0}=q$ since $R_{s_0}=0$. The injectivity (and thus bijectivity) of the linear map
		\begin{equation*}
			\Tr_s: \cK_{\la_0,s} \lra \Tr_s(\cK_{\la_0,s}) =\Lambda(\la_0,s)
		\end{equation*}
		(see \cref{rem:injective}) then implies that for each $h_s \in \Lambda(\la_0,s)$ there exists a unique $\mathbf{w}_s\in\cK_{\la_0,s}$ such that $\Tr_s \mathbf{w}_s = h_s$, and in particular $\Tr_{s_0} \mathbf{w}_{s_0} = h_{s_0} = q$. 
		
		We now turn to the computation of \eqref{maslovcrossform}. We have
		\begin{align*}
			\mathfrak{m}_{s_0}(q)  &=\de{}{s}\, \w(q, R_sq )\bigg|_{s=s_0} \\
			&=\de{}{s}\, \w(\Tr_{s_0}{\mathbf{u}_{s_0}}, \Tr_{s}{\mathbf{w}_s} )\bigg|_{s=s_0} \\
			&= \omega \left (\Tr_{s_0}{\mathbf{u}_{s_0}}, \de{}{s} {\Tr}_{s}\big|_{s=s_0} \mathbf{u}_{s_0}\right  ) + \w\left (\Tr_{s_0}{\mathbf{u}_{s_0}}, \Tr_{s_0} \de{}{s}\mathbf{w}_s \big|_{s=s_0} \right ).
		\end{align*}
		The first term is zero since $\Tr_{s_0}\mathbf{u}_{s_0}  \in\cD$ implies  $\Tr_{s_0} \mathbf{u}_{s_0} =  \left (0, s_0^{-1} \gamma_N\mathbf{u}_{s_0}\right )$ and $\de{}{s}{\Tr}_s\big|_{s=s_0}\mathbf{u}_{s_0} = \left(0, - s_0^{-2} \gamma_N\mathbf{u}_{s_0}\right)$, where $
		\gamma_N \mathbf{u} \coloneqq  \left(-u'(0), v'(0), u'(\ell), -v'(\ell)) \right)^\top $. For the second term, we differentiate the equation in \eqref{eq:K_eigv_eq} with respect to $s$ and apply $\langle \cdot, S\mathbf{w}_s\rangle$, 
		\begin{align}
			\label{eq:BK_dash}
			\langle (\p_sB_s -2s\la_0) \mathbf{w}_s, S\mathbf{w}_s\rangle + \langle (N_{s}- s^2\la_0) \,\p_s \mathbf{w}_s,  S\mathbf{w}_s \rangle = 0.
		\end{align}
		From the Green's-type identity \eqref{greens_rescaled_mod} with $\mathbf{u} = \mathbf{w}_s$ and $\mathbf{v} = \p_s \mathbf{w}_s$, we have
		\begin{equation*} 
			s\,\w(\Tr_s \mathbf{w}_s, \Tr_s \p_s \mathbf{w}_s  ) = \langle (N_{s}- s^2\la_0) \mathbf{w}_s,S \,\p_s\mathbf{w}_s  \rangle - \langle S\mathbf{w}_s, (N_{s}- s^2\la_0) \p_s \mathbf{w}_s \rangle,
		\end{equation*}
		and using \eqref{eq:K_eigv_eq} and \eqref{eq:BK_dash} this reduces to 
		\begin{equation}
			\label{keyingredient}
			s\,\w(\Tr_s \mathbf{w}_s, \Tr_s \p_s \mathbf{w}_s ) =  \langle (\p_sB_s-2s\la_0) \mathbf{w}_s, S\mathbf{w}_s\rangle.
		\end{equation}
		Evaluating \eqref{keyingredient} at $s=s_0$ and dividing by $s_0$,  \eqref{eq:maslov_vertical} follows. When $\la_0=0$, substituting the stated expression for $\p_sB_{s_0}$ in \eqref{eq:maslov_vertical} gives
		\begin{align*}
			\label{eqmbs}
			\mathfrak{m}_{s_0}(q) &= \left \langle \left (2 B(s_0x) + s_0 B'(s_0x)x \right )\mathbf{u}_{s_0},S\mathbf{u}_{s_0}\right \rangle \\
			&= \int_{0}^{\ell} \Big \{ \left [2h(s_0x)+ s_0 x h'(s_0x)\right ] v_{s_0}^2(x) - \left [2g(s_0x)+s_0x g'(s_0x)\right ] u_{s_0}^2(x)\Big\} dx.
		\end{align*}
		A direct calculation using the equation $L_-^{s_0} v_{s_0}=0$, i.e. $ v_{s_0}''(x)+s_0^2h(s_0x)v_{s_0}(x)=0$, gives 
		\[
		\de{}{x} \left [ \f{1}{s_0^2} x\left ( v_{s_0}'(x)\right )^2 + xv^2_{s_0}(x)h(s_0x) -\f{1}{s_0^2}v_{s_0}(x) v_{s_0}'(x)\right ] = \left [2h(s_0x) +s_0x h'(s_0x) \right ]v_{s_0}^2(x).
		\]
		Integrating and using the fact that $v_{s_0}(0) = v_{s_0}(\ell)=0$, we get 
		\[
		\int_{0}^{\ell} \left [2h(s_0x) +s_0x h'(s_0x) \right ]v_{s_0}^2(x) dx = \f{\ell}{s_0^2}\left ( v_{s_0}'(\ell) \right )^2.
		\]
		Computing similarly for the second term, we arrive at \eqref{maslov_form_s_lambda=0}. That the form is nondegenerate in the simple case follows from \eqref{eq:simple_eigfns}: if $\dim\ker(N_{s_0})=1$ then exactly one of the entries of $\mathbf{u}_s = (u_s,v_s)^\top \in \ker(N_{s_0})$ is nontrivial. Since this function satisfies a second order differential equation with Dirichlet boundary conditions, its derivative is nonzero at $x=\ell$, and therefore \eqref{maslov_form_s_lambda=0} is nonzero. 
	\end{proof}

	\begin{cor}\label{cor:higherdims}
		Assume $\dim\ker (N_{s_0}-s_0^2\la_0) = n$ and let $\{\mathbf{u}_{s_0}^{(1)}, \mathbf{u}_{s_0}^{(2)}, \dots, \mathbf{u}_{s_0}^{(n)} \}$ be a basis for $\ker (N_{s_0} - s_0^2\la_0)$. The $n\times n$ symmetric matrix $\mathfrak{M}_{s_0}$ induced from the quadratic form \eqref{eq:maslov_vertical} is given by 
		\begin{equation}\label{M^s}
			\left [\mathfrak{M}_{s_0}\right ]_{ij} = \f{1}{s_0} \left  \langle  \big(\p_sB_{s_0} - 2 s_0\lambda_0\big) \mathbf{u}_{s_0}^{(i)}, S\mathbf{u}_{s_0}^{(j)} \right \rangle, \qquad i,j=1,\dots, n.
		\end{equation}
		Consequently, when $\la_0=0$ and $n=2$, the form $\mathfrak{m}_{s_0}$ is nondegenerate. 
	\end{cor}
	\begin{proof}		
		Letting $q_i \coloneqq \Tr_{s_0} \mathbf{u}^{(i)}_{s_0}$, it follows from the linearity and injectivity of the trace map that $\{q_i\}_{i=1}^n$ is a basis for $\Lambda(\la_0,s_0) \cap \cD$. To construct the symmetric bilinear form associated with the quadratic form \eqref{eq:maslov_vertical}, we compute the off-diagonal terms $\mathfrak{m}_{s_0}(q_i,q_j)$ via the real polarisation identity
		\begin{equation}\label{}
			\mathfrak{m}_{s_0}(q_i,q_j) = \f{1}{4} \big[\mathfrak{m}_{s_0}(q_i + q_j) - \mathfrak{m}_{s_0}(q_i-q_j)\big ].
		\end{equation}
		Since both $S$ and $S\,(\p_sB_{s_0})$ are symmetric, we obtain
		\begin{align*}
			\mathfrak{m}_{s_0}(q_i,q_j)
			&= \f{1}{4}\left  \langle  \big(\p_sB_{s_0} - 2 s_0\lambda_0\big) \big(\mathbf{u}_{s_0}^{(i)}+\mathbf{u}_{s_0}^{(j)}\big), S\big(\mathbf{u}_{s_0}^{(i)}+\mathbf{u}_{s_0}^{(j)} \big)\right \rangle \\
			&\quad -  \f{1}{4}\left  \langle  \big(\p_sB_{s_0} - 2 s_0\lambda_0\big) \big(\mathbf{u}_{s_0}^{(i)}-\mathbf{u}_{s_0}^{(j)}\big), S\big(\mathbf{u}_{s_0}^{(i)}-\mathbf{u}_{s_0}^{(j)} \big)\right \rangle \\
			&= \left  \langle  \big(\p_sB_{s_0} - 2 s_0\lambda_0\big) \mathbf{u}_{s_0}^{(i)}, S\mathbf{u}_{s_0}^{(j)} \right \rangle.
		\end{align*}
		The corresponding matrix elements with respect to the basis $\{q_i\}$ are $\left [\mathfrak{M}_{s_0}\right ]_{ij} = \mathfrak{m}_{s_0}(q_i,q_j)$,
		and the first statement of the corollary follows. In the case $\la_0=0$ and $n=2$, using \eqref{maslov_form_s_lambda=0} and recalling \eqref{eq:doubleeigfns}, the matrix \eqref{M^s} reduces to 
		\begin{equation}
			\label{eq:capitalmathfrakMn=2}
			\mathfrak{M}_{s_0}
			= \f{\ell}{s_0^2} \begin{pmatrix}
				- \big ( \p_xu_{s_0}^{(1)}(\ell)  \big )^2 &  0\\ 
				0	 & \big ( \p_xv_{s_0}^{(2)}(\ell)\big )^2 
			\end{pmatrix},
		\end{equation}
		which clearly has full rank. Nondegeneracy of the quadratic form $\mathfrak{m}_{s_0}$  follows.  
	\end{proof}
	
	We now move to the $\la$-direction. Holding $s=s_0$ fixed, we compute the crossing form \eqref{maslovcrossform} with respect to $\la$. We denote $d/d\la$ with a dot. 
	
	\begin{lemma}
		\label{lemma:lambdacrossform}
		Let $(\la_0,s_0)$ be a crossing and fix any nonzero $q\in\Lambda(\la_0,s_0) \cap\cD$. Then there exists a unique $\mathbf{u}_{s_0}\in\cK_{\la_0,s_0}$ such that $q=\Tr_{s_0}{\mathbf{u}_{s_0}}$, and the crossing form for the Lagrangian path $\la\mapsto \Lambda(\la,s_0)$ at $\la=\la_0$ is given by 
		\begin{equation}\label{maslov_horizontal}
			\mathfrak{m}_{\la_0}(q) 
			= -s_0 \left \langle \mathbf{u}_{s_0}, S\mathbf{u}_{s_0} \right  \rangle = -2s_0 \left \langle u_{s_0}, v_{s_0} \right  \rangle.
		\end{equation}
	\end{lemma}
	
	\begin{proof}
		The argument is virtually identical to that in the $s$ direction. Fixing $s=s_0$, we consider a $C^1$ family of vectors $\la\mapsto \mathbf{w}_\la\in\cK_{\la,s_0}$ satisfying
		\begin{subequations}\label{wlambda_family}
			\begin{align}
				N_{s_0}\mathbf{w}_\la & = s_0^2\la \mathbf{w}_\la, \qquad x\in[0,\ell], \ \ \la\in (\la_0-\e, \la_0+\e) \label{eq:N_eigv_eq_s0}	
				\\
				\Tr_{s_0}\mathbf{w}_\la&= \Tr_{s_0}\mathbf{u} _{s_0} + R_\la \Tr_{s_0}\mathbf{u}_{s_0}, \qquad \mathbf{w}_{\la_0} = \mathbf{u}_{s_0},
			\end{align}
		\end{subequations}
		where now $R_\la : \Lambda(\la_0, s_0)\lra  \Lambda(\la_0,s_0)^\perp$ is such that $\Lambda(\la,s_0) = \graph(R_\la)$.  
		Similar to \eqref{eq:BK_dash} we have 
		\begin{equation*}\label{}
			\langle -s_0^2 \mathbf{w}_{\la}, S\mathbf{w}_{\la}\rangle  + \langle (N_{s_0}-s_0^2\la ) \dot{\mathbf{w}}_{\la}, S\mathbf{w}_{\la} \rangle =0,
		\end{equation*}
		and using the identity \eqref{greens_rescaled_mod} with $\mathbf{u} = \mathbf{w}_\la$ and $\mathbf{v} = \dot{\mathbf{w}}_\la$ yields
		\begin{equation*}\label{}
			s_0\,\w(\Tr_{s_0}  \mathbf{w}_{\la}, \Tr_{s_0}  \dot{\mathbf{w}}_{\la}  ) = \langle (N_{s_0}-s_0^2\la) \mathbf{w}_{\la}, S\dot{\mathbf{w}}_\la \rangle - \langle S\mathbf{w}_\la, (N_{s_0}-s_0^2\la) \dot{\mathbf{w}}_\la \rangle.
		\end{equation*} 
		The previous two equations along with \eqref{eq:N_eigv_eq_s0} give
		\begin{equation}\label{eq:Blambdadot}
			s_0\,\w(\Tr_{s_0}  \mathbf{w}_{\la}, \Tr_{s_0}  \dot{\mathbf{w}}_{\la}  ) =  -\langle s_0^2 \mathbf{w}_{\la}, S\mathbf{w}_{\la}\rangle.
		\end{equation}
		Therefore the crossing form \eqref{maslovcrossform} is
		\begin{align*}
			\mathfrak{m}_{\la_0}(q) 
			= \w\left (\Tr_{s_0}{\mathbf{u}_{s_0}}, \Tr_{s_0}{\dot{\mathbf{w}}_{\la}}\big|_{\la=\la_0} \right ) 
			= -s_0\langle \mathbf{u}_{s_0}, S\mathbf{u}_{s_0}\rangle=-2s_0 \left \langle u_{s_0}, v_{s_0} \right  \rangle,
		\end{align*}
		where we used \eqref{eq:Blambdadot} evaluated at $\la=\la_0$.
	\end{proof}
	
	Recalling \eqref{eq:simple_eigfns}, at a simple crossing $(0,s_0)$ one of $u_{s_0}$ or $v_{s_0}$ is always trivial. Degeneracy of the $\la$-crossing form immediately follows.
	\begin{cor}\label{cor:degeneracy_of_form_lambda}
		All conjugate points $(0,s_0)$ for which $\dim\ker(N_{s_0})=1$ are non-regular in the $\la$ direction, i.e. at all such points $\mathfrak{m}_{\la_0}=0$.
	\end{cor}
	For the case of higher dimensional crossings, we have the following corollary to \cref{lemma:lambdacrossform}.
	
	\begin{cor}\label{cor:higherdims_lambda}
		Assume $\dim\ker (N_{s_0}-s_0^2\la_0) = n$ and let $\{\mathbf{u}_{s_0}^{(1)}, \mathbf{u}_{s_0}^{(2)}, \dots, \mathbf{u}_{s_0}^{(n)} \}$ be a basis for $\ker (N_{s_0} - s_0^2\la_0)$. The $n\times n$ symmetric matrix $\mathfrak{M}_{\la_0}$ induced from the $n$-dimensional quadratic form \eqref{maslov_horizontal} is given by 
		\begin{equation}\label{M^la}
			\left [\mathfrak{M}_{\la_0}\right ]_{ij} = -s_0 \left \langle \mathbf{u}_{s_0}^{(i)}, S\mathbf{u}_{s_0}^{(j)} \right  \rangle, \qquad i,j=1,\dots n.
		\end{equation}
		Consequently, when $\la_0=0$ and $n=2$,  $\mathfrak{m}_{\la_0}$ is nondegenerate if and only if $\big<u^{(1)}_{s_0}, v^{(2)}_{s_0} \big> \neq 0$.
	\end{cor}
	\begin{proof}
		The first statement is proved as in \cref{cor:higherdims}. When $\la_0=0$ and $n=2$, due to \eqref{eq:doubleeigfns}, \eqref{M^la} reduces to
		\begin{equation}\label{matrix_firstorder_lambda}
			\mathfrak{M}_{\la_0} = -s_0	\begin{pmatrix}
				\big \langle \mathbf{u}_{s_0}^{(1)}, S\mathbf{u}_{s_0}^{(1)} \big \rangle & 
				\big \langle  \mathbf{u}_{s_0}^{(1)}, S\mathbf{u}_{s_0}^{(2)} \big \rangle \\[2mm]
				\big \langle  \mathbf{u}_{s_0}^{(2)}, S\mathbf{u}_{s_0}^{(1)} \big \rangle &
				\big \langle  \mathbf{u}_{s_0}^{(2)}, S\mathbf{u}_{s_0}^{(2)} \big \rangle
			\end{pmatrix}
			= -s_0\begin{pmatrix}
				0 &  \big<u^{(1)}_{s_0}, v^{(2)}_{s_0} \big> \\[2mm]
				\big<u^{(1)}_{s_0}, v^{(2)}_{s_0} \big> & 0
			\end{pmatrix},
		\end{equation}
		from which nondegeneracy of  $\mathfrak{m}_{\la_0}$ occurs if and only if  the condition stated holds.
	\end{proof}
	It follows from \cref{cor:higherdims_lambda,cor:degeneracy_of_form_lambda} that a calculation of the Maslov index at $\la=0$ in the $\la$-direction is not possible using the first order crossing form \eqref{maslovcrossform} if $\dim\ker(N_{s_0})=1$, or if $\dim\ker(N_{s_0})=2$ and $\big<u^{(1)}_{s_0}, v^{(2)}_{s_0} \big> =0$. In light of this, we define:
	\begin{define}
		\label{defn:c}
		The correction term $\mathfrak{c}$ is 
		\begin{equation}\label{eq:defn_c}
			\mathfrak{c} := \Mas\big(\Lambda(s,\la),\cD;s\in[1-\e,1]\big) + \Mas\big(\Lambda(\la,1),\cD; \lambda\in[0,\e]\big)
		\end{equation}
		for $0<\e\ll 1$.
	\end{define}
	
	That is, $\mathfrak c$ denotes the contribution to the Maslov index from the top left corner of the Maslov box (consisting of the arrival along $\Gamma_2$ and the departure along $\Gamma_3$).
	
	\begin{rem}
		
		\label{rem:Lisolated}
		To see that this does not depend on the choice of $0 < \e \ll 1$, we observe that $(0,1)$ is an isolated crossing for both $\Gamma_2$ and $\Gamma_3$. For $\Gamma_2$ this follows from the non-degeneracy of $\mathfrak m_{s_0}$ in \cref{lemma:scrossingform,cor:higherdims}. For $\Gamma_3$ we use the fact that the set
		$
		\{ \lambda : \Lambda(\la,1) \cap \cD \neq \{0\} \} = \spec(N) \cap \R
		$
		is discrete (because $N$ has compact resolvent), so there exists $\hat\la>0$ such that $\Lambda(\la,1) \cap \cD = \{0\}$ for $0 < \la < \hat\la$.
		
	\end{rem}
	
	We circumvent the issue of the non-regular crossing in \cref{subsec:degenerate} via a homotopy argument. This will be possible after having analysed the local behaviour of the eigenvalue curves in \cref{sec:proof1}. In the meantime, we compute the second order crossing form \eqref{secondorderform} from \cite[Proposition 3.10]{DJ11}.
	
	\begin{lemma}\label{lemma:second_order_lambda_form}
		Assume the conditions of \cref{lemma:lambdacrossform}. If the first order quadratic form in \eqref{maslov_horizontal} is identically zero, then the second order quadratic form \eqref{secondorderform} is given by
		\begin{equation}\label{second_order_lambda_form}
			\mathfrak{m}^{(2)}_{\la_0}(q) = -2s_0^3\langle {\mathbf{v}}_{s_0}, S\mathbf{u}_{s_0} \rangle,  \qquad q=\Tr_{s_0}\mathbf{u}_{s_0},
		\end{equation}
		where $\mathbf{u}_{s_0} \in \ker(N_{s_0}-s_0^2\la_0)$ and ${\mathbf{v}}_{s_0}\in\dom(N_{s_0})$ solves $(N_{s_0}-s_0^2\la_0){\mathbf{v}}_{s_0}= \mathbf{u}_{s_0}$. The $n\times n$ matrix $\mathfrak{M}^{(2)}_{\la_0}$ of the symmetric bilinear form associated with $\mathfrak{m}^{(2)}_{\la_0}$ has entries
		\begin{equation}\label{bigmfM2}
			\left [\mathfrak{M}^{(2)}_{\la_0}\right ]_{ij} = -2s_0^3 \left \langle {\mathbf{v}}_{s_0}^{(i)}, S\mathbf{u}_{s_0}^{(j)} \right \rangle,
		\end{equation}
		where ${\mathbf{v}}_{s_0}^{(i)}\in\dom(N_{s_0})$ solves $(N_{s_0}-s_0^2\la_0){\mathbf{v}}_{s_0}^{(i)} = \mathbf{u}_{s_0}^{(i)}$. In the case $\la_0=0$ and $n=1$, we have
		\begin{align}\label{2ndorder_lambda_1dim}
			\mathfrak{m}^{(2)}_{\la_0}(q) = \begin{cases}
				-2s_0^3 \left \langle \widehat{v}_{s_0}, u_{s_0}\right \rangle & 0\in\spec(L_+^{s_0})\setminus  \spec(L_-^{s_0}), \vspace{1mm}\\
				-2s_0^3 \left  \langle \widehat{u}_{s_0}, v_{s_0} \right \rangle & 0\in\spec(L_-^{s_0})\setminus  \spec(L_+^{s_0}),
			\end{cases}
		\end{align}
		where $\widehat{v}_{s_0}\in\dom(L_-^{s_0})$ and $\widehat{u}_{s_0} \in\dom(L_+^{s_0})$ solve $-L_-^{s_0}\widehat{v}_{s_0} = u_{s_0}$ and $L_+^{s_0}\widehat{u}_{s_0} = v_{s_0}$ respectively. In the case $\la_0=0$ and $n=2$ we have
		\begin{equation}\label{bigmfM2_2dim}
			\mathfrak{M}^{(2)}_{\la_0} = -2s_0^3\begin{pmatrix}
				\big \langle \widehat{v}_{s_0}^{(1)}, u_{s_0}^{(1)}\big \rangle & 0 \\
				0 & \big \langle \widehat{u}_{s_0}^{(2)}, v_{s_0}^{(2)} \big \rangle
			\end{pmatrix},
		\end{equation}
		where $\widehat{v}_{s_0}^{(1)}\in\dom(L_-^{s_0})$ and $\widehat{u}_{s_0}^{(2)} \in\dom(L_+^{s_0})$ solve $-L_-^{s_0}\widehat{v}_{s_0}^{(1)} = u_{s_0}^{(1)}$ and $L_+^{s_0}\widehat{u}_{s_0}^{(2)} = v_{s_0}^{(2)}$ respectively.
	\end{lemma}
	\begin{rem}\label{rem:contribute}
		The equation $(N_{s_0}-s_0^2\la_0){\mathbf{v}}^{(i)}_{s_0}= \mathbf{u}^{(i)}_{s_0}$ is always solvable by virtue of the Fredholm Alternative, since $\mathfrak m_{s_0} = 0$ means $\langle\mathbf{u}^{(i)}_{s_0}, S \mathbf{u}^{(j)}_{s_0}\rangle=0$ for all $i,j$ and hence implies $\mathbf{u}^{(i)}_{s_0}$ is orthogonal to $\ker(N^*_{s_0} - s_0^2\la_0)$. Such a solution is not unique; however, only the component of the solution in $\ker(N_{s_0}-s_0^2\la_0)^{\perp}$ (which is unique) contributes to \eqref{second_order_lambda_form}. It therefore suffices to consider those $\bv_{s_0}^{(i)}$ satisfying $\bv_{s_0}^{(i)}\perp \bu_{s_0}^{(j)}$ for all $j=1, \dots, n$. Notice that the $\mathbf{v}_{s_0}^{(i)}$ are generalised eigenfunctions: if $\mathfrak{m}_{\la_0} = 0$, the eigenvalue $s_0^2\la_0\in\spec(N_{s_0})$ has $n$ Jordan chains of length (at least) two. We thus see that loss of regularity of the crossing coincides precisely with loss of semisimplicity of the eigenvalue, which agrees with the result of \cite[Theorem 6.1]{Corn19}.
	\end{rem}
	
	\begin{proof}
		Consider a $C^2$ family of vectors $\la\mapsto \mathbf{w}_\la$ satisfying \eqref{wlambda_family}. Then
		\begin{align*}
			\mathfrak{m}^{(2)}_{\la_0}(q) 
			&=  \w\left ( \Tr_{s_0}\mathbf{u}_{s_0}, \Tr_{s_0}\ddot{\mathbf{w}}_\la \right )\big|_{\la=\la_0}.
		\end{align*} 
		Differentiating \eqref{eq:N_eigv_eq_s0} twice with respect to $\la$, applying  $\langle \cdot , S\mathbf{w}_{\la}\rangle$ and rearranging yields 
		\begin{equation*}
			\left \langle (N_{s_0}-s_0^2\la )\ddot{\mathbf{w}}_{\la},S\mathbf{w}_{\la}\right \rangle = 2s_0^2\langle \dot{\mathbf{w}}_{\la},S\mathbf{w}_{\la}\rangle.
		\end{equation*}	
		Now using \eqref{greens_rescaled_mod} with $\mathbf{u}=\mathbf{w}_{\la}$ and $\mathbf{v} = \ddot{\mathbf{w}}_{\la}$, we have
		\begin{equation*}
			s_0\,\w(\Tr_{s_0}  \mathbf{w}_{\la}, \Tr_{s_0}  \ddot{\mathbf{w}}_{\la}  ) = \langle (N_{s_0}-s_0^2\la) \mathbf{w}_{\la}, S\ddot{\mathbf{w}}_\la \rangle - \langle S\mathbf{w}_\la, (N_{s_0}-s_0^2\la) \ddot{\mathbf{w}}_\la \rangle.
		\end{equation*}
		Combining \eqref{eq:N_eigv_eq_s0} with the previous two equations, we get
		\begin{equation*}\label{}
			s_0\,\w(\Tr_{s_0}  \mathbf{w}_{\la}, \Tr_{s_0}  \ddot{\mathbf{w}}_{\la}  ) =  -2s_0^2\langle  \dot{\mathbf{w}}_{\la}, S\mathbf{w}_{\la}\rangle.
		\end{equation*} 
		Evaluating this last equation at $\la=\la_0$ and dividing through by $s_0$, we see that 
		\begin{align*}
			\mathfrak{m}^{(2)}_{\la_0}(q) = \w(\Tr_{s_0}  \mathbf{u}_{s_0}, \Tr_{s_0}  \ddot{\mathbf{w}}_{\la}  )\big|_{\la=\la_0} =  -2s_0\langle  \dot{\mathbf{w}}_{\la_0}, S\mathbf{u}_{s_0}\rangle.
		\end{align*} 
		To compute $\dot{\mathbf{w}}_{\la_0}$, we see that differentiating \eqref{eq:N_eigv_eq_s0} with resepct to $\la$, evaluating at $\la=\la_0$ and rearranging yields
		\begin{align}\label{inhomog1}
			\left ( N_{s_0} - s_0^2\la_0 \right  )\dot{\mathbf{w}}_{\la_0} = s_0^2 \mathbf{u}_{s_0}.
		\end{align}
		Setting $s_0^2\,{\mathbf{v}}_{s_0} =\dot{\mathbf{w}}_{\la_0}$, equation \eqref{second_order_lambda_form} follows. 
		
		The same arguments as in the proof of \cref{cor:higherdims} are used to prove \eqref{bigmfM2}. Equations \eqref{2ndorder_lambda_1dim} and \eqref{bigmfM2_2dim} follow from the structure of the eigenvectors and generalised eigenvectors when $\la_0=0$. If $0\in\spec(L_-^{s_0})\setminus \spec(L_+^{s_0})$ and $\widehat u_{s_0}$ is as stated in the lemma, we have
		\[
		\begin{pmatrix}
			0 & -L_-^{s_0} \\ L_+^{s_0} & 0
		\end{pmatrix} \begin{pmatrix} \widehat{u}_{s_0} \\ 0 \end{pmatrix} = \begin{pmatrix} 0 \\ v_{s_0} \end{pmatrix} = \bu_{s_0}, 
		\]
		so $\bv_{s_0} = (\widehat u_{s_0} ,0)^\top$ and hence $\big<\bv_{s_0}, S\bu_{s_0} \big> = \langle\widehat{u}_{s_0}, v_{s_0} \rangle$. If $0\in\spec(L_+^{s_0})\setminus \spec(L_-^{s_0})$, we similarly find that $\bv_{s_0} = ( 0, \widehat v_{s_0} )^\top$ and hence $\big<\bv_{s_0}, S\bu_{s_0} \big> = \langle\widehat{v}_{s_0}, u_{s_0} \rangle$.
		Finally, if $\dim\ker(N_{s_0})=2$, we have
		\begin{equation}
			\label{double_generalised_eigfns}
			\bv^{(1)}_{s_0} = \begin{pmatrix} 0 \\ \widehat v^{(1)}_{s_0} \end{pmatrix}, \quad   \bv^{(2)}_{s_0} = \begin{pmatrix} \widehat u^{(2)}_{s_0} \\ 0 \end{pmatrix},
		\end{equation}
		with $\bu_{s_0}^{(i)}$ given by \eqref{eq:doubleeigfns}. It follows that $\big<\bv_{s_0}^{(1)}, S\bu^{(2)}_{s_0} \big> = \big<\bv_{s_0}^{(2)}, S\bu^{(1)}_{s_0} \big> = 0$ and
		\begin{equation}
			\big<\bv_{s_0}^{(1)}, S\bu^{(1)}_{s_0} \big> = \langle \widehat{v}^{(1)}_{s_0}, u^{(1)}_{s_0}\rangle, \qquad
			\big<\bv_{s_0}^{(2)}, S\bu^{(2)}_{s_0}\big> = \langle \widehat{u}^{(2)}_{s_0}, v^{(2)}_{s_0}\rangle,
		\end{equation}
		which completes the proof.
	\end{proof}

	\begin{rem}\label{rem:not_monotone}
		
		The Maslov index is in general not monotone in $\la$, in the sense that the form \eqref{maslov_horizontal} is indefinite. Consequently, it does not necessarily give an exact count of the crossings along $\Gamma_3$ for $\la>0$, which by \cref{prop:Lag_interp} equals the number of real positive eigenvalues of $N$. Nonetheless, the Maslov index always provides a lower bound for this count, and this will be used in the proof of \cref{thm:N_bound}. In special cases it \emph{is} possible to have monotonicity in $\la$; this will be used to obtain stability results in \cref{thm:VK}, cf. \cref{lemma:sign_definite}.
		
	\end{rem}
	
	\subsection{Bounding the real eigenvalue count}\label{subsec:proof_thm22}
	
	Before proving \cref{thm:N_bound}, we list some preliminary results. The first is a version of the Morse Index theorem (see \cite[\S 15]{Milnor},\cite{Smale65}) for scalar-valued Schr\"odinger operators on bounded domains with Dirichlet boundary conditions. Recall that the Morse indices $P$ and $Q$ are the numbers of negative eigenvalues of the operators $L_+$ and $L_-$, respectively.

	\begin{lemma}\label{lemma:morse_schrodinger}
		The Morse index of $L_+$ equals the number of conjugate points for $L_+$ in $(0,1)$,
		\begin{align}\label{sturm_statements}
			\begin{gathered}
				P  = \# \{ s_0\in(0,1) : 0\in\spec(L_+^{s_0} ) \} , 
			\end{gathered}
		\end{align}
		and likewise for $L_-$ and $Q$.
	\end{lemma}
	The following lemma will not be needed until the proof of \cref{lemma:imag_eigvals}, but we list it here since its proof uses the same ideas that are used to prove the previous lemma.
	\begin{lemma}
		\label{lemma:morse_schrodinger_nonnegative}
		If $Q=0$ (respectively, $P=0$) then $L_-^{s}$ (respectively, $L_+^s$) is a strictly positive operator for all $s\in(0,1)$, and is nonnegative for $s=1$.
	\end{lemma}
	\begin{proof}
		This follows from monotonicity of the eigenvalues of the Schr\"odinger operators $L^s_\pm$ in the spatial parameter $s$, see \cite{Smale65}. Indeed, the eigenvalues $\lambda_j^\pm(s)\in\spec(L_\pm^s)$ are strictly decreasing functions of $s$, so $\lambda_j^\pm(1) \geq 0$ implies $\lambda_j^\pm(s) > 0$ for $s\in(0,1)$. 		
	\end{proof}
	The following selfadjoint formulation of the eigenvalue problem will be needed in \cref{lem:maslov_right_bottom_zero}. Some of the ideas used here, especially the use of the square root of a strictly positive operator to convert the eigenvalue problem to a selfadjoint one, can be found in \cite[\S 4]{pelinovsky}. 
	\begin{lemma}\label{lemma:equiv_problems}
		Fix $s\in(0,1]$ and suppose $\la\in\C \backslash \{0\}$. If $L_-^s$ is a nonnegative operator, the eigenvalue problem
		\begin{align}
			&\begin{cases}
				\text{There exists $v_s\in\dom(L_-^s)$, $u_s\in\dom(L_+^s)$ such that:} \\
				\quad  -L_-^sv_s = s^2\la u_s,   \quad L_+^su_s=s^2\la v_s
			\end{cases} \label{eq:VKsystem} && &&
		\end{align}
		is equivalent to
		\begin{align}
			&\begin{cases}
				\text{There exists $ w_s\in\dom\left ( L_-^s|_{X_c}\right )^{1/2}$ with $\Pi\left ( L_-^s|_{X_c}\right )^{1/2} w_s\in\dom(L_+^s)$} \\
				\text{and $L_+^s \Pi \left ( L_-^s|_{X_c}\right )^{1/2}  w_s \in \dom(L_-^s)$, such that:} \\
				\quad \left ( L_-^s|_{X_c}\right )^{1/2} \Pi L_+^s \Pi\left ( L_-^s|_{X_c}\right )^{1/2}  w_s= - s^4\la^2  w_s, 
			\end{cases}  \label{eq:VKsystem3}
		\end{align}
		where the domains $\dom(L_\pm^s)$ are given by \eqref{domLpm}, $X_c\coloneqq \ker(L_-^s)^{\perp}\subseteq L^2(0,\ell)$ and $\Pi$ is the orthogonal projection $\Pi: L^2(0,\ell) \ra X_c$. If $L_+^s$ is nonnegative, then \eqref{eq:VKsystem} is equivalent to
		\begin{align}
			\label{eq:VKsystem5}
			&\begin{cases}
				\text{There exists $ w_s\in\dom\left ( L_+^s|_{X_c}\right )^{1/2}$ with $\Pi\left ( L_+^s|_{X_c}\right )^{1/2} w_s\in\dom(L_-^s)$} \\
				\text{and $L_-^s \Pi\left ( L_+^s|_{X_c}\right )^{1/2}  w_s \in \dom(L_+^s)$, such that:} \\
				\quad \left ( L_+^s|_{X_c}\right )^{1/2} \Pi L_-^s \Pi\left ( L_+^s|_{X_c}\right )^{1/2}  w_s= - s^4\la^2  w_s,
			\end{cases}
		\end{align}
		where now $X_c\coloneqq \ker(L_+^s)^{\perp}\subseteq L^2(0,\ell)$.
	\end{lemma}
	\begin{proof}
		We begin with the case $L_-^s\geq0$. We prove the equivalence of \eqref{eq:VKsystem} and \eqref{eq:VKsystem3} via their equivalence with:
		\begin{align}
			&\begin{cases}
				\text{There exists $u_s\in\dom(L_+^s)\cap X_c$ with $L_+^su_s\in\dom(L_-^s)$, such that:} \\
				\quad L_-^s L_+^su_s=-s^4\la^2 u_s.
			\end{cases} \label{eq:VKsystem2}
		\end{align}
		Defining the restricted operator $L_-^s|_{X_c}$ acting in $X_c$ by
		\begin{align*}
			L_-^s|_{X_c}v\coloneqq L_-^s v, \qquad v\in\dom(L_-^s|_{X_c})  \coloneqq\dom(L_-^s) \cap X_c,
		\end{align*}
		note that $L_-^s|_{X_c}>0$ and $\left ( L_-^s|_{X_c}\right )^{1/2}$ is a well-defined and invertible operator acting in $X_c$.
		
		$\eqref{eq:VKsystem} \im \eqref{eq:VKsystem2}$: Clearly $ L_+^s u_s = s^2\la v_s \in \dom(L_-^s)$, and $u_s = -\f{1}{s^2\la} L_-^sv_s \in \ran L_-^s=X_c$ because $L_-^s$ is selfadjoint and Fredholm. Applying $L_-^s$ to the second equation in \eqref{eq:VKsystem} yields the equation in \eqref{eq:VKsystem2}. 
		
		$\eqref{eq:VKsystem2} \im \eqref{eq:VKsystem3}$: Set $ w_s \coloneqq \left ( L_-^s|_{X_c}\right )^{-1/2} u_s$. Then $ w_s \in \dom\left ( L_-^s|_{X_c}\right )^{1/2}$, and since $u_s \in X_c$ we have $\Pi\left ( L_-^s|_{X_c}\right )^{1/2} w_s = \Pi u_s = u_s \in \dom(L_+^s)$, and $L_+^s \Pi u_s = L_+^s u_s\in \dom(L_-^s)$. Now $ L_+^s u_s = \Pi L_+^s u_s + (I-\Pi) L_+^s u_s$, where the projection $(I-\Pi): L^2(0,\ell) \ra \ker (L_-^s) \subset \dom(L_-^s)$. Then $\Pi L_+^s u_s \in\dom(L_-^s)\cap X_c = \dom(L_-^s|_{X_c})$. Thus $L_-^s L_+^s u_s = L_-^s \Pi L_+^s \Pi u_s = L_-^s|_{X_c} \Pi L_+^s \Pi u_s = \left ( L_-^s|_{X_c}\right )^{1/2}\left ( L_-^s|_{X_c}\right )^{1/2}\Pi L_+^s \Pi u_s$. Substituting this into the equation in \eqref{eq:VKsystem2} and multiplying by $\left ( L_-^s|_{X_c}\right )^{-1/2}$ gives the equation in \eqref{eq:VKsystem3}.
		
		$\eqref{eq:VKsystem3} \im \eqref{eq:VKsystem}$: Set $u_s\coloneqq \Pi \left ( L_-^s|_{X_c}\right )^{1/2}  w_s \in \dom(L_+^s)$ and $v_s\coloneqq  \f{1}{s^2\la} L_+^s\Pi\left ( L_-^s|_{X_c}\right )^{1/2} w_s\in\dom(L_-^s)$. Then $L_+^s u_s = L_+^s \Pi \left ( L_-^s|_{X_c}\right )^{1/2}  w_s  = s^2\la v_s$, and since $\Pi$ projects onto $\ran(L_-^s)$,  $-L_-^s v_s = - \Pi L_-^s v_s=\f{-1}{s^2\la}\Pi L_-^s L_+^s \Pi \left ( L_-^s|_{X_c}\right )^{1/2} w_s $ $= \f{-1}{s^2\la}\Pi L_-^s(\Pi + (I - \Pi) L_+^s \Pi \left ( L_-^s|_{X_c}\right )^{1/2} w_s $ $= \f{-1}{s^2\la}\Pi L_-^s\Pi L_+^s \Pi \left ( L_-^s|_{X_c}\right )^{1/2} w_s =$ $ s^2\la \Pi \left (L_-^s|_{X_c}\right )^{1/2} w_s = s^2\la u_s$.
		
		The case $L_+^s\geq 0$ uses similar arguments, except now \eqref{eq:VKsystem} and \eqref{eq:VKsystem5} are equivalent via:
		\[
		\begin{cases}
			\text{There exists $v_s\in\dom(L_-^s)\cap X_c$ with $L_-^sv_s\in\dom(L_+^s)$, such that:} \\
			L_+^s L_-^sv_s=-s^4\la^2 v_s.
		\end{cases}
		\] 
		We omit the details.
	\end{proof} 
	
	We are now ready to compute the Maslov index of $\Gamma_2^\e$, the restriction of $\Gamma_2$ to $[\tau,1-\e]$.
	
	\begin{lemma}\label{lemma:Maslov_left}
		The Maslov index of the Lagrangian path $s\mapsto\Lambda(0,s) \subset\R^8$, $s\in[\tau,1-\e]$ is 
		\begin{equation}\label{Mas_left_shelf}
			\Mas(\Lambda,\cD; \Gamma_2^\e) = Q-P.
		\end{equation}
	\end{lemma}
	
	\begin{proof}
		Consider the crossing form
		\[
		\mathfrak{m}_{s_0}(q) =  \f{\ell}{s_0^2} \left[ - \left ( u_{s_0}'(\ell)\right )^2 +  \left (v_{s_0}'(\ell)\right )^2\right]
		\]
		from \eqref{maslov_form_s_lambda=0} and recall \eqref{eq:simple_eigfns}. If $(0,s_0)$ is a simple crossing, we obtain $\mathfrak{m}_{s_0} < 0$ if $0\in\spec(L_+^{s_0})$ and $\mathfrak{m}_{s_0} > 0$ if $0\in\spec(L_-^{s_0})$. On the other hand, if $0\in\spec(L_+^{s_0}) \cap \spec(L_-^{s_0})$, the $2\times2$ matrix $\mathfrak{M}_{s_0}$ in \eqref{eq:capitalmathfrakMn=2} has eigenvalues of opposite sign, so we conclude that
		\begin{equation}
			\sgn(\mathfrak{m}_{s_0}) = \begin{cases} -1 & 0\in\spec(L_+^{s_0}) \setminus \spec(L_-^{s_0}), \\
				+1 & 0\in\spec(L_-^{s_0})\setminus \spec(L_+^{s_0}), \\
				0 & 0\in\spec(L_+^{s_0}) \cap \spec(L_-^{s_0}).
			\end{cases}
		\end{equation}
		From the definition \eqref{eq:maslov_defn} we then have
		\begin{align*}
			\begin{split}
				\Mas(\Lambda(0,s),\cD; s\in[\tau,1-\e]) &= -\#\{ s_0\in[\tau,1-\e] : 0\in\spec(L_+^{s_0} )\setminus \spec(L_-^{s_0}) \} \\ & \qquad  + \#\{ s_0 \in[\tau,1-\e]: 0\in\spec(L_-^{s_0} ) \setminus \spec(L_+^{s_0})\} 
			\end{split}  \\
			\begin{split}
				&= - \#\{ s_0\in[\tau,1-\e] : 0\in\spec(L_+^{s_0} ) \} \\ & \qquad  + \#\{ s_0 \in[\tau,1-\e]: 0\in\spec(L_-^{s_0} ) \}, 
			\end{split} 
		\end{align*}
		and the result follows using \cref{lemma:morse_schrodinger}.
	\end{proof}
	
	Next, we prove that there are no crossings along $\Gamma_1$ and $\Gamma_4$; we refer to \cref{fig:MaslovBox}.
	\begin{lemma}\label{lem:maslov_right_bottom_zero}
		$\Mas(\Lambda,\cD;\Gamma_1) = \Mas(\Lambda,\cD;\Gamma_4) =0$ provided $\tau>0$ is sufficiently small and $\la_\infty>0$ is sufficiently large.
	\end{lemma}
	\begin{proof}
		For the case of no crossings along $\Gamma_1$, we prove that $N_s$ has no real eigenvalues for $s=\tau$ small enough. Seeking a contradiction, assume there exists $\tau^2\la \in \spec(N_\tau)\cap \R$ with eigenfunction $\mathbf{u}_\tau=(u_\tau,v_\tau)^\top$. 
		
		First, note that the operators $L_\pm^\tau$ with domains given by \eqref{domLpm} are strictly positive: by the Poincar\'e and Cauchy-Schwarz inequalities,
		\begin{align*}
			\langle L_+^\tau v,v \rangle
			= \| v' \|^2 - \langle \tau^2 g(\tau x)v,v \rangle 
			&\geq C\| v\|^2 - \tau^2\| g\|_{\infty} \|v\|^2
		\end{align*}
		for some $C>0$ and all $v\in\dom(L_+^\tau)$, so we choose $\tau$ small enough that $C> \tau^2\| g\|_{\infty}$. Owing to the decoupling of the eigenvalue equations for $N_\tau$ when $\la=0$, it follows that $0\notin \spec(N_\tau$).
		
		Next, for $\la \in \R\backslash \{0\}$, we note that by \cref{lemma:equiv_problems} the eigenvalue equations for $N_\tau$ are equivalent to 
		\begin{equation}\label{wtau}
			\left ( L_-^\tau\right )^{1/2} L_+^\tau \left ( L_-^\tau\right )^{1/2}  w_\tau = - \tau^4\la^2  w_\tau, 
		\end{equation}
		since the positivity of $L_-^\tau$ implies that $X_c = \ker(L_-^\tau)^{\perp}$ is all of $L^2(0,\ell)$ and hence the resulting projection $\Pi$ is the identity. Applying $\langle\, \cdot ,  w_\tau\rangle$ to \eqref{wtau}, we immediately see that the right hand side is negative, while for the left hand side we obtain 
		\begin{align*}
			\langle\left ( L_-^\tau \right )^{1/2} L_+^\tau \left ( L_-^\tau\right )^{1/2}  w_\tau,  w_\tau \rangle 
			&=  \langle L_+^\tau \left ( L_-^\tau\right )^{1/2}  w_\tau, \left ( L_-^\tau\right )^{1/2}  w_\tau \rangle \\
			& \geq C_+ \langle \left ( L_-^\tau\right )^{1/2}  w_\tau, \left ( L_-^\tau\right )^{1/2}  w_\tau \rangle \\
			& =  C_+ \langle  L_-^\tau w_\tau, w_\tau \rangle \\
			& \geq C_+ C_- \|  w_\tau\|^2 > 0,
		\end{align*}
		for some positive constants $C_\pm$ (using the positivity of $L_\pm^\tau$ and selfadjointness of $\left (L_-^\tau\right )^{1/2}$), a contradiction. We conclude that no such real $\tau^2\la\in\spec(N_\tau)$ exists, and there are no crossings along $\Gamma_1$.  
		
		Moving to $\Gamma_4$, we show that the spectrum of $N_s$ lies in a vertical strip around the imaginary axis in the complex plane for all $s\in(0,1])$. For this, it suffices to show that $\spec(iN_s)$ lies in a horizontal strip around the real axis, since $\spec(N_s) =  -i\spec(iN_s)$ by the spectral mapping theorem.
		Fixing $s\in(0,1]$ we have
		\begin{align}\label{}
			iN_s = i D + i B_s(x)
		\end{align}
		where $iD$ is selfadjoint and  $i B_s(x)$ is bounded. It then follows from \cite[Remark 3.2, p.208]{Kato} and \cite[eq.(3.16), p.272]{Kato} that
		\begin{align}
			\zeta \in \spec(iD+i B_s(x)) \im |\imp{\zeta}|  \leq  \|i B_s(x)\|,
		\end{align}
		as required.  Choosing $\la_\infty > \sup_{s\in(0,1]} \|B_s(x)\|$ ensures there are no crossings along $\Gamma_4$. 
	\end{proof}

	We are our ready to prove our first main result. 
	
	\begin{proof}[Proof of \cref{thm:N_bound}]
		As already observed in \eqref{mas_concat}, the homotopy invariance and additivity of the Maslov index yield
		\begin{equation}\label{eq:box=0}
			\Mas(\Lambda,\cD;\Gamma_1) + \Mas(\Lambda,\cD;\Gamma_2) + \Mas(\Lambda,\cD;\Gamma_3) + \Mas(\Lambda,\cD;\Gamma_4) =0 ,
		\end{equation}
		hence
		\begin{equation}\label{masleft=-mastop}
			\Mas(\Lambda,\cD;\Gamma_2) + \Mas(\Lambda,\cD;\Gamma_3) = 0
		\end{equation}
		by \cref{lem:maslov_right_bottom_zero}. Again using additivity and recalling the definition of $\mathfrak c$ in \cref{defn:c}, we rewrite this as
		\begin{equation}\label{masleft=-mastop2}
			\Mas(\Lambda,\cD;\Gamma_2^\e) +\mathfrak c +  \Mas(\Lambda,\cD;\Gamma^\e_3) = 0,
		\end{equation}
		where $\Gamma_2^\e$ was defined in \cref{lemma:Maslov_left} and $\Gamma_3^\e$ is the restriction of $\Gamma_3$ to $[\e,\la_\infty]$. Using \cref{lemma:Maslov_left} we thus obtain
		\begin{equation}\label{mas_top_fmla}
			\Mas(\Lambda,\cD;\Gamma^\e_3) = P - Q - \mathfrak c.
		\end{equation}
		As discussed in \cref{rem:not_monotone}, the lack of monotonicity in $\la$ means that $\Mas(\Lambda,\cD;\Gamma_3^\e)$ does not necessarily count the number of real, positive eigenvalues of $N$. Nonetheless, we still have that 
		\begin{equation}\label{eq:lowerbound}
			n_+(N) \geq |\Mas(\Lambda,\cD;\Gamma_3^\e)|,
		\end{equation}
		and \eqref{eq:bound_positive_evals} follows. 
	\end{proof}

	\section{The eigenvalue curves}\label{sec:proof1}
	
	In this section we analyse the real eigenvalue curves of $N_s$ in the $\la s$-plane. We consider the general case of a crossing $(\la_0,s_0)$ corresponding to an eigenvalue $s_0^2\la_0 \in \spec(N_{s_0}) $ with $\dim \ker(N_{s_0} - s_0^2\la_0)=n$, paying special attention to the cases $\la_0=0$ and $n=1,2$. We use the results obtained to compute the correction term $\mathfrak{c}$ from \cref{thm:N_bound}, and relate a component of it to the signature of the second order crossing form \eqref{second_order_lambda_form} in \cref{prop:c}.

	\subsection{Numerical description}\label{subsec:flow}
	
	We begin with a brief description of a tool that is useful for numerically computing the eigenvalue curves. The idea is to globally characterise the set of points $(\la,s)$ such that $s^2\la\in\spec(N_s)\cap \R$ as the zero set of a function called the \textit{characteristic determinant}. 
	
	Converting the restricted problem \eqref{eq:N_EVP_restricted} with $y\in[0,s\ell]$ to a first order system yields
	\begin{equation}
		\de{}{y}\left(\begin{array}{c}
			u \\ v \\ r \\ z
		\end{array} \right) =  \left(\begin{array}{cccc}
			0 & 0 & 1 & 0 \\ 0 & 0 & 0 & -1 \\ -g(y) & -\lambda & 0 & 0 \\ -\lambda & h(y) & 0 & 0 
		\end{array} \right)   \left(\begin{array}{c}
			u \\ v \\ r \\ z
		\end{array} \right).
	\end{equation} 
	
	Notice that we use the substitution $\p_yv= -z$ in order to preserve the Hamiltonian structure. Rescaling as in \cref{subsec:lagrangian}, we define $u_s(x) \coloneqq u(sx)$ for $x\in[0,\ell]$, and similarly for $v_s$, $r_s$ and $z_s$. Then, the equivalent system on $[0,\ell]$ is
	\begin{equation}\label{firstordersys}
		\de{}{x}\left(\begin{array}{c}
			u_s \\ v_s \\ r_s \\ z_s
		\end{array} \right) =  \left(\begin{array}{cccc}
			0 & 0 & s & 0 \\ 0 & 0 & 0 & -s \\ -sg(sx) & -s\lambda & 0 & 0 \\ -s\lambda & sh(sx) & 0 & 0 
		\end{array} \right)   \left(\begin{array}{c}
			u_s \\ v_s \\ r_s \\ z_s
		\end{array} \right).
	\end{equation} 
	Consider a fundamental matrix solution $\Phi(x;\la,s)\in\R^{4\times 4}$ to \eqref{firstordersys} with $\Phi(0;\la,s) = I_{4}$. For convenience, we write $\Phi$ as the block matrix
	\[
	\Phi(x;\la,s)= \left(\begin{array}{cc}
		U(x;\la,s) & X(x;\la,s) \\
		V(x;\la,s) & Y(x;\la,s)
	\end{array}\right ), \qquad U,V,X,Y \in \R^{2\times 2},
	\]
	where
	\begin{equation}\label{ics_matrix}
		U(0;\la,s) = Y(0;\la,s) = I_2, \qquad  V(0;\la,s) = X(0;\la,s) = 0_2.
	\end{equation}
	Because $\Phi$ is a matrix solution for \eqref{firstordersys}, we have
	\begin{equation}
		\label{}
		\de{}{x} \left(\begin{array}{cc}
			U & X\\
			V & Y
		\end{array}\right ) = \begin{pmatrix}  0 & s\sigma_3 \\ s\left (SB(sx)-\la S\right ) & 0 \end{pmatrix} \left(\begin{array}{cc}
			U & X\\
			V &  Y
		\end{array}\right ), \qquad \sigma_3  = \begin{pmatrix}
			1 & 0 \\ 0 & -1
		\end{pmatrix}.
	\end{equation}
	
	\begin{prop}
		\label{lemma:characteristic_det}
		For all $(\la,s) \in \R \times (0,1]$, the following are equivalent: 
		\vspace{-3mm}
		\begin{enumerate}
			\item $\la \in\spec(N|_{[0,s\ell]})\cap \R $,
			\item $ s^2\la \in \spec(N_s) \cap \R$,
			\item $\Lambda(\la,s)\cap \cD \neq\{0\}$,
			\item $\det X(\ell;\la,s) = 0$.
		\end{enumerate}
	\end{prop}
	
	We thus call $\det X(\ell;\la,s) $ the \emph{characteristic determinant}: the real eigenvalue curves in the $\la s$-plane are given by the zero set $\{(\la,s): \det X(\ell;\la,s) =0\}$. \Cref{spectralflow} illustrates some examples of these curves under \cref{hypo:NLS}.
	
	\begin{proof}
		The discussion following \eqref{eq:N_EVP_rescaled} gives the equivalence of (1) and (2), while the equivalence of (2) and (3) was given in \cref{prop:Lag_interp}. We show the equivalence of (3) and (4). Fix $s\in(0,1]$ and $\la \in\R$ and consider the $8\times 4$ matrix
		\begin{equation*}
			\cZ(\la,s)
			\coloneqq \begin{pmatrix}
				U(0;\la,s) & X(0;\la,s) \\
				U(\ell;\la,s) & X(\ell;\la,s)  \\
				-V(0;\la,s) & -Y(0;\la,s) \\
				V(\ell;\la,s)  & Y(\ell;\la,s) 
			\end{pmatrix} = \begin{pmatrix}
				I_2 & 0_2 \\
				U(\ell;\la,s) & X(\ell;\la,s)  \\
				0_2 & -I_2 \\
				V(\ell;\la,s)  & Y(\ell;\la,s)
			\end{pmatrix}.
		\end{equation*}
		Notice that the columns of $\cZ(\la,s)$ are precisely the rescaled trace (cf. \eqref{Trace}) of four linearly independent functions in $\cK_{\la,s}$ (recall that the entries of $Y(\cdot\,; \la,s)$ and $V(\cdot\,; \la,s)$ satisfy $r_s = s^{-1} \p_x u_s$ and $z_s = -s^{-1}  \p_x v_s$), and thus are a basis for our Lagrangian subspace $\Lambda(\la,s)$.
		
		A nontrivial intersection of the four-dimensional linear subspaces $\Lambda(\la, s)$ and $\cD$ of $\R^8$ occurs if and only if the $8\times 8$ matrix formed by their bases has zero determinant. Therefore, 
		\begin{equation*}
			\Lambda(\la, s) \cap \cD \neq \{0\} \iff \det \begin{pmatrix}
				I & 0 & 0 & 0  \\
				U(\ell;\la,s) & X(\ell;\la,s)  & 0 & 0 \\
				-0 & -I & I & 0 \\
				V(\ell;\la,s)  & Y(\ell;\la,s) & 0 & I
			\end{pmatrix} =0 \iff \det X(\ell;\la,s) =0,
		\end{equation*}
		as required.
	\end{proof}

	\begin{figure}[]
		\hspace*{\fill}
		\subcaptionbox{ \label{specflow1} }
		{\includegraphics[width=0.4\textwidth]{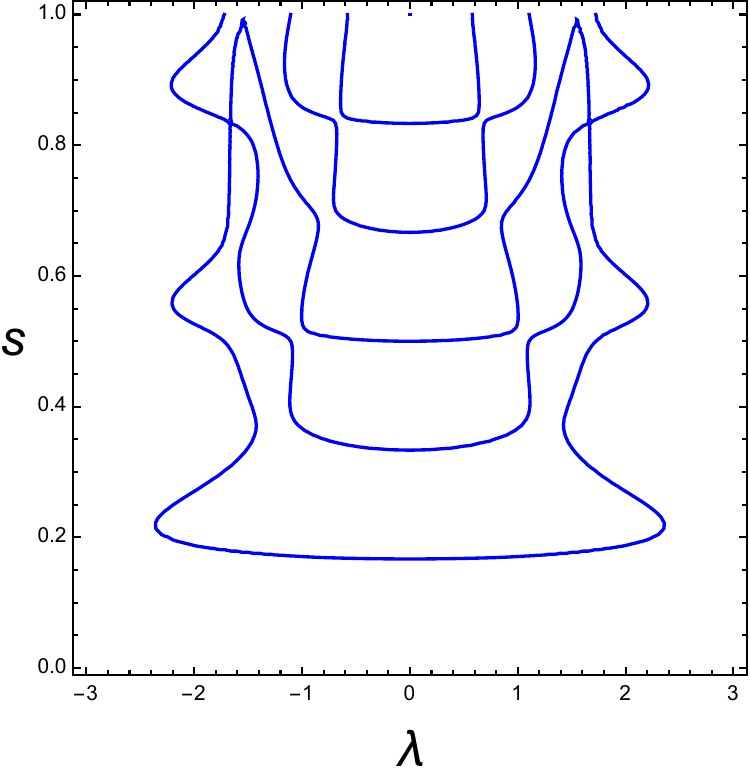}}\hfill  
		\subcaptionbox{\label{specflow3} }
		{\includegraphics[width=0.4\textwidth]{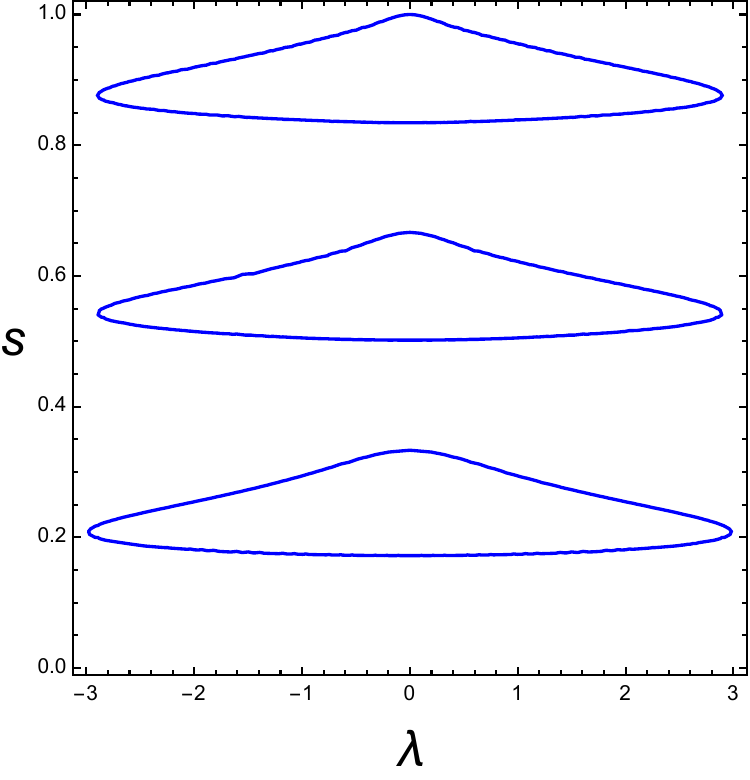}} 
		\hspace*{\fill}
		\caption{Real eigenvalue curves $s^2\la\in\spec(N_s)\cap\R$  {under \cref{hypo:NLS}(i) associated with a $T$-periodic stationary state $\phi_0$}   with nonlinearity $f(\phi^2)=\phi^2$ and $\be=-2$. In (a) $\phi_0$ is a  { positive Jacobi dnoidal function (i.e. an orbit located inside the homoclinic orbit in the right half plane in \cref{fig:phaseA})}   satisfying $\phi_0'(0)=\phi_0'(\ell)=0$ with $\ell = 3T=9.9398$. In (b) $\phi_0$ is a Jacobi cnoidal function  {(i.e. an orbit located outside the homoclinic orbit in \cref{fig:phaseA})}   satisfying $\phi_0(0)=\phi_0(\ell)=0$ with $\ell = 3T/2=10.0391$.}
		\label{spectralflow}
	\end{figure}

	\subsection{Analytic description}
	\label{subsec:Lyap_Schmidt}
	
	We will generalise \cref{thm:sL} to \cref{thm:sL:general}, which is a consequence of the following general result. We remind the reader that $n\leq 2$ in the current paper; see \cref{rem:notation}.  Below, dot denotes $d/d\la$.
	
	\begin{prop}\label{prop:M}
		Assume $\dim \ker(N_{s_0} - s_0^2\la_0)=n$ with basis $\{\bu_{s_0}^{(1)}, \dots, \bu_{s_0}^{(n)} \}$. There exists an $n\times n$ matrix $M(\lambda,s)$, defined near $(\lambda_0,s_0)$, such that $s^2\lambda \in \spec(N_s)$ if and only if $\det M(\lambda,s) = 0$. This matrix satisfies $M(\lambda_0,s_0) = 0$ and 
		\begin{equation}
			\label{41}
			\frac{\p M_{ij}}{\p \lambda}(\lambda_0,s_0) = - s_0^2 \left<\bu_{s_0}^{(i)}, S\bu_{s_0}^{(j)}\right>, \quad
			\frac{\p M_{ij}}{\p s}(\lambda_0,s_0) = \left< \big(\p_sB_{s_0} - 2 s_0\lambda_0\big) \bu_{s_0}^{(i)}, S\bu_{s_0}^{(j)}\right>.
		\end{equation}
		Moreover, if $\big<\bu_{s_0}^{(i)}, S\bu_{s_0}^{(j)}\big> = 0$ for all $i,j=1,\dots,n$, then
		\begin{equation}
			\label{M^lala}
			\frac{\p^2 M_{ij}}{\p \lambda^2}(\lambda_0,s_0) = -2s_0^4 \left<\bv_{s_0}^{(i)}, S\bu_{s_0}^{(j)}\right>,
		\end{equation}
		where $\bv_{s_0}^{(i)}\in\dom(N_{s_0})$ solves the inhomogeneous equation $(N_{s_0} - s_0^2 \lambda_0)\bv^{(i)}_{s_0} = \bu^{(i)}_{s_0}$.
	\end{prop}
	\begin{rem}
		Just as in \cref{rem:contribute}, for \eqref{M^lala} it suffices to consider those solutions to the inhomogeneous equation that satisfy $\bv^{(i)}_{s_0} \perp \bu^{(j)}_{s_0}$ for $i,j=1,\dots,n$.
	\end{rem}
	
	The definition of $M$, which requires some preparation, is given in \eqref{mdef}. 
	
	\begin{proof}
		We construct $M(\lambda,s)$ using Lyapunov--Schmidt reduction. The first step is to split the eigenvalue equation $(N_s - s^2\lambda)\bu = 0$ into two parts, one of which can always be solved uniquely. Let $P$ denote the $L^2$-orthogonal projection onto $\ker (N^*_{s_0} - s_0^2\lambda_0)$, so that $I-P$ is the projection onto $\ker (N^*_{s_0} - s_0^2\lambda_0)^\perp = \ran (N_{s_0} - s_0^2\lambda_0)$.  It follows that $s^2\lambda$ is an eigenvalue of $N_s$ if and only if there exists a nonzero $\bu\in\dom(N_s)$ such that both
		\begin{equation}
			\label{Gsplit1}
			P(N_s - s^2\lambda)\bu  = 0
		\end{equation}
		and
		\begin{equation}
			\label{Gsplit2}
			(I-P)(N_s - s^2\lambda)\bu  = 0
		\end{equation}
		hold.
		
		We first consider \eqref{Gsplit2}. Defining $X_0 = \ker (N_{s_0} - s_0^2\lambda_0) ^\perp \cap H^2(0,\ell) \cap H^1_0(0,\ell)$, we have that any $\bu \in H^2(0,\ell) \cap H^1_0(0,\ell)$ can be written uniquely as 
		\[
		\bu = \sum_{i=1}^n t_i \bu_{s_0}^{(i)} +  \tilde{\bu},
		\]
		where $t_i \in \R$ and $\tilde{\bu} \in X_0$. This means \eqref{Gsplit2} holds if and only if there exists a vector $\mathbf{t} = (t_1, \dots ,t_n) \in \R^n$ and a function $\tilde{\bu} \in X_0$ such that
		\begin{equation}
			\label{Gsplit3}
			(I-P)(N_s - s^2\lambda) \left (\sum_{i=1}^n t_i \bu_{s_0}^{(i)} + \tilde{\bu}\right ) = 0.
		\end{equation}
		We claim that for each $(\mathbf{t},\lambda,s)$ there exists a unique $\tilde{\bu} = \tilde{\bu}(\mathbf{t},\lambda,s) \in X_0$ satisfying \eqref{Gsplit3}. Writing this equation out explicitly, it is
		\[
		(I-P)(N_s - s^2\lambda)\tilde{\bu}(\mathbf{t},\lambda,s) = - (I-P)(N_s - s^2\lambda)\sum_{i=1}^n t_i \bu_{s_0}^{(i)}.
		\]
		We define
		\[
		T(\lambda,s) \colon X_0 \to \ran (N_{s_0} - s_0^2\lambda_0) , \qquad T(\lambda,s) = (I-P)\big(N_s - s^2\lambda\big)\Big|_{X_0},
		\]
		and observe that $T(\lambda_0,s_0)$ is invertible, hence $T(\lambda,s)$ is also invertible for nearby $(\lambda,s)$. Defining
		\[
		A(\lambda,s) : X_0^\perp  \to X_0, \qquad A(\lambda,s)  = -T^{-1}(\lambda,s) (I-P)\big(N_s - s^2\lambda\big)\Big|_{X_0^\perp},
		\]
		where $X_0^\perp= \ker (N_{s_0} - s_0^2\lambda_0) $, the unique solution to \eqref{Gsplit3} is thus
		\begin{equation}
			\label{uhat}
			\tilde{\bu}(\mathbf{t},\lambda,s) = A(\lambda,s) \sum_{i=1}^n t_i \bu_{s_0}^{(i)}.
		\end{equation}
		So far we have shown that the equation $(I-P)(N_s - s^2\lambda)\bu  = 0$ is satisfied if and only if $\bu$ has the form
		\begin{equation}
			\label{uform}
			\bu = \sum_{i=1}^n t_i \bu_{s_0}^{(i)} + A(\lambda,s)\sum_{i=1}^n t_i \bu_{s_0}^{(i)} = \big(I + A(\lambda,s)\big)\sum_{i=1}^n t_i \bu_{s_0}^{(i)}
		\end{equation}
		for some $\mathbf{t}\in \R^n$. We conclude that there exists $\bu$ for which $(N_s - s^2\lambda)\bu  = 0$ holds if and only if
		\begin{equation}
			\label{PG}
			P(N_s - s^2\lambda) \big(I + A(\lambda,s)\big)\left( \sum_{i=1}^n t_i \bu_{s_0}^{(i)}\right) = 0
		\end{equation}
		for some $\mathbf{t}\in \R^n$. Moreover, $\bu$ is nonzero if and only if $\mathbf{t}$ is nonzero. Finally, we observe that $\ker (N^*_{s_0} - s_0^2\lambda_0)$ is spanned by $\{S\bu_{s_0}^{(1)}, S\bu_{s_0}^{(2)}, \dots, S\bu_{s_0}^{(n)}\}$, and so \eqref{PG} is equivalent to
		\begin{equation}
			\label{systemeqs}
			\left<(N_s - s^2\lambda)\big(I + A(\lambda,s)\big) \left(\sum_{i=1}^n t_i \bu_{s_0}^{(i)}\right), S\bu_{s_0}^{(j)}\right> = 0, \qquad j=1,\dots, n. 
		\end{equation}
		Defining the $n\times n$ matrix $M(\la,s)$ by 
		\begin{equation}
			\label{mdef}
			M_{ij}(\la,s) = \left<(N_s - s^2\lambda) \big(I + A(\lambda,s)\big)\bu^{(i)}_{s_0}, S\bu_{s_0}^{(j)}\right>, \quad i,j=1, \dots, n,
		\end{equation}
		the system of $n$ equations \eqref{systemeqs} may be written as $M(\la,s) \mathbf{t} = 0$, which is satisfied for a nonzero vector $\mathbf{t}$ if and only if $\det M(\la,s)=0$. This completes the first part of the proof. 
		
		It follows that $M(\lambda_0,s_0) = 0$. We then compute
		\begin{align}
			\frac{\p M_{ij}}{\p \lambda}(\lambda_0,s_0) &= \left<-s_0^2 \big(I + A(\lambda_0,s_0)\big)\bu^{(i)}_{s_0} + (N_{s_0} - s_0^2\lambda_0) \p_\la A(\lambda_0,s_0) \bu^{(i)}_{s_0}, S\bu^{(j)}_{s_0}\right> \\
			&= -s_0^2 \left<\bu^{(i)}_{s_0}, S\bu^{(j)}_{s_0}\right>,
		\end{align}
		where in the second line we have used the fact that $A(\lambda_0,s_0) \bu^{(i)}_{s_0}  = 0$ and 
		\[
		\left<(N_{s_0} - s_0^2\lambda_0) \p_\la A(\lambda_0,s_0) \bu^{(i)}_{s_0}, S\bu^{(j)}_{s_0}\right> = 
		\left< \p_\la A(\lambda_0,s_0) \bu^{(i)}_{s_0}, (N^*_{s_0} - s_0^2\lambda_0) S\bu^{(j)}_{s_0}\right> = 0,
		\]
		because $S\bu^{(j)}_{s_0} \in \ker(N^*_{s_0} - s_0^2\lambda_0)$. The $s$ derivative is computed similarly.
		
		Finally, if $\p_\la M(\la_0,s_0)  = 0$, we have
		\begin{align}
			\frac{\p^2 M_{ij}}{\p \lambda^2}(\lambda_0,s_0) 
			&= -2s_0^2 \left<\p_\la A(\lambda_0,s_0)\bu^{(i)}_{s_0}, S\bu^{(j)}_{s_0}\right>,
		\end{align}
		where $
		\big<(N_{s_0} - s_0^2\lambda_0) \p_{\la\la} A(\lambda_0,s_0) \bu^{(i)}_{s_0}, S\bu^{(j)}_{s_0}\big> = 0
		$ again using $S\bu^{(j)}_{s_0} \in \ker(N^*_{s_0} - s_0^2\lambda_0)$. To compute $\p_\la A(\lambda_0,s_0)\bu^{(i)}_{s_0}$, we use the definition of $A(\la,s)$ to write
		\[
		T(\lambda,s)A(\lambda,s) \bu^{(i)}_{s_0} = - (I-P)\big(N_s - s^2\lambda\big)\bu^{(i)}_{s_0}.
		\]
		Differentiating in $\lambda$ and again using the fact that $A(\lambda_0,s_0) \bu^{(i)}_{s_0}  = 0$, we get
		\[
		T(\lambda_0,s_0)\p_\la A(\lambda_0,s_0) \bu^{(i)}_{s_0} = s_0^2(I-P)\bu^{(i)}_{s_0}.
		\]
		The fact that $\big<\bu^{(i)}_{s_0}, S\bu^{(j)}_{s_0}\big>=0$ for all $i,j$ implies $(I-P)\bu^{(i)}_{s_0} = \bu^{(i)}_{s_0}$. Setting $s_0^2\bv^{(i)}_{s_0} = \p_\la A(\lambda_0,s_0)\bu^{(i)}_{s_0}$, we see from the definition of $T$ that
		\[
		T(\lambda_0,s_0)(s_0^2\bv^{(i)}_{s_0}) = s_0^2(I-P)(N_{s_0} - s_0^2\lambda_0)\bv^{(i)}_{s_0} = s_0^2(N_{s_0} - s_0^2\lambda_0)\bv^{(i)}_{s_0}
		\]
		and the result follows. 
	\end{proof}
	
	Comparison with the symmetric matrices \eqref{M^s}, \eqref{M^la} and \eqref{bigmfM2} associated with the first and second order crossing forms reveals that the partial derivatives of the matrix $M$ satisfy
	\begin{equation}\label{MandmathfrakM}
		\frac{\p M}{\p s}(\lambda_0,s_0) = s_0 \,\mathfrak M_{s_0}, \quad 	 \frac{\p M}{\p \lambda}(\lambda_0,s_0) = s_0 \,\mathfrak M_{\la_0}, \quad  \pdes{M}{\la}(\la_0,s_0) = s_0\,\mathfrak{M}^{(2)}_{\la_0},
	\end{equation}
	where the last formula holds when $\p_\la M(\la_0,s_0)=0$. In particular, in the case $\dim\ker(N_{s_0}-s_0^2\la_0) = 1$ (so that $M$ is a scalar), we have
	\begin{equation}
		\label{mcomp}
		\frac{\p M}{\p s}(\lambda_0,s_0) = s_0 \,\mathfrak m_{s_0}(q), \quad 	 \frac{\p M}{\p \lambda}(\lambda_0,s_0) = s_0 \,\mathfrak m_{\la_0}(q), \quad \pdes{M}{\la}(\la_0,s_0) = s_0\mathfrak{m}^{(2)}_{\la_0}(q),
	\end{equation}
	where again the last formula holds when $\p_\la M(\la_0,s_0)=0$. Combining \eqref{mcomp} with the implicit function theorem immediately yields the following Hadamard-type formulas for the derivatives of the real eigenvalue curves in terms of the crossing forms. 
	
	\begin{cor}\label{cor:Hadamard}
		Under the assumption that $\dim \ker(N_{s_0} - s_0^2\la_0)=1$, the following hold:
		\begin{enumerate}
			\item If $\mathfrak m_{\lambda_0} \neq 0$, then there exists a $C^2$ curve $\lambda(s)$ near $s_0$ such that
			\begin{equation}\label{eq:lamdash}
				\lambda'(s_0) = - \frac{\mathfrak m_{s_0}(q)}{\mathfrak m_{\la_0}(q)}.
			\end{equation}
			\item If $\mathfrak m_{s_0} \neq 0$, then there exists a $C^2$ curve $s(\lambda)$ near $\lambda_0$ such that
			\begin{equation}\label{eq:sdot}
				\dot s(\lambda_0) = - \frac{\mathfrak m_{\la_0}(q)}{\mathfrak m_{s_0}(q)}.
			\end{equation}
			Moreover, $\dot{s}(\lambda_0) = 0$ if and only if $\mathfrak{m}_{\la_0}(q)=0$, and in this case
			\begin{equation}\label{eq:sddot}
				\ddot s(\lambda_0) = - \frac{\mathfrak m_{\la_0}^{(2)}(q)}{\mathfrak m_{s_0}(q)}.
			\end{equation}
		\end{enumerate}
	\end{cor}
	
	Using this, we can construct a curve $s(\lambda)$ through any simple conjugate point and determine its concavity by an explicit formula.
	
	\begin{theorem}
		\label{thm:sL:general}
		If $\dim \ker N_{s_0} = 1$, then for $|\lambda| \ll 1$ there exists a $C^2$ curve $s(\lambda)$ such that $s(\la)^2\lambda\in\spec(N_{s(\la)})$, and a continuous curve $\bu_{s(\lambda)}$ of eigenfunctions such that $\bu_{s(\lambda)} \to \bu_{s_0}$ as $\lambda \to 0$. Moreover, $s(0) = s_0$, $\dot s(0) = 0$, and the concavity of $s(\lambda)$ can be determined as follows:
		\begin{enumerate}
			\item If $0 \in \spec(L_-^{s_0}) \setminus\spec(L_+^{s_0})$ with eigenfunction $v_{s_0} \in \ker L_-^{s_0}$, then
			\begin{equation}
				\label{s''L-1}
				\ddot s(0) = \f{2s_0^5}{\ell} \,\f{\langle \widehat{u}_{s_0}, v_{s_0} \rangle }{ \left( v_{s_0}'(\ell)\right )^2}
			\end{equation}
			where $\widehat{u}_{s_0} \in H^2(0,\ell) \cap H^1_0(0,\ell)$ is the unique solution to $L_+^{s_0} \widehat{u}_{s_0}  = v_{s_0}$.
			
			\item If $0 \in \spec(L_+^{s_0}) \setminus \spec(L_-^{s_0})$ with eigenfunction $u_{s_0} \in \ker L_+^{s_0}$, then
			\begin{equation}
				\label{s''L+2}
				\ddot s(0) = - \f{2s_0^5}{\ell}  \f{\langle  \widehat{v}_{s_0},  u_{s_0} \rangle}{ \left (u_{s_0}'(\ell)\right )^2 }
			\end{equation}
			where $\widehat{v}_{s_0}\in H^2(0,\ell) \cap H^1_0(0,\ell)$ is the unique solution to $-L_-^{s_0} \widehat{v}_{s_0} = u_{s_0}$.
		\end{enumerate}
	\end{theorem}
	
	\begin{proof}
		\cref{lemma:scrossingform} implies $\mathfrak m_{s_0} \neq 0$, so the existence of $s(\lambda)$ follows from \cref{cor:Hadamard}. \cref{cor:degeneracy_of_form_lambda} then gives $\dot s(0) = 0$. From \eqref{uform} we see that $\bu_{s(\lambda)} = \big(I + A(\lambda,s(\lambda)) \big) \bu_{s_0}$ is an eigenfunction of $N_{s(\lambda)}$ for the eigenvalue $s^2(\lambda)\lambda$. Since $A(\lambda,s(\lambda))$ is continuous in $\lambda$ and $A(0,s_0) \bu_{s_0} = 0$, the convergence of $\bu_{s(\lambda)}$ to $\bu_{s_0}$ follows.
		
		It thus remains to prove \eqref{s''L-1} and \eqref{s''L+2}.
		If $0 \in \spec(L_-^{s_0}) \setminus \spec(L_+^{s_0})$ then $u_{s_0}$ is trivial, so equations \eqref{maslov_form_s_lambda=0} and \eqref{2ndorder_lambda_1dim} give
		\begin{equation}\label{L-den}
			\mathfrak{m}_{s_0}(q)  = \f{\ell}{s_0^2}\left (v_{s_0}'(\ell)\right )^2, \qquad \mathfrak m_{\la_0}^{(2)}(q) = - 2 s_0^3  \langle \widehat{u}_{s_0}, v_{s_0} \rangle.
		\end{equation}
		Substituting these into \eqref{eq:sddot} immediately gives \eqref{s''L-1}.
		The case $0 \in \spec(L_+^{s_0}) \setminus \spec(L_-^{s_0})$ is almost identical. Here we have
		\[
		\mathfrak{m}_{s_0}(q)  = -\f{\ell}{s_0^2}\left (u_{s_0}'(\ell)\right )^2, \qquad \mathfrak m_{\la_0}^{(2)}(q) = - 2 s_0^3  \langle \widehat{v}_{s_0}, u_{s_0} \rangle,
		\]
		and \eqref{s''L+2} follows.
	\end{proof}

	\subsection{When $\la_0=0$ has geometric multiplicity two}\label{subsec:multiplicity2}
	
	In this section we focus on the case of a geometrically double eigenvalue at zero.  Since $0 \in \spec(L_+^{s_0}) \cap \spec(L_-^{s_0})$, we have $\ker( N_{s_0}) = \spn\{ \bu^{(1)}_{s_0},\bu^{(2)}_{s_0} \}$ where the $\bu_{s_0}^{(i)}$ are given in \eqref{eq:doubleeigfns}. Applying \cref{prop:M} with $\lambda_0 = 0$ and $n=2$, we will show the following. Again, dot denotes $d/d\la$.
	
	\begin{theorem}
		\label{thm:touching}
		Suppose $\dim \ker N_{s_0} = 2$, and denote the corresponding eigenfunctions of $L_+^{s_0}$ and $L_-^{s_0}$ by $u^{(1)}_{s_0}$ and $v^{(2)}_{s_0}$, respectively.
		\begin{enumerate}
			\item  If $\big<u^{(1)}_{s_0}, v^{(2)}_{s_0} \big> \neq 0$, then $s^2\lambda \notin \spec(N_s)$ for $(\lambda,s)$ in a punctured neighbourhood of $(0,s_0)$.
			\item If $\big<u^{(1)}_{s_0}, v^{(2)}_{s_0} \big> = 0$ and
			\begin{equation}\label{uglycondition}
				\f{\big \langle \widehat{v}^{(1)}_{s_0}, u^{(1)}_{s_0} \big \rangle}{\big (\p_x u^{(1)}_{s_0}(\ell)\big )^2} + \f{\big \langle \widehat{u}^{(2)}_{s_0}, v^{(2)}_{s_0}\big \rangle }{\big ( \p_x v^{(2)}_{s_0}(\ell)\big )^2 } \neq 0,
			\end{equation}
			where $\widehat{u}^{(2)}_{s_0}\in\dom(L_+^{s_0})$ and $\widehat{v}^{(1)}_{s_0}\in\dom(L_-^{s_0})$ denote solutions to
			\begin{equation}
				\label{LUV}
				L_+^{s_0} \widehat{u}^{(2)}_{s_0} = v^{(2)}_{s_0}, \qquad -L_-^{s_0} \widehat v_{s_0}^{(1)} = u^{(1)}_{s_0},
			\end{equation}
			then for $|\lambda| \ll 1$ there exist $C^2$ curves $s_1(\lambda)$ and $s_2(\lambda)$ such that
			\begin{enumerate}
				\item[(i)] $s^2_{1,2}(\lambda) \lambda \in \spec\big(N_{s_{1,2}(\la)}\big)$,
				\item[(ii)] $s_{1,2}(0) = s_0$,
				\item[(iii)] $\dot s_{1,2}(0) = 0$, 
			\end{enumerate}
			and the concavities satisfy
			\begin{equation}\label{concavities}
				\ddot s_1(0) = -\frac{2s_0^5}{\ell} \f{\big \langle \widehat{v}^{(1)}_{s_0}, u^{(1)}_{s_0} \big \rangle}{\big ( \p_x u^{(1)}_{s_0}(\ell)\big )^2 }, \qquad 
				\ddot s_2(0) = \frac{2s_0^5}{\ell} \f{\big \langle \widehat{u}^{(2)}_{s_0}, v^{(2)}_{s_0}\big \rangle }{\big ( \p_x v^{(2)}_{s_0}(\ell)\big )^2 }.
			\end{equation}
			Moreover, there exist continuous curves $\bu_{s_1(\lambda)}$ and $\bu_{s_2(\lambda)}$ of eigenfunctions such that
			\begin{equation}
				\bu_{s_1(\lambda)} \to \bu^{(1)}_{s_0} = \begin{pmatrix} u^{(1)}_{s_0} \\ 0 \end{pmatrix}, \qquad 
				\bu_{s_2(\lambda)} \to \bu^{(2)}_{s_0} = \begin{pmatrix} 0 \\ v^{(2)}_{s_0}  \end{pmatrix}
			\end{equation}
			as $\lambda \to 0$.
			
		\end{enumerate}
	\end{theorem}
	
	The condition \eqref{uglycondition} will be discussed in \cref{rem:concavities_signs} below.
	
	\begin{rem}
		As in \cref{rem:contribute} the solutions $\widehat{u}^{(2)}_{s_0}$ and $\widehat{v}^{(1)}_{s_0}$ in \eqref{LUV} are not unique, but the expressions in \eqref{uglycondition} and \eqref{concavities} do not depend on the choice of solution.
	\end{rem}
	
	We prove the theorem by studying the zero set of $m(\la,s) \coloneqq \det M(\la,s)$, where $M$ is given in \eqref{mdef}. We thus start with some elementary calculations for the higher order derivatives of $m$. These will be used to prove the existence of the eigenvalue curves $s_{1,2}(\lambda)$ and also to evaluate their first and second derivatives.
	
	\begin{lemma}
		\label{m:deriv}
		Under the assumptions of \cref{thm:touching}, we have
		\begin{align}
			m(0,s_0) = \frac{\p m}{\p s}(0,s_0) = \frac{\p m}{\p \lambda}(0,s_0) = 
			\frac{\p^2 m}{\p s \p \lambda}(0,s_0) = 0
		\end{align}
		and
		\begin{equation}
			\label{m2nd_derivs}
			\frac{\p^2 m}{\p s^2}(0,s_0) =  -\f{2\ell^2}{s_0^2}  \left ( \p_x u^{(1)}_{s_0}(\ell) \right )^2 \left ( \p_x v^{(2)}_{s_0}(\ell) \right )^2, \qquad
			\frac{\p^2 m}{\p \lambda^2}(0,s_0) =  -2s_0^4  \left<u^{(1)}_{s_0}, v^{(2)}_{s_0}\right>^2.
		\end{equation}
		Moreover, if $\big<u^{(1)}_{s_0},v^{(2)}_{s_0}\big> = 0$, then
		\begin{align}
			&\frac{\p^3 m}{\p s\p \lambda^2}(0,s_0) = 2 \ell s_0^3 \left ( \p_x u^{(1)}_{s_0}(\ell) \right )^2 \left \langle \widehat{u}^{(2)}_{s_0}, v^{(2)}_{s_0}\right \rangle - 2\ell s_0^3 \left ( \p_x v^{(2)}_{s_0}(\ell) \right )^2\left \langle  \widehat{v}^{(1)}_{s_0}, u^{(1)}_{s_0}\right \rangle \label{m3mixed}\\
			&\frac{\p^3 m}{\p \lambda^3}(0,s_0) = 0, \qquad 
			\frac{\p^4 m}{\p \lambda^4}(0,s_0) =  24 s_0^8 \left \langle \widehat{u}^{(2)}_{s_0}, v^{(2)}_{s_0}\right \rangle\left \langle\widehat{v}^{(1)}_{s_0}, u^{(1)}_{s_0} \right \rangle
			\label{m4}
		\end{align}
		with $\widehat{u}^{(2)}_{s_0}$ and $\widehat{v}^{(1)}_{s_0}$ as in \eqref{LUV}.
	\end{lemma}
	
	\begin{proof}
		Writing $M = \begin{pmatrix} a & b \\ c & d \end{pmatrix}$, so that $m = ad - bc$, we compute
		\begin{align*}
			\p_s m &= (\p_s a) \,d + a \,(\p_s d) - (\p_s b) \,c - b \,(\p_s c), \quad \\
			\p_{s}^2 m &= (\p_{s}^2 a)\, d + 2\, (\p_s a)\, (\p_s d) + a \,(\p_{s}^2 d) - (\p_{s}^2  b) \,c - 2 \,(\p_{s} b)\, (\p_s c) - b \,(\p_{s}^2  c)
		\end{align*}
		and so at $(0,s_0)$ we have
		\begin{subequations}
			\begin{equation}
				\p_s m = 0, \qquad
				\p_{s}^2 m = 2 \,(\p_s a) \,(\p_s d) - 2\, (\p_s b) \,(\p_s c)
			\end{equation}
			because $a=b=c=d = 0$ there (recall that $M(\la_0,s_0)=0$). Similarly, we find that
			\label{subeqns}
			\begin{align}
				\p_{\la} m &= 0, \\
				\p_{\la}^2m &= 2 (\p_{\la}a) (\p_{\la} d) - 2 (\p_{\la} b) (\p_{\la} c), \\
				\p_{s\la} m &= (\p_s a) (\p_{\la} d) + (\p_{\la} a) (\p_{s} d) - (\p_{s} b) (\p_{\la} c) - (\p_{\la} b) (\p_{s} c).
			\end{align}			
		\end{subequations}
		at $(0,s_0)$. To evaluate the second derivatives, it remains to differentiate the components of $M$. By \cref{prop:M}, for $i,j=1,2$ we have
		\begin{equation}
			\label{433}
			\frac{\p M_{ij}}{\p \lambda}(0,s_0) = - s_0^2 \left<\bu^{(i)}_{s_0},  S\bu^{(j)}_{s_0}\right>, \qquad
			\frac{\p M_{ij}}{\p s}(0,s_0) = \left< \p_sB_{s_0} \bu^{(i)}_{s_0}, S\bu^{(j)}_{s_0}\right>.
		\end{equation}
		It follows from \eqref{MandmathfrakM} and \eqref{matrix_firstorder_lambda} that
		\[
		\frac{\p M}{\p \lambda}(0,s_0) = - s_0^2 \begin{pmatrix} 0 & \big<u^{(1)}_{s_0},v^{(2)}_{s_0}\big> \\ \big<u^{(1)}_{s_0},v^{(2)}_{s_0}\big> & 0 \end{pmatrix},
		\]
		so that at $(0,s_0)$, we have $\p_{\la} a = \p_{\la} d = 0$ and $\p_{\la} b = \p_{\la} c= -s_0^2 \big<u^{(1)}_{s_0},v^{(2)}_{s_0}\big>$. Similarly, it follows from \eqref{MandmathfrakM} and \eqref{eq:capitalmathfrakMn=2} that 
		\begin{equation}\label{matrix:Ms}
			\frac{\p M}{\p s}(0,s_0) = \f{\ell}{s_0} \begin{pmatrix} - \big ( \p_x u^{(1)}_{s_0}(\ell) \big )^2 & 0 \\ 0 & \big ( \p_x v^{(2)}_{s_0}(\ell) \big )^2 \end{pmatrix},
		\end{equation}
		hence at $(0,s_0)$ we have $\p_{s} a = - s_0^{-1} \ell \big ( \p_x u^{(1)}_{s_0}(\ell) \big )^2 $, $\p_{s} d = s_0^{-1} \ell\big ( \p_x v^{(2)}_{s_0}(\ell) \big )^2 $ and $\p_{s} b = \p_{s} c = 0$. The claimed formulas for $\p_{s}^2 m$, $\p_{s\la} m$ and $\p_{\la}^2 m$ now follow from \eqref{subeqns}.
		
		If $\big<u^{(1)}_{s_0},v^{(2)}_{s_0}\big> = 0$, then $\p_\lambda b = \p_\lambda c = 0$ at $(0,s_0)$. This implies that $\p_{\la}^3 m= 0$ and 
		\begin{equation}\label{higher_derivs_m}
			\p_{\la}^4 m = 6 \left ((\p_{\la}^2 a) \,(\p_{\la}^2 d) - (\p_{\la}^2 b) \,(\p_{\la}^2 c) \right ), \quad \p_{s\la\la} m = (\p_s a)\,( \p_{\la}^2 d) + (\p_{\la}^2a)\,( \p_s d	)
		\end{equation}
		at $(0,s_0)$. 
		Using \eqref{MandmathfrakM} and \eqref{bigmfM2_2dim} we obtain
		\begin{equation}\label{MLL}
			\frac{\p^2 M}{\p \lambda^2}(0,s_0) = -2s_0^4\begin{pmatrix}
				\big \langle \widehat{v}_{s_0}^{(1)}, u_{s_0}^{(1)}\big \rangle & 0 \\
				0 & \big \langle \widehat{u}_{s_0}^{(2)}, v_{s_0}^{(2)} \big \rangle
			\end{pmatrix},
		\end{equation}
		hence $\p_{\la}^2 b  = \p_{\la}^2 c = 0$ and it follows that
		\begin{align*}
			\p_{\la}^4 m &= 6 (\p_{\la}^2 a) (\p_{\la}^2 d) = 24 s_0^8  \big \langle \widehat{v}_{s_0}^{(1)}, u_{s_0}^{(1)}\big \rangle   \big \langle \widehat{u}_{s_0}^{(2)}, v_{s_0}^{(2)} \big \rangle.
		\end{align*}
		The claimed formula for $\p_{s\la\la} m$ follows directly from \eqref{higher_derivs_m}.
	\end{proof}
	
	The next elementary lemma will be used to prove differentiability of the eigenvalue curves in the second part of \cref{thm:touching}. In what follows, dot denotes $d/d\la$.

	\begin{lemma}
		\label{lem:Delta}
		If $\Delta$ is a smooth function with $\Delta(\lambda) = \alpha \lambda^4 + O(\lambda^5)$ as $|\la|\to 0$ for some $\alpha>0$, then $\delta(\lambda) := \sqrt{\Delta(\lambda)}$ is $C^2$ near $\lambda=0$, with $\dot\delta(0) = 0$ and $\ddot \delta(0) = 2\sqrt\alpha$.
	\end{lemma}

	\begin{proof}
		It is clear that $\delta$ is smooth except possibly at $\lambda=0$. For the first derivative we note that $\delta(\lambda)/\lambda \to 0$ as $\lambda \to 0$, so $\dot\delta(0) = 0$. For $\lambda \neq 0$ we compute
		\[
		\dot\delta(\lambda) = \frac12 \Delta(\lambda)^{-1/2} \dot \Delta(\lambda).
		\]
		Using $\Delta(\lambda) = \alpha \lambda^4 + O(\lambda^5)$ and $\dot\Delta(\lambda) = 4\alpha \lambda^3 + O(\lambda^4)$, we see that $\dot\delta(\lambda) \to 0$ as $\lambda \to 0$ and conclude that $\delta$ is $C^1$. Next, we observe that
		\[
		\frac{\dot\delta(\lambda) - \dot\delta(0)}{\lambda} = \frac12 \frac{\lambda^2}{\sqrt{\Delta(\lambda)}} \frac{\dot\Delta(\lambda)}{\lambda^3} \to 2 \sqrt\alpha,
		\]
		and hence $\ddot\delta(0)$ exists. A similar argument gives
		\[
		\ddot\delta(\lambda) = -\frac14 \frac{\dot\Delta(\lambda)^2}{\Delta(\lambda)^{3/2}} + \frac12 \frac{\ddot\Delta(\lambda)}{\sqrt{\Delta(\lambda)}} \to 2\sqrt\alpha
		\]
		as $\lambda \to 0$, so $\delta$ is $C^2$.
	\end{proof}

	\begin{proof}[Proof of Theorem \ref{thm:touching}]
		By assumption we have $m(0,s_0) = 0$. If $\big<u^{(1)}_{s_0}, v^{(2)}_{s_0} \big> \neq 0$, \cref{m:deriv} implies $m$ has a strict local maximum at $(0,s_0)$, so $m$ is negative (and in particular nonzero) in a punctured neighborhood of $(0,s_0)$. This proves the first case.

		For the second case we use the Malgrange preparation theorem (see \cite[\S IV.2]{Golubitsky_Guillemin}). We know from Lemma~\ref{m:deriv} that $m(0,s_0) = \p_{s} m(0,s_0) = 0$ and $\p_{s}^2 m(0,s_0) < 0$, so we can write
		\begin{equation}
			m(\lambda,s) = Q(\lambda,s)P(\lambda,s)
		\end{equation}
		in a neighbourhood of $(0,s_0)$, where
		\begin{equation}
			\label{eq:P}
			P(\lambda,s) = (s-s_0)^2 + B(\lambda)(s-s_0) + C(\lambda),
		\end{equation}
		$Q$, $B$ and $C$ are smooth, real-valued functions, and $Q$ does not vanish in a neighbourhood of $(0,s_0)$. This means $m$ locally has the same zero set as $P$.

		We claim that the discriminant $\Delta(\lambda) = B^2(\lambda) - 4 C(\lambda)$ satisfies
		\begin{equation}
			\label{claim}
			\Delta(\lambda) = \alpha \lambda^4 + O(\lambda^5) \,\,\,\text{as}\,\,\,|\la|\to 0, \qquad \alpha = \frac{\ddot B(0)^2}{4} - \frac{C^{(4)}(0)}{6} > 0.
		\end{equation}
		To see this, we compute the Taylor expansion of $\Delta(\la) = B(\la)^2 - 4C(\la)$ about $\la=0$ and show that $\Delta(0) = \dot \Delta(0) = \ddot\Delta(0) = \dddot\Delta(0) = 0$. For this it suffices to show that $B(0) = \dot B(0) = C(0) = \dot C(0) = \ddot C(0) = \dddot C(0) = 0$. That $\Delta^{(4)}(0) = 4!\alpha$ follows from the definition of $\Delta(\la)$.

		Using \cref{m:deriv} we obtain
		\[
		m(0,s_0) = Q(0,s_0) C(0) = 0.
		\]
		Since $Q(0,s_0) \neq 0$, this implies $C(0) = 0$. Similarly, we find that
		\begin{align*}
			\p_\la m(0,s_0) &= Q(0,s_0) \dot C(0) = 0 \\
			\p_\la^2 m(0,s_0) &= Q(0,s_0) \ddot C(0) = 0\\
			\p_\la ^3 m(0,s_0) &= Q(0,s_0) \dddot C(0) = 0\\
			\p_\la ^4 m(0,s_0) &= Q(0,s_0) C^{(4)}(0)
		\end{align*}
		and
		\begin{align*}
			\p_s m(0,s_0) &= Q(0,s_0) B(0) = 0 \\
			\p_{s \la} m(0,s_0) &= Q(0,s_0) \dot B(0) = 0 \\
			\p_{s \la\la} m(0,s_0) &= Q(0,s_0) \ddot B(0),
		\end{align*}
		which gives
		\[
		B(0) = \dot B(0) = C(0) = \dot C(0) = \ddot C(0) = \dddot C(0) = 0.
		\]
		We now observe that
		\[
		\p_{s}^2 m(0,s_0) = Q(0,s_0) \, \p_{s}^2 P(0,s_0) = 2 Q(0,s_0).
		\]
		Using the first formula from \eqref{m2nd_derivs}, this implies that
		\begin{equation}\label{Q(0,s_0)}
			Q(0,s_0) = -\f{\ell^2}{s_0^2}  \left ( \p_x u^{(1)}_{s_0}(\ell) \right )^2 \left ( \p_x v^{(2)}_{s_0}(\ell) \right )^2.
		\end{equation}
		Therefore, using \eqref{m4}, 
		\begin{equation}\label{C^{(4)}(0) }
			C^{(4)}(0) = \frac{\p_{\la}^4 m(0,s_0)}{Q(0,s_0)} =
			-24\f{s_0^{10}}{\ell^2} \f{ \big \langle \widehat{v}^{(1)}_{s_0}, u^{(1)}_{s_0}\big \rangle \big\langle\widehat{u}^{(2)}_{s_0}, v^{(2)}_{s_0}\big \rangle}
			{  \big ( \p_x u^{(1)}_{s_0}(\ell) \big )^2 \big ( \p_x v^{(2)}_{s_0}(\ell) \big )^2}\,.
		\end{equation}
		We similarly use \eqref{m3mixed} to compute
		\begin{equation}\label{ddotB(0)}
			\ddot B(0) = \frac{\p_{s \la\la} m(0,s_0)}{Q(0,s_0)}  = \f{2s_0^5}{\ell}  \left \{ \f{\big\langle\widehat{v}^{(1)}_{s_0}, u^{(1)}_{s_0}\big \rangle}{\big ( \p_x u^{(1)}_{s_0}(\ell) \big )^2 } - \f{\big\langle\widehat{u}^{(2)}_{s_0}, v^{(2)}_{s_0} \big\rangle}{\big ( \p_x v^{(2)}_{s_0}(\ell) \big )^2}  \right \}.
		\end{equation}
		
		Therefore
		\begin{align}\label{alpha}
			\alpha &= \frac{\ddot B(0)^2}{4} - \frac{C^{(4)}(0)}{6} = \f{s_0^{10}}{\ell^2} \left ( \f{\big\langle\widehat{v}^{(1)}_{s_0}, u^{(1)}_{s_0}\big \rangle}{\big ( \p_x u^{(1)}_{s_0}(\ell) \big )^2 } + \f{\big\langle\widehat{u}^{(2)}_{s_0}, v^{(2)}_{s_0} \big\rangle}{\big ( \p_x v^{(2)}_{s_0}(\ell) \big )^2}  \right )^2>0
		\end{align}
		on account of \eqref{uglycondition}, thus proving the claim. 
		
		Given \eqref{claim}, we have $\Delta(\lambda) > 0$ for small nonzero $\lambda$, and so the equation $P(\lambda,s) = 0$ has two solutions in $s$,
		\begin{equation}
			\label{spm}
			s_\pm(\lambda) := \frac{-B(\lambda) \pm \sqrt{\Delta(\lambda)}}{2} + s_0.
		\end{equation}
		It then follows from Lemma~\ref{lem:Delta} that both $s_\pm(\lambda)$ are $C^2$ in a neighbourhood of $\lambda=0$, with $\dot s_\pm(0) = -\dot B(0)/2 = 0$ and
		\begin{equation}
			\label{s:deriv}
			\ddot s_\pm(0) = \frac{-\ddot B(0) \pm 2\sqrt\alpha}{2},
		\end{equation}
		so the curves $s_\pm(\lambda)$ satisfy properties (i)--(iii) in the theorem. Substituting \eqref{ddotB(0)} and \eqref{alpha} into \eqref{s:deriv}, we obtain
		\begin{equation}\label{concavitiesexact}
			\ddot s_\pm(0) = \f{s_0^5}{\ell} \left\{ \f{\big\langle\widehat{u}^{(2)}_{s_0}, v^{(2)}_{s_0} \big\rangle}{\big ( \p_x v^{(2)}_{s_0}(\ell) \big )^2} - \f{\big\langle\widehat{v}^{(1)}_{s_0}, u^{(1)}_{s_0}\big \rangle}{\big ( \p_x u^{(1)}_{s_0}(\ell) \big )^2 } \pm 
			\left| \f{\big\langle\widehat{v}^{(1)}_{s_0}, u^{(1)}_{s_0}\big \rangle}{\big ( \p_x u^{(1)}_{s_0}(\ell) \big )^2 } + \f{\big\langle\widehat{u}^{(2)}_{s_0}, v^{(2)}_{s_0} \big\rangle}{\big ( \p_x v^{(2)}_{s_0}(\ell) \big )^2}\right | \right\}.
		\end{equation}
		If the quantity inside the absolute value (which is nonzero by \eqref{uglycondition}) is positive, we get
		\begin{equation}\label{s+s-1}
			\ddot s_+(0) =  \f{2s_0^5}{\ell}\f{\big \langle \widehat{u}^{(2)}_{s_0}, v^{(2)}_{s_0}\big \rangle }{\big ( \p_x v^{(2)}_{s_0}(\ell) \big )^2 }, \qquad  \ddot s_-(0) = - \f{2s_0^5}{\ell}\f{\big \langle \widehat{v}^{(1)}_{s_0}, u^{(1)}_{s_0} \big \rangle}{\big ( \p_x u^{(1)}_{s_0}(\ell) \big )^2 },
		\end{equation}
		in which case we define $s_1 \coloneqq s_-$ and $s_2 \coloneqq s_+$. If it is negative we get
		\begin{equation}\label{s+s_2}
			\ddot s_-(0)= \f{2s_0^5}{\ell}\f{\big \langle \widehat{u}^{(2)}_{s_0}, v^{(2)}_{s_0}\big \rangle }{\big ( \p_x v^{(2)}_{s_0}(\ell) \big )^2 }, \qquad\ddot s_+(0) = - \f{2s_0^5}{\ell}\f{\big \langle \widehat{v}^{(1)}_{s_0}, u^{(1)}_{s_0} \big \rangle}{\big ( \p_x u^{(1)}_{s_0}(\ell) \big )^2 },
		\end{equation}
		and we define $s_1\coloneqq s_+$ and $s_2 \coloneqq s_-$.
		
		To prove the existence of a continuous family of eigenfunctions, we define $M_1(\lambda) = M(\lambda, s_1(\lambda))$. If $\big(t_1(\lambda), t_2(\lambda) \big)^\top \in \ker M_1(\lambda)$ is nonzero, we know from \eqref{uform} that
		\[
		\bu_{s_1(\lambda)} = \big(I + A(\lambda,s_1(\lambda)) \big) \left( t_1(\lambda) \bu^{(1)}_{s_0} + t_2(\lambda) \bu^{(2)}_{s_0} \right)
		\]
		is an eigenfunction of $N_{s_1(\lambda)}$ for the eigenvalue $s_1^2(\lambda)\lambda$. We therefore need to understand the kernel of $M_1(\lambda)$.
		
		By construction we have $M_1(0) = 0$. Since $(\partial_\lambda M)(0,s_0) = 0$ and $\dot s_1(0) = 0$, we find that $\dot M_1(0) = 0$ and $\ddot M_1(0) = (\partial^2_\lambda M)(0,s_0) + (\partial_s M)(0,s_0) \ddot s_1(0)$.
		Using \eqref{concavities}, \eqref{matrix:Ms} and \eqref{MLL}, we get
		\begin{equation}
			\ddot M_1(0) = -2 s_0^4 \big ( \p_x v^{(2)}_{s_0}(\ell) \big )^2 \left( \frac{\big<\widehat v_{s_0}^{(1)}, u_{s_0}^{(1)}\big>}{\big ( \p_x u^{(1)}_{s_0}(\ell) \big )^2 } + \frac{\big<\widehat u_{s_0}^{(2)}, v_{s_0}^{(2)} \big>}{\big ( \p_x v^{(2)}_{s_0}(\ell) \big )^2 } \right) \begin{pmatrix} 0 & 0 \\ 0 & 1 \end{pmatrix},
		\end{equation}
		which is nonzero by \eqref{uglycondition}. Writing $M_1(\lambda) = \left(\begin{smallmatrix} a(\lambda) & b(\lambda) \\ c(\lambda) & d(\lambda) \end{smallmatrix}\right)$, it follows that $d(\lambda) \neq 0$ for small, nonzero values of $\lambda$, and so we can choose
		\[
		\begin{pmatrix} t_1(\lambda) \\ t_2(\lambda) \end{pmatrix} = \begin{pmatrix} 1 \\ - c(\lambda)/d(\lambda) \end{pmatrix} \in \ker M_1(\lambda)
		\]
		for $\lambda \neq 0$. Since $c(0) = \dot c(0) = \ddot c(0) = d(0) = \dot d(0) = 0$ but $\ddot d(0) \neq 0$, we get $c(\lambda)/d(\lambda) \to 0$ as $\lambda \to 0$, and so
		\[
		\lim_{\lambda \to 0}\big(I + A(\lambda,s_1(\lambda)) \big) \left( t_1(\lambda) \bu^{(1)}_{s_0} + t_2(\lambda) \bu^{(2)}_{s_0} \right) = \bu^{(1)}_{s_0}
		\]
		as claimed. The result for $\bu_{s_2(\lambda)}$ is proved in the same way.
	\end{proof}

	\begin{rem}\label{rem:concavities_signs}
		The condition \eqref{uglycondition} implies $\Delta(\lambda) > 0$ for small nonzero $\lambda$, and hence guarantees the existence of $s_\pm(\lambda)$. It also guarantees that $\ddot s_+(0) \neq \ddot s_-(0)$, as can be seen from \eqref{concavitiesexact}. If \eqref{uglycondition} fails then $\al=0$ and we cannot use the result of \cref{lem:Delta}. In this (nongeneric) case one may compute higher derivatives of $m$ in order to determine higher order coefficients in the Taylor expansion of $\Delta(\la)$, but we do not pursue this here.
	\end{rem}

	The following examples illustrate the two scenarios detailed in \cref{thm:touching}.
	
	\begin{ex}
		The conditions in case (1) of \cref{thm:touching} are satisfied if we take $L_+^{s} = L_-^{s}$, in which case $u^{(1)}_{s_0} = v^{(2)}_{s_0}$ at any crossing $(0,s_0)$, so that $\langle u^{(1)}_{s_0} , v^{(2)}_{s_0} \rangle\neq0$. Each isolated crossing $(\la,s)=(0,s_0)$ is a consequence of a pair of purely imaginary eigenvalues passing through the origin as $s$ increases. For clarity, in \cref{fig:isolated} we have plotted the \textit{imaginary} eigenvalue curves $s^2\la\in \spec(N_s)\cap i\R$ for the case when $L_-^s=L_+^s=-\p_{xx} - 4s^2$ and $\ell=12$ (here $\la\in\C$).
		\begin{figure}[]
			\hfill
			{\includegraphics[width=0.4\textwidth]{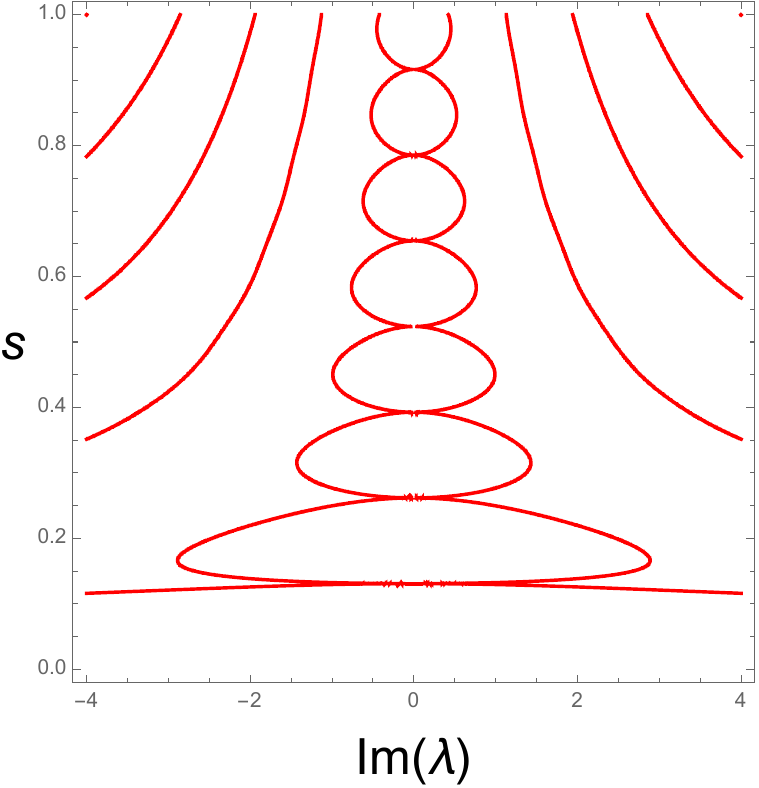}} \hfill\hfill
			\caption{Imaginary eigenvalue curves $s^2\la \in\spec(N_s)\cap i\R $, where $L_-^s=L_+^s=-\p_{xx} - 4s^2$ and $\ell=12$. Viewed from the $\eta s$-plane where $\eta=\text{Re}(\la)$, a series of isolated crossings appear at $\eta=0$ as $s$ increases from $0$ to $1$.}
			\label{fig:isolated}
		\end{figure}
	\end{ex}
	
	\begin{ex}
		Let $L=-\p_{xx} + V(x)$ with domain \eqref{domLpm}, and define $L_\pm = L  - \la_\pm$, where $\la_\pm \in \spec(L)$ are distinct eigenvalues with eigenfunctions $u_1$ and $v_2$, so that $L_+u_1 = L_-v_2 = 0$. Since $L_\pm$ is selfadjoint and $\la_+\neq \la_-$, we have $\langle u_1,v_2 \rangle=0$, and the conditions of case (2) in \cref{thm:touching} are satisfied. (Recall the notation of \eqref{notations0=1} when $s_0=1$.)
		
		The equations $L_+\widehat{u} _2 = v_2$ and $-L_-\widehat{v}_1  = u_1$
		are solved by $\widehat{u}_2  = \f{1}{\la_--\la_+}v_2 $ and $\widehat{v}_1  = \f{1}{\la_--\la_+}u_1$,
		and it follows that
		\[
		\int_{0}^{\ell} \widehat{u}_2\, v_2 \,dx =  \f{1}{\la_--\la_+}\int_{0}^{\ell}v_2^2 \,dx \quad \text{and}\quad  \int_{0}^{\ell} \widehat{v}_1\, u_1 dx = \f{1}{\la_--\la_+}\int_{0}^{\ell} u_1^2\, dx
		\]
		are nonzero and have the same sign. According to \eqref{concavities} this means the curves $s_{1,2}(\la)$ passing through $(0,1)$ will have opposite concavity. This is illustrated in \cref{fig:touching}, where we have plotted the real eigenvalue curves for a domain of length $\ell = 1$, choosing $L=-\p_{xx}$, $\la_+ = 9\pi^2$ and $\la_-=4\pi^2$. 
	\end{ex}

	\begin{figure}
		\hspace*{\fill}
		\subcaptionbox{ \label{} }
		{\includegraphics[width=0.35\textwidth]{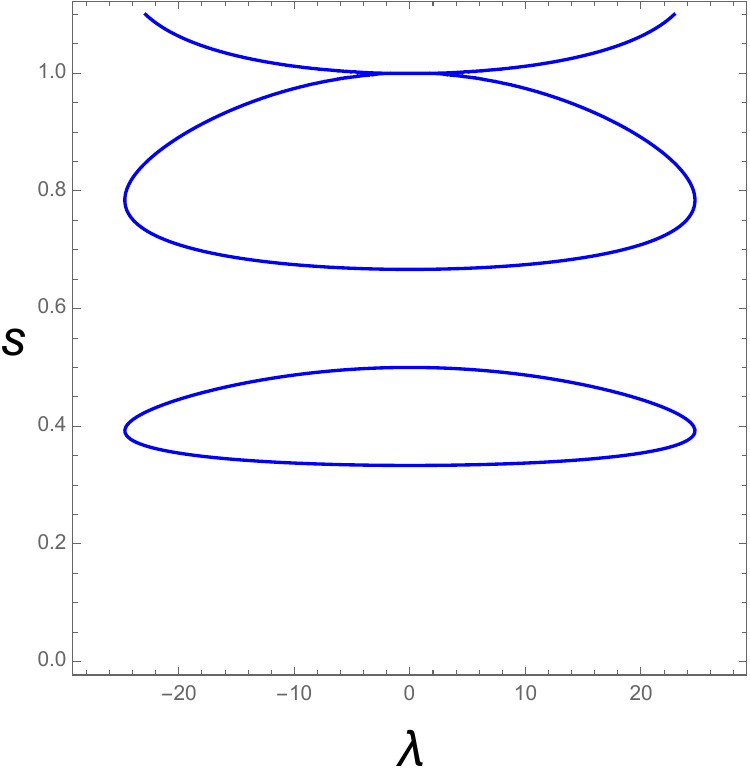}}\hfill 
		\subcaptionbox{\label{} }
		{\includegraphics[width=0.355\textwidth]{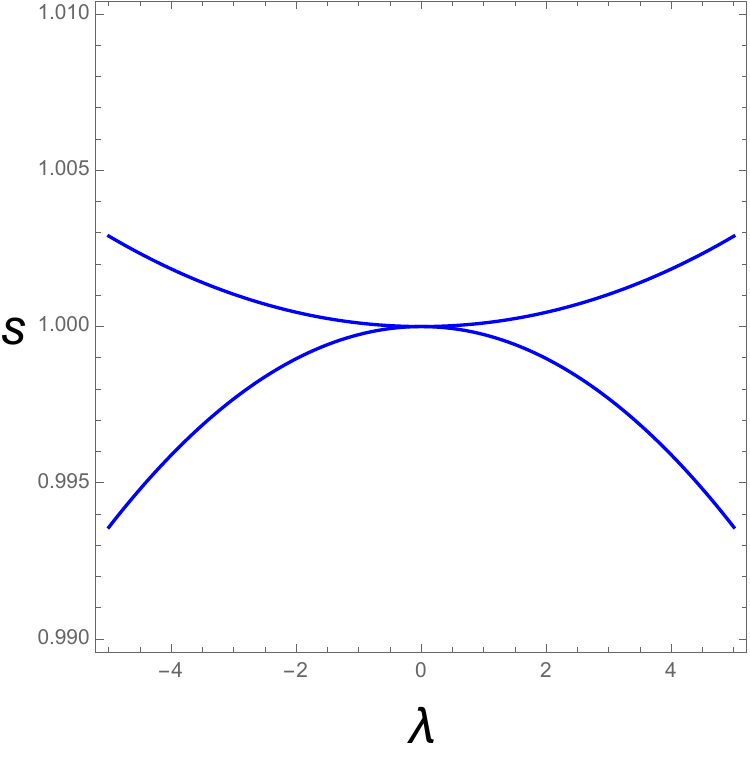}} 
		\hspace*{\fill}
		\caption{(a) Real eigenvalue curves $s^2\la \in\spec(N_s) \cap \R$ where $L_-^s = -\p_{xx} - 4\pi^2s^2$, $L_+^s = -\p_{xx} - 9\pi^2s^2$ and $\ell=1$, and (b) a blow-up of the conjugate point $(\la,s)=(0,1)$.}
		\label{fig:touching}
	\end{figure}

	\subsection{The Maslov index at the non-regular corner}\label{subsec:degenerate}
	
	We are now in a position to calculate the corner term $\mathfrak c$ appearing in \cref{thm:N_bound} (and defined in \cref{defn:c}) using the tools developed in \cref{subsec:Lyap_Schmidt,subsec:multiplicity2}. 
	
	Since a non-regular crossing occurs at the initial point of $\Gamma_3$, we cannot use \eqref{eq:maslov_defn} to compute the Maslov index. We therefore take advantage of homotopy invariance, deforming the corner of the Maslov box to a path that only has simple regular crossings.
	
	The index can then be deduced from the local behaviour of the eigenvalue curves through $(0,1)$ (see \cref{thm:touching,thm:sL}), which we quantify as follows. Given the curve $s(\lambda)$ from \cref{thm:sL}, there is an interval $(0,\hat\la)$ on which either $s(\la) > 1$ or $s(\la) < 1$, since the set $\{ \lambda : s(\lambda) = 1\}$ is discrete; cf. \cref{rem:Lisolated}. Therefore, the quantity
	\begin{equation}\label{s_sharp}
		s^\sharp(0) := \lim_{\lambda \to 0^+} \sgn \big( s(\lambda) - 1 \big) \in \{\pm 1\}
	\end{equation}
	is well-defined. In the case that $s=s(\la)$ is analytic, $s^\sharp(0) $ is the sign of the first nonzero Taylor coefficient at $\la=0$.
	
	\begin{rem}\label{rem:s_sharp}
		Recall from  \cref{thm:sL} that $\dot s(0)=0$. Therefore, in the generic case where $\ddot s(0) \neq 0$, we simply have
		\begin{equation}\label{ddot s0}
			s^\sharp(0) = \sgn \ddot s(0).
		\end{equation}
		That is, the VK-type integrals in \cref{thm:sL} determine $s^\sharp(0)$ (and hence the index $\mathfrak c$) provided the integrals are nonzero. However, it is important to note that the dichotomy $s^\sharp(0) = \pm 1$ holds even if $\ddot s(0) = 0$. 
	\end{rem}
	
	The same considerations apply to the curves $s_{1,2}(\lambda)$ from \cref{thm:touching} (for which $\dot s_{1,2}(0)=0$), so we define $s^\sharp_{1,2}(0)$ analogously, and emphasize that in the generic case $\ddot s_{1,2}(0)\neq 0$ we have
	\begin{equation}\label{ddot s120}
		s_{1,2}^\sharp(0) = \sgn \ddot s_{1,2}(0).
	\end{equation}
	
	With this notation in place, we are ready to calculate $\mathfrak c$.
	\begin{theorem}\label{thm:compute_c}
		The corner term $\mathfrak{c}$ from \cref{defn:c} is calculated as follows:
		\begin{itemize}
			\item[(1)] Suppose $\dim\ker(N)=1$, and let $s=s(\la)$ be the eigenvalue curve through $(0,1)$. 
			\begin{itemize}
				\vspace{2mm}
				\item[(i)] If $0\in\spec(L_+)\backslash \spec(L_-)$ then 
				 
				\[\mathfrak{c} = \f{1}{2} (s^\sharp (0)-1).\]
				That is, $\mathfrak{c} = 0$ if $ s^\sharp(0) = +1$ and $\mathfrak{c} = -1$ if $ s^\sharp(0) = -1$.
				\vspace{2mm}
				\item[(ii)] If $0\in\spec(L_-)\backslash \spec(L_+)$ then 
				 
				\[\mathfrak{c} = \f{1}{2} (1-s^\sharp (0)).\]
				That is, $\mathfrak{c} = 0$ if $ s^\sharp(0) = +1$ and $\mathfrak{c} = +1$ if $ s^\sharp(0) = -1$.
			\end{itemize}
			\vspace{2mm}
			\item[(2)] Suppose $\dim\ker(N)=2$, with $\ker(L_+) = \spn\{ u_1\}$ and $\ker(L_-) = \spn\{ v_2\}$. If $\langle u_1,v_2 \rangle \neq 0$, then $\mathfrak{c}=0$.  If $\langle u_1,v_2 \rangle = 0$ and the condition \eqref{uglycondition} holds, we denote by $s_{1,2}(\la)$ the eigenvalue curves passing through $(0,1)$, as in \cref{thm:touching}. Then 
			 
			\begin{equation}\label{cdim2}
				\mathfrak{c} = \f{1}{2}(s^\sharp_1(0) - s^\sharp_2(0)).
			\end{equation}
		\end{itemize}
	\end{theorem}
	 
	We remark that formula \eqref{cdim2} is simply the sum of the formulas for $\mathfrak{c}$ in cases (i) and (ii) of the simple case, identifying $s$ with $s_1$ if $0\in\spec(L_+)\backslash\spec(L_-)$ and $s$ with $s_2$ if $0\in\spec(L_-)\backslash\spec(L_+)$. It is perhaps interesting to note that in \eqref{cdim2} we have $\mathfrak{c}\in \{ -1,0,1 \}$, so that $\mathfrak{c}$ can never be $+2$ or $-2$, despite it being the contribution to the Maslov index from a two dimensional crossing in this case.

	\begin{proof}
		We use a homotopy argument, deforming the top left corner of the Maslov box as shown in \cref{fig:alternate_paths_up}.
		
		We first consider the case $\dim\ker(N) = 1$. If $s^\sharp(0) > 0$ then the deformed path does not intersect $\cD$, so we have $\mathfrak c = 0$. On the other hand, if $s^\sharp(0) < 0$, there will be a crossing at some point $(\lambda_*,s_*) = (\lambda_*, s(\lambda_*))$ with $0 < \lambda_* \ll 1$. This segment of the deformed path is parameterized by increasing $s$, so the relevant crossing form is
		\begin{equation}
			\label{mstar}
			\mathfrak m_{s_*}(q) = \frac{1}{s_*} \big<\big(\p_s B_{s_*} - 2 s_*\lambda_* \big) \bu_{s_*}, S \bu_{s_*} \big>,
		\end{equation}
		where $q = \Tr_{s_*}{\mathbf{u}_{s_*}}$. From \cref{thm:sL:general} we obtain a continuous family of eigenfunctions with $\bu_{s(\lambda)} \to \bu$ as $\lambda \to0$, so we can use \cref{lemma:scrossingform} to compute
		\[
		\lim_{\lambda\to0}\frac{1}{s(\lambda)} \big<\big(\p_s B_{s(\lambda)} - 2 s(\lambda)\lambda \big) \bu_{s(\lambda)}, S \bu_{s(\lambda)} \big> = \big<\p_s B_1 \bu_1, S \bu_1 \big>
		= {\ell} \left[ - \left (u_{1}'(\ell)\right )^2 +  \left (  v_{1}'(\ell)\right )^2\right].
		\]
		By continuity this has the same sign as the crossing form \eqref{mstar} at $(\lambda_*,s_*)$, so we conclude that $\mathfrak c = -1$ if $0 \in \spec(L_+)$ and $\mathfrak c = 1$ if $0 \in \spec(L_-)$.
		
		The argument for the case $\dim\ker(N) = 2$ is similar. Depending on the values of $s_1^\sharp(0)$ and $s_2^\sharp(0)$, there will be zero, one or two crossings that contribute to the index $\mathfrak c$. These are necessarily simple crossings, since $s_1(\lambda) \neq s_2(\lambda)$ for $\lambda \neq 0$ (see \cref{rem:concavities_signs}). Moreover, if either $s^\sharp_1(0)$ or $s^\sharp_2(0)$ is positive, it does not contribute to the index.
		
		Suppose $s^\sharp_1(0) < 0$, so there is a crossing at some point $(\lambda_*,s_*) = (\lambda_*, s_1(\lambda_*))$. As in the first case, we need to compute the crossing form
		\[
		\mathfrak m_{s_*}(q) = \frac{1}{s_*} \big<\big(\p_s B_{s_*} - 2 s_*\lambda_* \big) \bu_{s_*}, S \bu_{s_*} \big>.
		\]
		We use \cref{thm:touching} to get
		\[
		\lim_{\lambda\to0}\frac{1}{s_1(\lambda)} \big<\big(\p_s B_{s_1(\lambda)} - 2 s_1(\lambda)\lambda \big) \bu_{s_1(\lambda)}, S \bu_{s_1(\lambda)} \big> =  \big<\p_s B_1 \bu_1^{(1)}, S \bu_1^{(1)} \big> = -{\ell}\left ( \p_x u^{(1)}_{1}(\ell) \right )^2  < 0,
		\]
		and hence conclude that the crossing form at $(\lambda_*,s_*)$ is negative.
		Similarly, if $s^\sharp_2(0) < 0$, there is a crossing at some point $(\lambda_*, s_2(\lambda_*))$ whose crossing form is positive, because
		\[
		\lim_{\lambda\to0}\frac{1}{s_2(\lambda)} \big<\big(\p_s B_{s_2(\lambda)} - 2 s_2(\lambda)\lambda \big) \bu_{s_2(\lambda)}, S \bu_{s_2(\lambda)} \big> =  \big<\p_s B_1 \bu_1^{(2)}, S \bu_1^{(2)} \big> = {\ell}\left ( \p_x v^{(2)}_{1}(\ell) \right )^2> 0.
		\]
		In summary, the curve $s_1$ contributes $0$ to $\mathfrak c$ if $s^\sharp_1(0) > 0$ and $-1$ if $s^\sharp_1(0) < 0$, whereas $s_2$ contributes $0$ if $s^\sharp_2(0) > 0$ and $1$ if $s^\sharp_2(0) < 0$. Adding these contributions  completes the proof.
	\end{proof}
	
	\begin{figure}
		\hspace*{\fill} 
		\subcaptionbox{\label{fig:alternate_a}}{\includegraphics[width=0.25\textwidth]{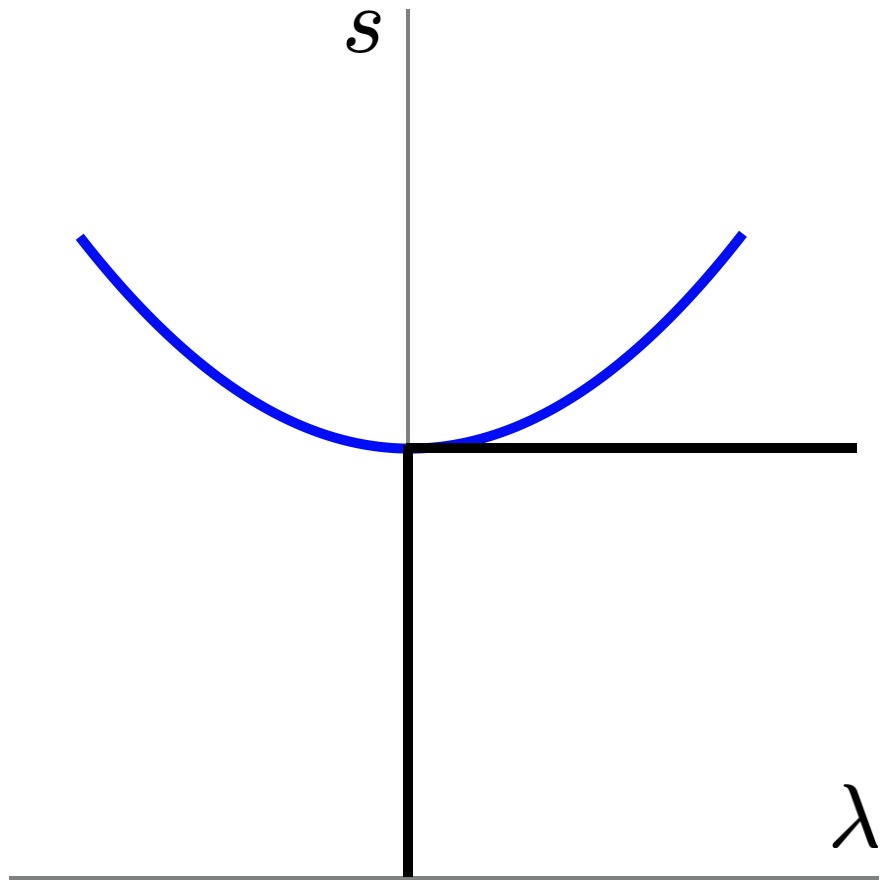}}
		\hspace*{2mm}
		\subcaptionbox{\label{fig:alternate_b}}{\includegraphics[width=0.25\textwidth]{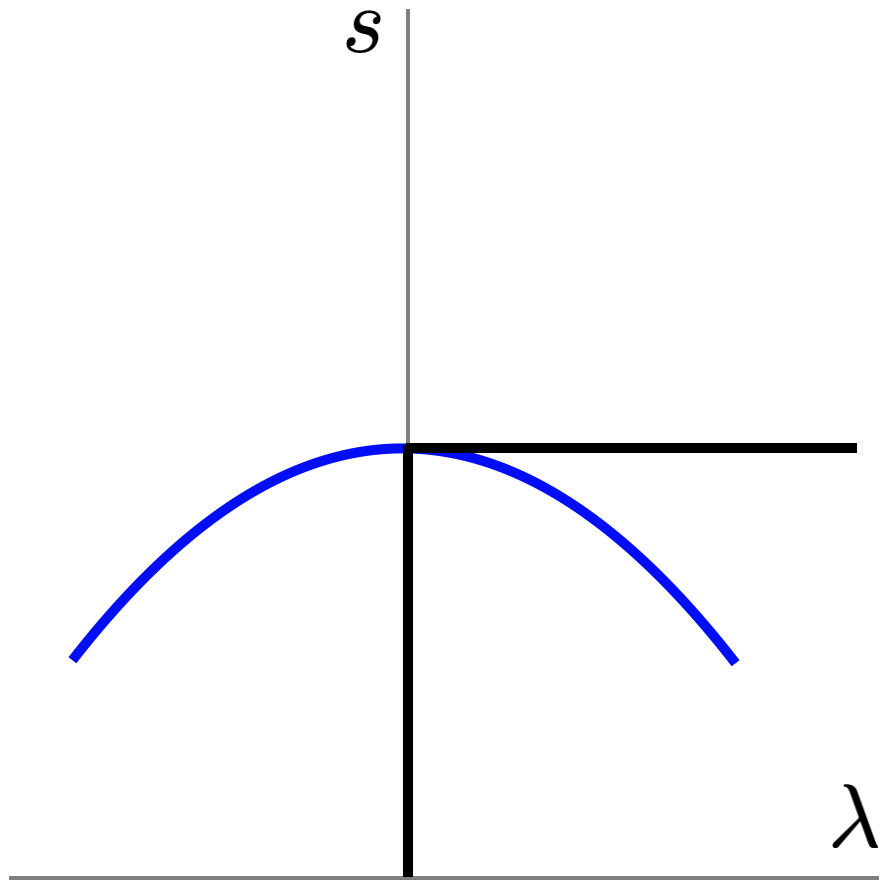}} \hspace*{2mm}
		\subcaptionbox{\label{fig:alternate_c}}{\includegraphics[width=0.25\textwidth]{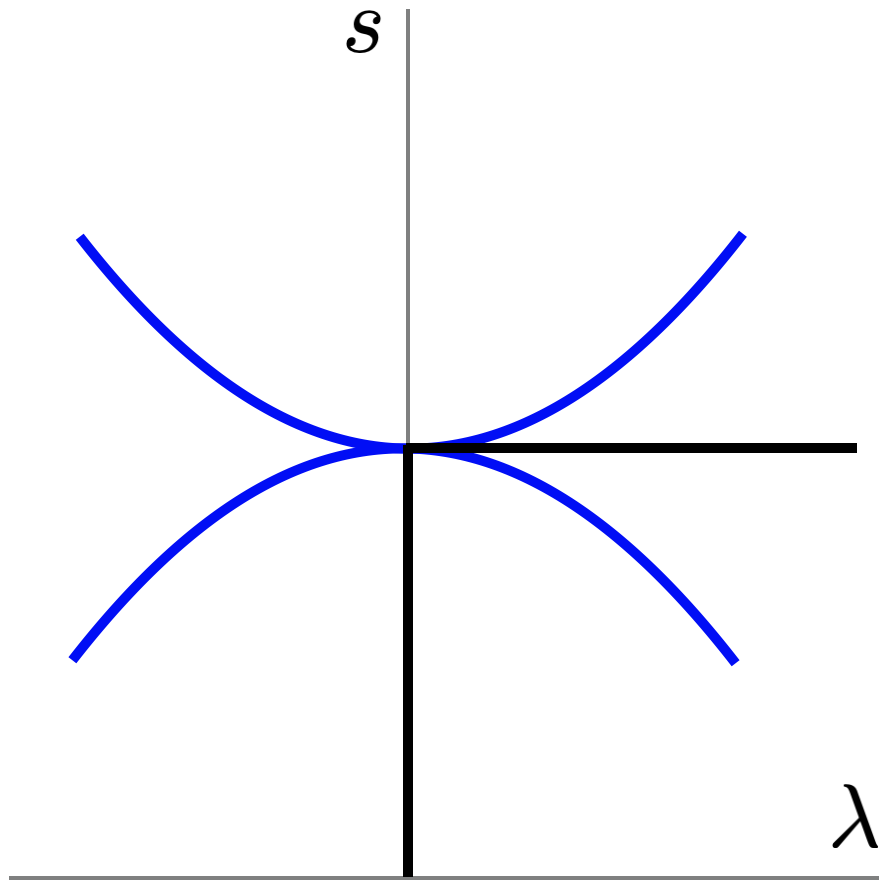}}
		\hspace*{\fill} \\[5mm]
		\hspace*{\fill}
		\subcaptionbox{\label{fig:alternate_d}}{\includegraphics[width=0.25\textwidth]{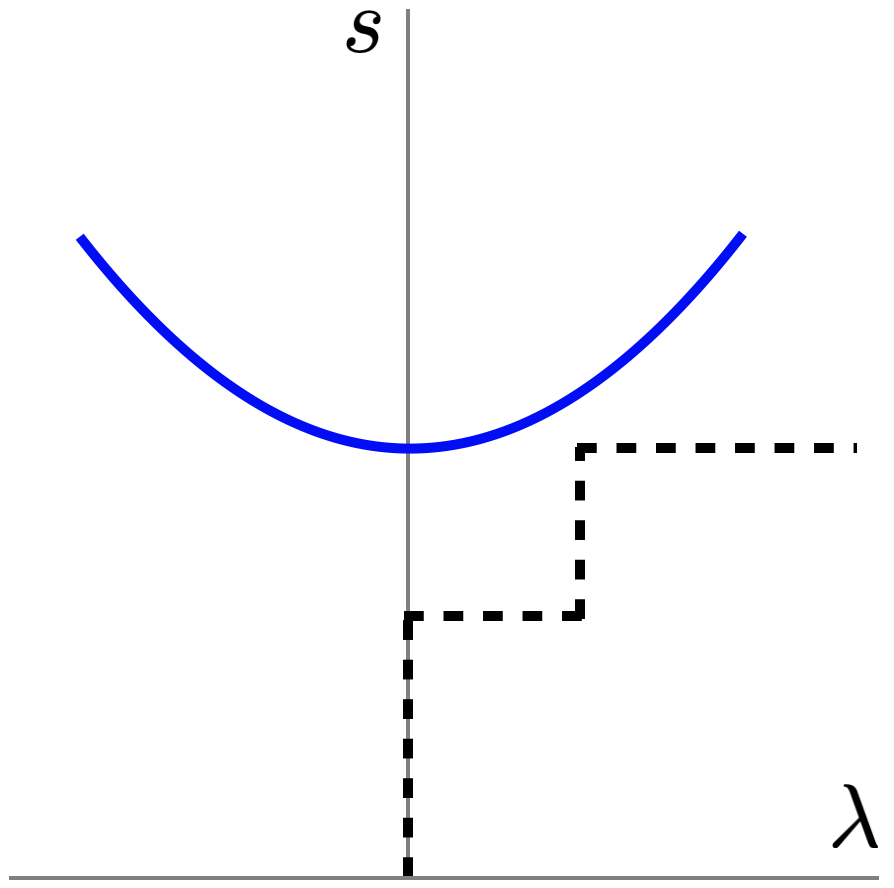}} 
		\hspace*{2mm}
		\subcaptionbox{\label{fig:alternate_e}}{\includegraphics[width=0.25\textwidth, trim= 0 0 0 0]{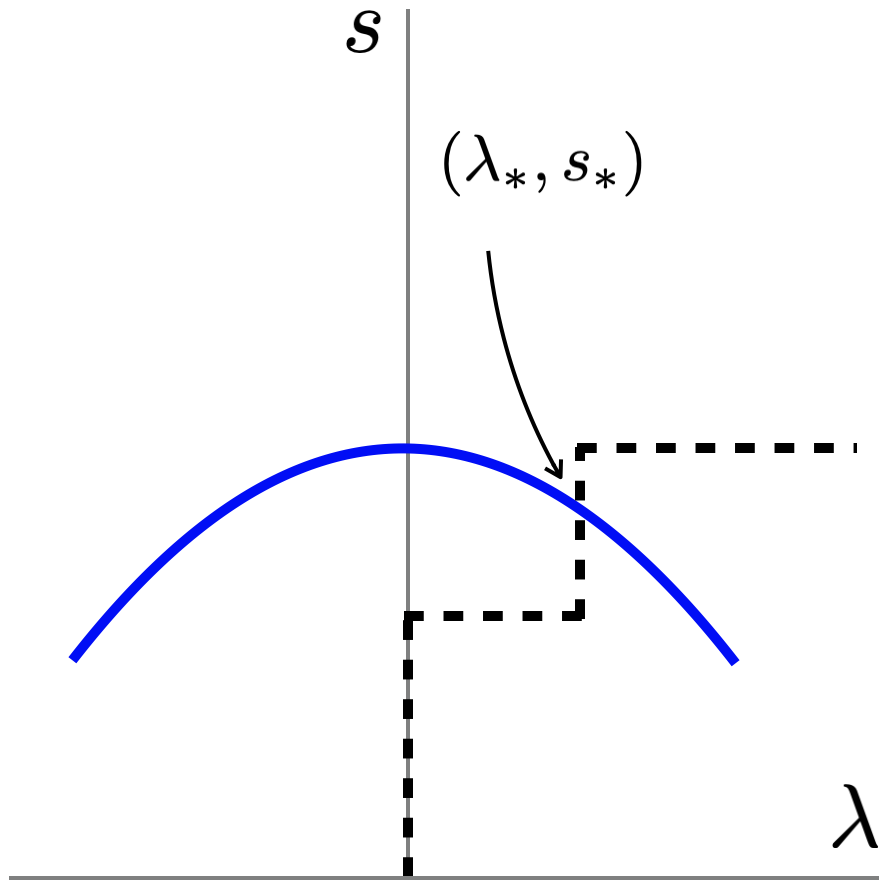}}
		\hspace*{2mm}	\subcaptionbox{\label{fig:alternate_f}}{\includegraphics[width=0.25\textwidth,trim= 0 0 0 0]{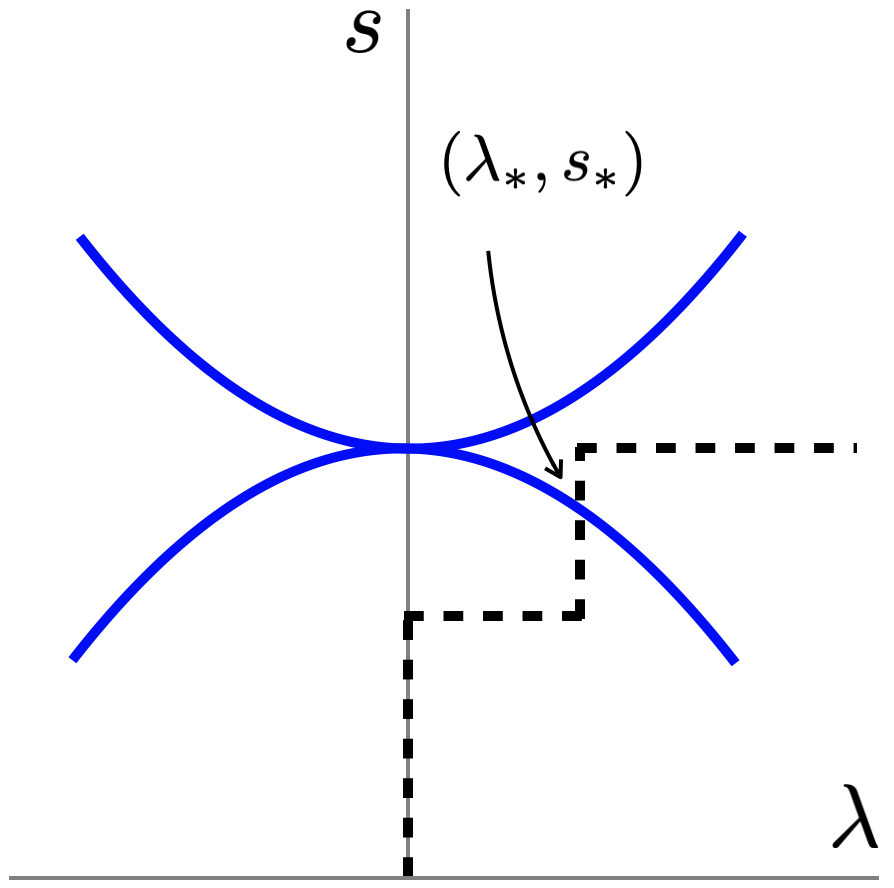}} 
		\hspace*{\fill}
		\caption{Neighbourhood of the crossing $(\la_0,s_0)=(0,1)$ featuring the eigenvalue curves (parabolas in blue) and the portion of the Maslov box passing through the corner $(0,1)$ (in black) when (a) $\dim\ker(N)=1$ and $s^\sharp(0)>0$,  (b) $\dim\ker(N)=1$ and $s^\sharp(0)<0$, and (c) $\dim\ker(N)=2$ and $s^\sharp_1(0) s^\sharp_2(0) < 0$. The path (dashed) to which we homotope the top left corner of the Maslov box in (a), (b) and (c) is given in (d), (e) and (f) respectively.}
		\label{fig:alternate_paths_up}
	\end{figure}
	
	We conclude this section by relating the concavity of the eigenvalue curves to the second order Maslov crossing form.
	
	\begin{prop}\label{prop:c}
		Assume the first order crossing form $\mathfrak{m}_{\la_0}$ is identically zero at the crossing $(\la_0,s_0)=(0,1)$. If  the second order crossing form $\mathfrak{m}_{\la_0}^{(2)}$ given in \cref{lemma:second_order_lambda_form} is nondegenerate, then
		\begin{equation}\label{b=sign_second_form}
			\Mas(\Lambda(\la,1),\cD;\la\in[0,\e]) = -n_-(\mathfrak{m}_{\la_0}^{(2)}).
		\end{equation}
	\end{prop}
	\begin{proof}
		We will prove this statement in the cases relevant to the current paper, that is, when $\dim\ker(N)=1,2$. Recall that nondegeneracy of $\mathfrak{m}_{\la_0}^{(2)}$ implies that $\ddot s(0)\neq 0$ if $\dim\ker(N)=1$ and $\ddot s_{1,2}(0)\neq 0$ if  $\dim\ker(N)=2$. Therefore, \eqref{ddot s0} and \eqref{ddot s120} hold.
		
		For the right hand side of \eqref{b=sign_second_form}, if $\dim\ker(N)=1$, \cref{thm:sL} shows that the sign of $\ddot s(0)$ determines the sign of the VK-type integrals in \eqref{s''L-} and \eqref{s''L+}, and therefore the sign of $\mathfrak{m}^{(2)}_{\la_0}$ given in \eqref{2ndorder_lambda_1dim}. In particular, we observe:
		\begin{itemize}
			\vspace{-4mm}
			\item[(i)] If  $0\in\spec(L_+)\backslash\spec(L_-)$ then
			$n_-(\mathfrak{m}_{\la_0}^{(2)}) =\begin{cases}
				0\qquad \ddot s(0)>0, \\
				1 \qquad \ddot s(0)<0.
			\end{cases}$
			\vspace{2mm}
			\item[(ii)] If $0\in\spec(L_-)\backslash\spec(L_+)$ then $n_-(\mathfrak{m}_{\la_0}^{(2)}) =\begin{cases}
				1\qquad \ddot s(0)>0, \\
				0 \qquad \ddot s(0)<0.
			\end{cases}$
			\vspace{-4mm}
		\end{itemize}
		If $\dim\ker(N)=2$, consider the matrix $\mathfrak{M}_{\la_0}^{(2)}$ of the second order form $\mathfrak{m}_{\la_0}^{(2)}$, which is given in \eqref{bigmfM2_2dim}. Using \eqref{concavities}, we see that:
		\begin{itemize}
			\item[(iii)] If  $0\in\spec(L_+)\cap\spec(L_-)$ then
			$n_-(\mathfrak{m}_{\la_0}^{(2)}) =\begin{cases}
				0 \qquad &\ddot s_1(0)>0, \ \ddot s_2(0) <0, \\
				1 \qquad  &\ddot s_1(0) \ddot s_2(0) >0, \\
				2 \qquad &\ddot s_1(0)<0, \ \ddot s_2(0) >0.
			\end{cases}$
		\end{itemize}
		For the left hand side of \eqref{b=sign_second_form}, let us define $\mathfrak{a}\coloneqq \Mas(\Lambda(s,0),\cD;s\in[1-\e,1])$ and  $\mathfrak{b}\coloneqq \Mas(\Lambda(\la,1),\cD;\la\in[0,\e])$, and notice from \eqref{eq:defn_c} that $\mathfrak{c}=\mathfrak{a}+\mathfrak{b}$. 
		From the proof of \cref{lemma:Maslov_left} we know that the crossing form at $(0,1)$ has $n_+(\mathfrak m_{s_0}) = \dim \ker(L_-)$, so \cref{defn:maslov} gives $\mathfrak{a}=\dim\ker(L_-)$. Therefore 
		\begin{equation}\label{bc_dimkerL}
			\mathfrak b = \mathfrak c - \dim\ker(L_-).
		\end{equation}
		Using the values of $\mathfrak{c}$ computed in in \cref{thm:compute_c}, we confirm that $\mathfrak{b}=-n_-(\mathfrak{m}_{\la_0}^{(2)})$ in cases (i), (ii) and (iii) described above, as claimed.
	\end{proof}

	\section{Applications}\label{sec:applications}
	
	In this section we give some applications of the theory of \cref{sec:symplectic_view,sec:proof1}. We begin with the proof of  \cref{cor:JonesGrillakis,cor:exact_count_Q=0,thm:VK}, which are consequences of \cref{thm:N_bound} and \cref{thm:compute_c}. We then give formulas for the concavity of the NLS spectral curves, and recover the classical VK criterion for a particular one-parameter family of stationary states. Finally, we relate our results to the Krein index theory.

	\subsection{The Jones--Grillakis instability theorem}\label{subsec:51}
	
	We first prove the compact interval analogue of the Jones--Grillakis instability theorem,  \cref{cor:JonesGrillakis}, and its consequence \cref{cor:exact_count_Q=0}.
	
	\begin{proof}[Proof of \cref{cor:JonesGrillakis}]
		From \cref{thm:N_bound} we have $n_+(N) \geq 1$ provided $P-Q \neq \mathfrak c$. The result now follows from \cref{thm:compute_c}, which guarantees $\mathfrak c \in \{-1,0\}$ when $0 \in \spec(L_+) \setminus \spec(L_-)$, and $\mathfrak c \in \{0,1\}$ when $0 \in \spec(L_-) \setminus \spec(L_+)$.
	\end{proof}
	
	\begin{proof}[Proof of \cref{cor:exact_count_Q=0}]
		We claim that $Q = 0$,  $P \geq 1$ and $0 \in \spec(L_+) \setminus \spec(L_-)$ {  under the assumptions of the Corollary}. Once this has been shown, the result follows immediately from \cref{cor:JonesGrillakis}.
		
		Since $\phi$ is nonconstant and satisfies Neumann boundary conditions, we have $0 \in \spec(L_+)$, with eigenfunction $\phi'$.  Moreover, each stationary point of $\phi$ in the interior of its domain corresponds to a conjugate point for $L_+$: If $\phi'(x_0) = 0$ for some $x_0 \in (0,\ell)$, then $0 \in \spec(L_+^{s_0})$ for $s_0 = x_0/\ell$, with eigenfunction $\phi(s_0 x)$. It then follows from \cref{lemma:morse_schrodinger} that $P \geq 1$.
		
		We next consider $L^s_-$ for $s\in(0,1]$. Under \cref{hypo:NLS}, the general solution to the differential equation $L_-^{s}w=0$ is
		\begin{align}
			w(x) = c_1\phi(sx) + c_2 \phi(sx) \int_{0}^{x} \f{1}{\phi(st)^2} \,dt,
		\end{align}
		where the second fundamental solution was obtained via the method of reduction of order, and is well defined since $\phi(x) \neq 0 $ for all $x\in[0,\ell]$ implies $1/\phi^2$ is integrable. It follows that 
		\begin{equation}\label{}
			\phi(sx) \int_{0}^{x} \f{1}{\phi(st)^2} \,dt \geq 0 
		\end{equation}
		for all $x\in[0,\ell]$, with equality when $x=0$. Dirichlet boundary conditions on $w$ then dictate that $c_1=c_2=0$, and we conclude that $0\notin\spec(L_-^{s})$ for all ${s}\in(0,1]$. In particular, $0 \notin \spec(L_-)$, and \cref{lemma:morse_schrodinger} implies $Q=0$.
	\end{proof}

	\subsection{VK-type (in)stability criteria}\label{subsec:52}
	
	For the proof \cref{thm:VK} we will need two preliminary results. The first of these mimics \cite[Corollary 1.1]{Grill88}, and follows from the equivalent selfadjoint formulation of the eigenvalue problem \eqref{eq:VKsystem}; see \cref{lemma:equiv_problems}.
	
	\begin{lemma}\label{lemma:imag_eigvals}
		If $Q=0$ or $P=0$ then $\spec(N_s) \subset \R\cup i\R$ for all $s\in(0,1]$.
	\end{lemma}
	\begin{proof}
		Fix $s\in(0,1]$. If $Q=0$ then $L_-^s$ is nonnegative by \cref{lemma:morse_schrodinger_nonnegative}. By \cref{lemma:equiv_problems} the eigenvalue problem \eqref{eq:VKsystem} is equivalent to \eqref{eq:VKsystem3}. The operator $\left ( L_-^s|_{X_c}\right )^{1/2} \Pi L_+^s \Pi\left ( L_-^s|_{X_c}\right )^{1/2}$ acting in $X_c$ is selfadjoint, and therefore $s^4\la^2 \in \R$. Then $s\in\R$ implies $\la\in\R \cup i\R$. The case $P=0$ follows similarly.
	\end{proof}
	
	We next prove that the Maslov index is monotone in $\la$ if either $Q=0$ or $P=0$.
	
	\begin{lemma}
		\label{lemma:sign_definite}
		If $Q=0$ then the crossing form $\mathfrak{m}_{\la_0}$ is strictly positive for any crossing with $\la_0>0$ and $s_0=1$, while if $P=0$ then $\mathfrak{m}_{\la_0}$ is strictly negative at all such crossings. Consequently,
		\begin{equation}\label{P=0_Q=0}
			n_+(N) = \begin{cases}
				\Mas(\Lambda,\cD;\Gamma^\e_3)\qquad &\text{if}\,\,\, Q=0, \\
				- \Mas(\Lambda,\cD;\Gamma^\e_3)\qquad  &\text{if}\,\,\, P=0.
			\end{cases}
		\end{equation}
		(Recall that $\Mas(\Lambda,\cD;\Gamma^\e_3)=\Mas(\Lambda(\la,1),\cD;\la\in[\e, \la_\infty]) $.)
	\end{lemma}
	\begin{proof}
		Assume $\la_0>0$ with eigenfunction $\mathbf{u}_1=(u_1,v_1)^\top$, so that \eqref{eq:VKsystem} holds with $\la=\la_0$ and $s=1$. Note that both $u_1$ and $v_1$ are necessarily nontrivial due to the coupling of the eigenvalue equations for $\la\neq0$.  If $Q=0$, we apply $\langle\cdot, v_1\rangle$ to the first equation of \eqref{eq:VKsystem} to obtain
		\begin{align}\label{L-_quad_form}
			\langle L_- v_1, v_1 \rangle =  -\la_0 \langle u_1, v_1 \rangle = \f{\la_0}{2} \mathfrak{m}_{\la_0}(q), \quad q=\Tr \mathbf{u}_1,
		\end{align}
		using formula \eqref{maslov_horizontal}. Now $0\neq u_1\in\ran(L_-)$ implies $v_1$ has a component lying in $\ker(L_-)^\perp$. Since $Q=0$, it follows that  $\langle L_- v_1, v_1 \rangle > 0$. Thus $\mathfrak{m}_{\la_0}(q)>0$ at all crossings along $\Gamma^\e_3$ if $Q=0$. If $P=0$, one applies $\langle \cdot, u_1\rangle$ to the second equation of \eqref{eq:VKsystem} at $(\la_0,1)$, and a similar argument yields that $\langle L_+ u_1, u_1 \rangle = -\f{\la_0}{2} \mathfrak{m}_{\la_0}(q)> 0$. Thus $\mathfrak{m}_{\la_0}(q)<0$ at all crossings on $\Gamma^\e_3$ if $P=0$. 
	\end{proof}

	\begin{proof}[Proof of \cref{thm:VK}]
		Consider the eigenvalue curve $s=s(\la)$ through the point $(\la,s)=(0,1)$, for which $\dot s(0)=0$ as stated in \cref{thm:sL}. 
		
		We start with the case $P=1, Q=0$ and $0\in\spec(L_-)\backslash \spec(L_+)$. If $\ddot s(0)>0$, then by \cref{thm:compute_c} we have $\mathfrak{c}=0$. {  Since $Q=0$, by \cref{lemma:sign_definite} and \eqref{mas_top_fmla} we have $n_+(N) = P - \mathfrak{c} = 1$}. On the other hand, if $\ddot s(0)<0$, then by \cref{thm:compute_c} we have $\mathfrak{c}=1$, and {  by the same argument $n_+(N)=P-\mathfrak{c} =0$. It then follows from \cref{lemma:imag_eigvals} that $\spec(N)\subset i\R$.}
		
		The case where $Q=1, P=0$ and $0\in\spec(L_+)\backslash \spec(L_-)$ is similar. If $\ddot s(0)>0$, then $\mathfrak{c}=0$ by \cref{thm:compute_c}, and { \cref{lemma:sign_definite} and \eqref{mas_top_fmla} imply $n_+(N)=Q+\mathfrak{c} = 1$}. If $\ddot s(0)<0$, then $\mathfrak{c}=-1$ by \cref{thm:compute_c}, {  hence $n_+(N)=0$. By  \cref{lemma:imag_eigvals} we deduce that $\spec(N)\subset i\R$.}
	\end{proof}
	
	\subsection{Concavity computations for NLS}
	
	Working under \cref{hypo:NLS}, in this subsection we compute the sign of $\ddot s(0)$ via the VK-type integrals given in \cref{thm:sL}. In what follows, $s(\la)$ is the eigenvalue curve through $(\la_0,s_0)=(0,1)$.
	
	\subsubsection{The $L_+$ integral}
	We first consider the case when $L_+$ has a nontrivial kernel. The following result allows us to compute $\ddot s(0)$ when $\phi$ satisfies Neumann boundary conditions.
	
	\begin{prop}\label{prop:concave_up}
		Assume \cref{hypo:NLS} and that $0\in\spec(L_+)\backslash$ $\spec(L_-)$ with eigenfunction $\phi'$. If $\{p,q\}$ is a fundamental set of solutions to the differential equation $L_-v=0$ initialised at the identity, then $q(\ell)\neq0$ and 
		\begin{align}\label{eq:convexity_s_L+}
			\sgn \ddot{s}(0) = \sgn \left [ \left( \int_{0}^{\ell} p^2 dx\right) - \f{ p(\ell)}{q(\ell)} \ell^2 \right ].
		\end{align}
	\end{prop}
	\begin{proof}
		First, note that $\ker(N) = \spn \{(\phi' ,0)^\top\}$. Now by case (2) of \cref{thm:sL} we have
		\[
		\sgn \ddot{s}(0) = \sgn \int_0^{\ell}  \widehat{v} \,\phi' \,dx
		\]
		where $\widehat{v}$ is the unique solution to the inhomogeneous boundary value problem
		\begin{align}\label{eq:inhomog_eqn}
			L_-\widehat{v} = \phi', \qquad \widehat{v}(0)=\widehat{v}(\ell)=0.
		\end{align} 
		Let $\{p, q\}$ be a fundamental set of solutions to the homogeneous equation $L_{-}\widehat{v}= 0$ such that 
		\begin{equation}\label{initial_pq}
			\begin{pmatrix}
				p(0) & q(0) \\ p'(0) & q'(0)
			\end{pmatrix} = \begin{pmatrix}
				1 & 0 \\ 0 & 1
			\end{pmatrix}.
		\end{equation}
		Since $\phi(0) \neq 0$, the first solution is given by $p(x) = \phi(x)/\phi(0)$. We have $p'(\ell) =0, \,p(\ell)\neq 0$, while $q(\ell)\neq 0$ since $q(0)=0$ and $0\notin \spec(L_-)$. By Abel's identity,
		\begin{align}\label{eq:abels}
			p(x)q'(x) - q(x) p'(x) =1 \quad \forall \,\,x\in[0,\ell].
		\end{align}
		The general solution to the differential equation $L_-\widehat{v}= \phi'$ is thus
		\begin{equation}\label{xphi2}
			\widehat{v}(x) = Ap(x) + Bq(x) - \f{x\phi(x)}{2},
		\end{equation}
		where it is easily verified that $-x\phi(x)/2$ is a particular solution. Imposing the boundary conditions on $\widehat{v}$ to determine the constants $A$ and $B$, we find that the unique solution to \eqref{eq:inhomog_eqn} is 
		\[
		\widehat{v}(x) =   \f{1}{2} \left ( \f{\ell\phi(\ell)}{q(\ell)} q(x) -  x\phi(x)\right ).
		\]
		It remains to compute $\sgn\int_{0}^{\ell}\widehat{v}\phi'dx$.  Since $\phi(x) =p(x) \phi(0)$, we have 
		\begin{align*}
			\int_{0}^{\ell}\widehat{v}(x)\phi'(x) dx & = \int_{0}^{\ell} \f{1}{2} \left ( \f{\ell\phi(\ell)}{q(\ell)} q(x) -  x\phi(x)\right ) p'(x)\phi(0)\,  dx \\
			&= \f{\phi(0) ^2 \ell p(\ell)}{2q(\ell)}\int_0^\ell q(x)p'(x)dx - \f{\phi(0)^2}{2} \int_0^\ell xp(x)p'(x) dx.
		\end{align*}
		For the second integral we obtain 
		\[
		\int_{0}^{\ell} x p(x) p'(x) dx = \frac{1}{2}\left(\ell p(\ell)^2- \int_{0}^{\ell}p(x)^{2} dx \right),
		\]
		while for the first we integrate by parts and appeal to \eqref{eq:abels} to arrive at 
		\[
		\int_0^\ell q(x)p'(x)dx =\f{1}{2}\left (q(\ell) p(\ell) -  \ell\right ).
		\]
		Therefore
		\begin{align*}
			\int_{0}^{\ell}\widehat{v}(x)\phi'(x) dx & = \f{\phi(0)^2  \ell p(\ell)}{4q(\ell)} \left (q(\ell) p(\ell) -  \ell\right ) -   \f{\phi(0)^2}{4}\left(\ell p(\ell)^2- \int_{0}^{\ell}p(x)^{2} dx \right) \\
			&= \f{\phi(0)^2}{4}\left (  \int_{0}^{\ell}p(x)^{2} dx  -\f{ p(\ell)}{q(\ell)} \ell^2  \right )
		\end{align*}
		and \eqref{eq:convexity_s_L+} follows. 
	\end{proof}
	
	\begin{rem}
		If $\phi$ is nonvanishing, the second solution $q$ can be determined using reduction of order; see \eqref{qsoln} and also the proof of \cref{cor:exact_count_Q=0}. When $\phi$ has zeros the second solution is given by the Rofe--Beketov formula \cite[Lemma 2]{Schmidt}; however, the resulting expression is significantly more complicated and does not appear to be useful for our analysis.
	\end{rem}
	
	The following result serves as an application of \cref{prop:concave_up} in the case when the stationary state is either strictly positive or strictly negative over its domain.
	\begin{cor}\label{cor:concaveup}
		Under the assumptions of \cref{prop:concave_up}, for nonconstant solutions to \eqref{standwave} satisfying $\phi(x)\neq0$ for all $x\in[0,\ell]$,  we have $\ddot{s}(0)>0$.
	\end{cor}
	
	\begin{proof}
		In the case when $\phi$ has no zeros on the interval $[0,\ell]$, the method of reduction of order allows us to write
		\begin{equation}\label{qsoln}
			q(x) = p(x) \int_{0}^{x}\frac{1}{p(t)^{2}} dt,
		\end{equation}
		where the nonvanishing of $p$ ensures $1/p^2$ is integrable. This gives
		\begin{align*}
			\int_{0}^{\ell}p(x)^{2} dx  -\f{ p(\ell)}{q(\ell)} \ell^2
			&= \f{\left( \int_{0}^{\ell} \frac{1}{p^{2}} dx\right)\left( \int_{0}^{\ell} p^{2} dx\right) - \ell^2}{\left( \int_{0}^{\ell} \frac{1}{p^{2}} dx\right)},
		\end{align*}
		and so
		\begin{equation}\label{eq:sign_intermediate}
			\sgn  \ddot s(0)= \sgn \left [\left( \int_{0}^{\ell} \frac{1}{p^{2}} dx\right)\left( \int_{0}^{\ell} p^{2} dx\right) - \ell^2 \right ].
		\end{equation}
		By virtue of the Cauchy Schwarz inequality,
		\begin{equation*}
			\ell =  \int_0^\ell p(x) \f{1}{p(x)} \,dx 
			\leq \sqrt{\int_0^\ell p^2(x)  dx }\,\sqrt{\int_0^\ell \f{1}{p(x)^2} dx}
		\end{equation*}
		where we have equality only when $p$ and $1/p$ are linearly dependent, that is, when $\phi$ is constant. Since we have assumed a nonconstant solution, the inequality is strict, and we conclude that \eqref{eq:sign_intermediate} is positive.
	\end{proof}
	
	\begin{rem}
		The statement of \cref{cor:concaveup} may also be proven using \cref{rem:positivity}, since $L_->0$ for stationary states that are nonvanishing over $[0,\ell]$ (as was shown in the proof of \cref{cor:exact_count_Q=0}). However, the proof given above is a nice illustration of \cref{prop:concave_up}, a more general result that holds for any nonconstant $\phi$. 
	\end{rem}

	\subsubsection{The $L_-$ integral: Recovering classical VK}
	\label{sec:VKrecover} 
	We now consider the case when $L_-$ has a nontrivial kernel (spanned by $\phi$). We show that the associated VK-type integral in equation \eqref{s''L-} of \cref{thm:sL} recovers a compact interval analogue of the classical VK integral expression
	\begin{equation}\label{VK_classical}
		\pde{}{\beta} \int_{-\infty}^{\infty} \phi^2 \,dx
	\end{equation}
	associated with a stationary state $\phi\in L^2(\R)$ solving \eqref{standwave} (see \cite[Theorem 4.4, p.215]{pelinovsky}). The key observation is that $\p_\beta \phi(\cdot;\be)$ solves the differential equation $L_+\widehat{u} = \phi$ associated with case (1) of \cref{thm:sL}, and this naturally leads to the expressions \eqref{VK_integral_classical} and \eqref{eqn:VK}, which clearly resemble \eqref{VK_classical}. This is not true for the equation $L_-\widehat{v} = \phi'$ associated with case (2) of \cref{thm:sL}, for which a recovery of a compact interval analogue of \eqref{VK_classical} is thus not possible. In what follows, $\phi'(x;\be)$ refers to $\de{\phi}{x} (x;\be)$, while the $\be$ derivative will be denoted by $\p_\be$. 
	
	\begin{prop}\label{prop:VK_compact}
		Assume \cref{hypo:NLS} and let $\phi_0$ be a solution to \eqref{standwave} with parameter $\be_0$ that satisfies $\phi_0(0) = \phi_0(\ell)=0$. There exists a unique one-parameter family of solutions $\beta \mapsto \hat\phi(\cdot;\beta)$ to \eqref{standwave}, defined in a neighbourhood of $\be_0$, such that
		\begin{equation}
			\label{BC}
			\hat\phi(0;\beta) = \hat\phi(\ell;\beta) = 0
		\end{equation}
		for all $\be$ near $\be_0$ and $\hat\phi(\cdot;\beta_0) = \phi_0$. In terms of this family, the VK-type integral in \eqref{s''L-} is
		\begin{align}
			\label{VK_integral_classical}
			\int_0^\ell \widehat{u}\,v \,dx  = \f{1}{2} \pde{}{\beta}\bigg|_{\be=\be_0} \int_{0}^{\ell} \hat\phi(x;\beta) ^2 \,dx.
		\end{align}
		More generally, if $\beta \mapsto \phi(\cdot;\beta)$ is any $C^1$ family of solutions to \eqref{standwave} satisfying $\phi(\cdot;\beta_0) = \phi_0$, then the integral in \eqref{s''L-} can be written
		\begin{align}\label{eqn:VK}
			\begin{split}
				&\int_{0}^{\ell} \widehat{u}\, v \,dx = \f{1}{2}\pde{}{\beta}\bigg|_{\be=\be_0} \int_{0}^{\ell} \phi(x;\beta) ^2 \,dx \\ &\,\,\,\,\,\,+\big ((-1)^{Q}\p_{\beta}\phi(0;\beta_0)+ \p_{\beta}\phi(\ell;\beta_0)\big ) \left (\f{\p_{\beta}\phi(0;\beta_0) + (-1)^{Q}\p_{\beta}\phi(\ell;\beta_0)}{q(\ell)} + \p_{\beta}\phi'(\ell;\beta_0)  \right ) .
			\end{split}
		\end{align}
		Furthermore, if $P=1$, $Q=0$ and \eqref{VK_integral_classical} or \eqref{eqn:VK} is positive (resp. negative), then the standing wave $\widehat{\psi}(x,t) = e^{i\be_0 t} \phi_0(x)$ is spectrally unstable (resp. spectrally stable).
	\end{prop}
	
	\begin{proof}
		The existence of $\phi_0$ implies that the associated operators
		\begin{equation*}
			\begin{aligned}
				L_-&=-\partial_{xx}-f(\phi_0^2)-\be_0,  \\
				L_+&=-\partial_{xx}-2f'(\phi_0^2)\phi_0^2-f(\phi_0^2)-\be_0
			\end{aligned}
		\end{equation*}
		have $\phi_0\in\ker(L_-)$ and hence $0\in\spec(L_-)\backslash\spec(L_+)$.
		Consider the function
		\begin{equation}
			F \colon \big(H^2(0,\ell) \cap H^1_0(0,\ell)\big) \times \R \longrightarrow L^2(0,\ell), \qquad
			F(\phi,\beta) = \phi'' + f(\phi^2)\phi + \beta \phi,
		\end{equation}
		in terms of which \eqref{standwave} and \eqref{BC} become $F(\phi,\beta) = 0$. It can be shown that $F$ is continuously Fr\'echet differentiable (see \cite[\S 2.2]{coleman}), with 
		\begin{equation}
			\label{DF}
			DF(\phi_0,\beta_0)(u,\gamma) = \gamma \phi_0 - L_+u.
		\end{equation}
		Since $0\notin \spec(L_+)$, this implies $DF(\phi_0,\beta_0)(\cdot,0) = - L_+$ is invertible, so the implicit function theorem guarantees the existence of a $C^1$ function
		\begin{equation}\label{beta_fam}
			(\beta_0 - \epsilon, \beta_0 + \epsilon) \to H^2(0,\ell) \cap H^1_0(0,\ell),\quad  \beta \mapsto \hat\phi(\cdot;\beta),
		\end{equation}
		such that $F(\hat\phi(\cdot;\beta),\beta) = 0$ for all $|\beta - \beta_0| < \epsilon$.
		
		Turning to the integral in \eqref{s''L-}, where now  $v=\phi_0$, we need to solve
		\begin{equation}\label{inhomog_eq}
			L_+\widehat{u} = \phi_0, \qquad \widehat{u}(0) = \widehat{u}(\ell)=0.
		\end{equation}
		Using the family constructed above, which is $C^1$ in $\be$, we differentiate \eqref{standwave} with respect to $\beta$ and evaluate at $\be_0$ to obtain
		\begin{equation}
			\label{Lplus_inhomog}
			L_+ \p_{\beta}\hat\phi(x;\beta_0)= \phi_0(x).
		\end{equation}
		Now differentiating \eqref{BC} (which holds for \textit{all} $\be$ near $\be_0$) with respect to $\be$  and evaluating at $\be_0$ yields 
		\begin{equation}\label{}
			\p_{\beta} \hat\phi(0;\beta_0) = \p_{\beta} \hat\phi(\ell;\beta_0) = 0.
		\end{equation}
		Therefore, $\widehat{u}(x)=\p_{\beta}\hat\phi(x;\beta_0)$ is the \textit{unique} solution to \eqref{inhomog_eq}, and substituting this into the VK-type integral in \eqref{s''L-} with $v=\phi_0$ yields \eqref{VK_integral_classical}.
		
		Now let $\beta \mapsto \phi(\cdot;\beta)$ be an arbitrary family of solutions to \eqref{standwave} (again for $\beta$ close to $\beta_0$) such that $\phi(x;\beta_0)= \phi_0(x)$. To solve \eqref{inhomog_eq}, note that \eqref{Lplus_inhomog} still holds for the family $\phi(\cdot;\be_0)$, and thus the general solution to $L_+ \widehat{u} = \phi_0$ is
		\begin{equation}\label{gensoln}
			\widehat{u}(x) = Ap(x) + Bq(x) + \p_{\beta}\phi(x;\beta_0),
		\end{equation}
		where $\{p,q\}$ is now a fundamental set of solutions to the homogeneous equation $L_+\widehat u=0$ satisfying \eqref{initial_pq}. Since $\phi'(0;\be_0)\neq0$, we may set $p(x) = \phi'(x;\be_0) / \phi'(0;\be_0)$. A brief look at the  Hamiltonian for \eqref{standwave} indicates that intersections of any fixed orbit with $\phi=0$ are symmetric about $\phi'=0$; from this, along with Sturm-Liouville theory applied to $\phi(\cdot;\be_0) = \phi_0\in\ker(L_-)$, we deduce that we necessarily have $\phi'(\ell;\be_0) = (-1)^{Q+1} \phi'(0;\be_0)$, and therefore that $p(\ell)=(-1)^{Q+1}$. Evaluating \eqref{standwave} at $x=\ell$ we also find that $\phi''(\ell;\be_0)=0$, hence $p'(\ell) = 0$. Thus
		\begin{align}
			\label{BCs_at_ell}
			\begin{pmatrix}
				p(\ell) & q(\ell) \\ p'(\ell) & q'(\ell)
			\end{pmatrix} = \begin{pmatrix}
				(-1)^{Q+1} & * \vspace{0.7mm}\\ 0 & (-1)^{Q+1} 
			\end{pmatrix}
		\end{align}
		where $q'(\ell)=(-1)^{Q+1} $ because \eqref{BCs_at_ell} must have unit determinant by virtue of Abel's identity (see \eqref{eq:abels}). In addition, $q(\ell)\neq0$ since $0\notin\spec(L_+)$ and $q(0)=0$.
		
		Imposing the boundary conditions $\widehat{u}(0)=\widehat{u}(\ell)=0$ and using \eqref{BCs_at_ell} allows us to determine the constants $A$ and $B$. We find that the unique solution to \eqref{inhomog_eq} is 
		\begin{equation}\label{soln}
			\widehat{u}(x) = -\p_{\beta}\phi(0;\be_0)\,p(x) + \f{(-1)^{Q+1}\p_{\beta}\phi(0;\be_0) - \p_{\beta}\phi(\ell;\be_0)}{q(\ell)}\,q(x) + \p_{\beta}\phi(x;\beta_0 ). 
		\end{equation}
		Multiplying \eqref{soln} by $\phi_0$ and integrating the first two terms by parts yields \eqref{eqn:VK}. The statement regarding spectral stability follows immediately from \cref{thm:VK}.
	\end{proof}
	
	\begin{rem}
		The one-parameter family constructed abstractly in \eqref{beta_fam} via the implicit function theorem leads to the simplest expression for the VK-type integral on a compact interval. However, this is only useful in practice if one can determine this family explicitly, which may not be possible. For this reason, we have included formula \eqref{eqn:VK}, which holds for \textit{any}  one-parameter family of solutions to the standing wave equation that starts at $\phi_0$.
	\end{rem}
	
	\begin{rem}
		When the spatial domain is the entire real line, it is known that for power-law nonlinearities of the form $f(\phi^2) = \phi^{2p}$, $p>0$,  {strictly positive localised stationary states}   (for which $\be<0$, $P=1$ and $Q=0$) are spectrally stable\footnote{The critical case $p=2$ is spectrally stable but nonlinearly unstable due to algebraically growing solutions of the linearised system; see \cite[Remark 4.3, p.217]{pelinovsky}.} for $p\leq 2$ and spectrally unstable for $p>2$ (see \cite[Corollary 4.3, p.216]{pelinovsky}). The result follows from a change in sign of the VK integral \eqref{VK_classical} (see \cite[Theorem 4.4, p.215]{pelinovsky}).  {Moving to the compact interval, we investigated whether an analogous phenomenon holds for stationary states $\phi_0$ that likewise satisfy $\be<0$, $P=1$ and $Q=0$}.   We found that our numerical experiments were in line with the result on the real line when $p=1,2$, for which we found no spectrally unstable waves. Interestingly, however, for $p\in (2,p_0), p_0\approx5$, we observed the existence of a $\be$-dependent threshold value of the interval length $\ell=\ell^*$ separating spectral stability ($\ell<\ell^*)$ and spectral instability ($\ell>\ell^*$). This agrees with the instability result on the real line (for these values of $p$), in the sense that we recover it (numerically) upon taking $\ell\ra +\infty$.  {\Cref{thm:VK} indicates that this change in stability at $\ell=\ell^*$ should be reflected in a change in concavity of the eigenvalue curve passing through $(\la,s)=(0,1)$, and indeed we observe this numerically. }   \Cref{finalfig} displays the real eigenvalue curves for three  {$T$-periodic stationary states $\phi_0$ satisfying the Dirichlet boundary conditions $\phi_0(0)=\phi_0(\ell)=0$, $\ell=T/2$, for differing $\ell$.}   The sign of $\ddot s(0)$ at $(\la,s)=(0,1)$ switches from negative to positive as $\ell$ increases through $\ell=\ell^*$. By \cref{thm:VK} the underlying standing wave becomes unstable, which is confirmed by the emergence of a positive real eigenvalue in \cref{finalfigC}.
	\end{rem}
	
	\begin{figure}[h!]
		\hspace*{\fill}
		\subcaptionbox{ \label{finalfigA}$\ell \approx 2.1274$}
		{\includegraphics[width=0.28\textwidth]{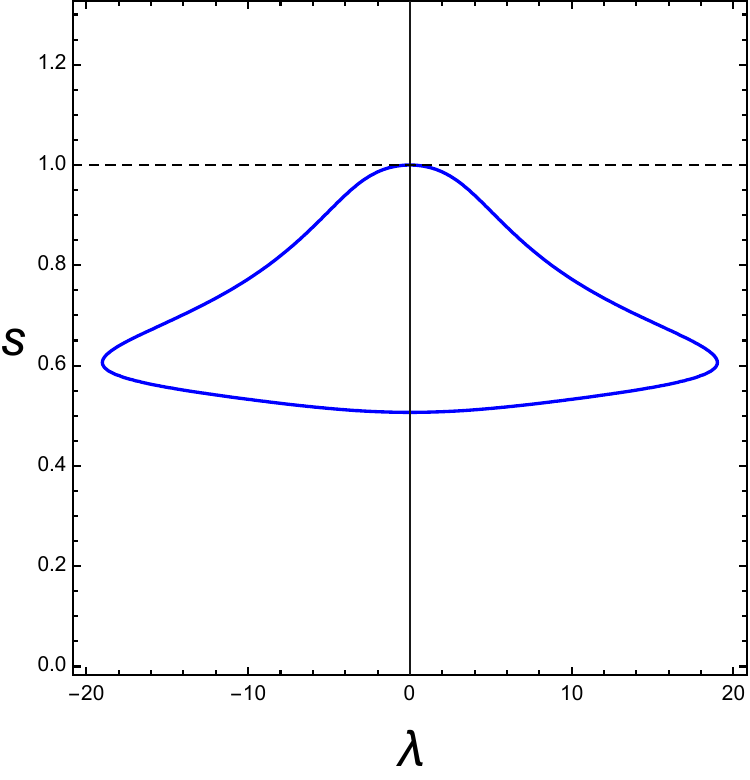}}\hfill 
		\subcaptionbox{\label{finalfigB}$\ell=\ell^* \approx 2.5666$ } 
		{\includegraphics[width=0.28\textwidth]{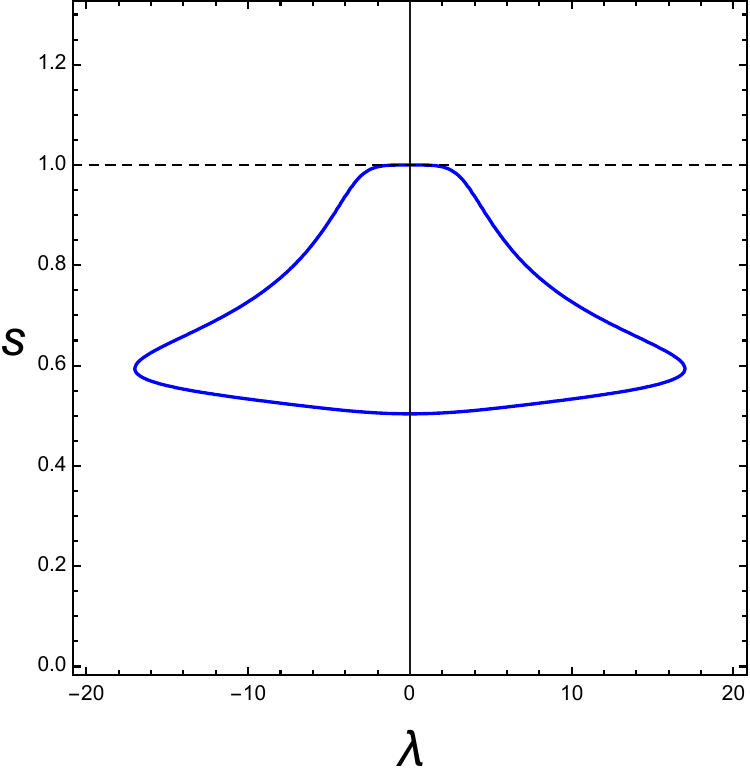}} \hfill 
		\subcaptionbox{\label{finalfigC}$\ell \approx 2.7760$}
		{\includegraphics[width=0.28\textwidth]{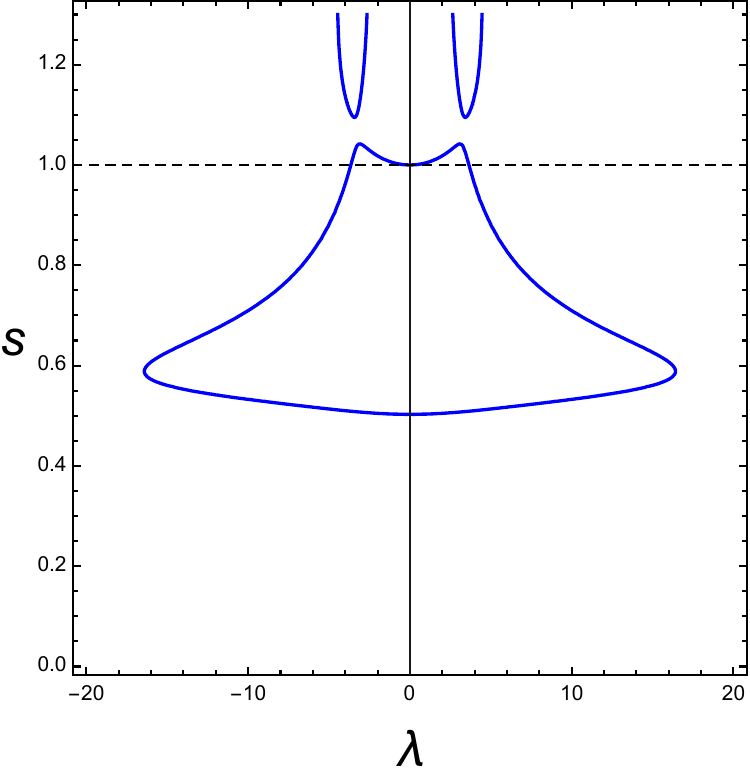}} 
		\hspace*{\fill}
		\caption{Eigenvalue curves $s^2\la \in\spec(N_s)\cap \R$ under \cref{hypo:NLS}(i) for {  $T$-periodic stationary states $\phi_0$ satisfying $\phi_0(0)=\phi_0(\ell)=0$}, with nonlinearity $f(\phi^2)=\phi^{6}$, $\be=-2$, and domain length $\ell=T/2$ indicated. {  These $\phi_0$ correspond to orbits located outside the homoclinic orbit and in the right half plane of \cref{fig:phaseA}. (Note the phase plane for \eqref{standwave} with $f(\phi^2)=\phi^{6}$ is qualitatively similar to \cref{fig:phaseA}.)} Eigenvalues of $N$ are given by intersections with the dashed line at $s=1$. At $\ell=\ell^*$, we computed $\ddot s(0)\approx0$ to four decimal places.} 
		\label{finalfig}
	\end{figure}

	\begin{rem}
		In the previous example, note that at the critical value $\ell=\ell^*$ we have $\dim\ker(N)=1$ and $\ddot s(0) = 0$. This corresponds to the non-generic case in \cref{rem:s_sharp} where $s^\sharp(0) \neq \sgn \ddot s(0)$ and the second order crossing form $\mathfrak{m}_{\la_0}^{(2)}$ in \cref{lemma:second_order_lambda_form} is degenerate. A brief calculation using the Fredholm Alternative indicates that the algebraic multiplicity of $\la=0\in\spec(N)$ is at least four. 
	\end{rem}
	
	\subsection{Connections with existing eigenvalue counts}\label{subsec:comparison}
	
	We now give a comparison of our lower bound \eqref{eq:bound_positive_evals} with the one given in \cite[Eq.(3.9)]{KKS04}  (see \eqref{Krein_lower_bound} below); see also \cite[Theorem 7.1.16]{KapProm}. We will show that the contribution to the Maslov index from the non-regular crossing (see \cref{defn:c}) is equal to the difference in negative indices of matrices arising in constrained eigenvalue counts for $L_\pm$. We refer the reader to \cite{CoxMarz19} for an alternate approach to the constrained eigenvalue problem using the Maslov index. Throughout this section, $\{\mathbf{u}_{1},\dots \mathbf{u}_{n} \}$ is a basis for $\ker(N)$ with $n\leq 2$. We assume the crossing $(\la_0,s_0)=(0,1)$ is non-regular in the $\la$ direction, with first order crossing form $\mathfrak{m}_{\la_0}$ in \eqref{maslov_horizontal} that is identically zero. We further assume that the second-order crossing form $\mathfrak{m}_{\la_0}^{(2)}$ in \eqref{second_order_lambda_form} is nondegenerate. The notation $n_-(A)$ refers to the number of negative eigenvalues of the selfadjoint operator or symmetric matrix $A$. Recall then that $P=n_-(L_+)$ and $Q=n_-(L_-)$. 
	
	Define the diagonal, selfadjoint operator
	\begin{equation}\label{}
		L \coloneqq \begin{pmatrix}
			L_+ & 0 \\ 0 & L_-
		\end{pmatrix}, \qquad \dom(L) \coloneqq \dom(N),
	\end{equation}
	so that $N=JL$. The eigenvalue problem \eqref{eq:N_EVP_full} may then be written as 
	\begin{equation}\label{}
		JL\mathbf{u} = \la \mathbf{u}, \qquad \mathbf{u}(0)=\mathbf{u}(\ell)=0.
	\end{equation}
	We denote the generalised eigenvectors of $N=JL$ by $\widehat{\mathbf{v}}_i$, i.e.
	\begin{equation}\label{genker}
		JL\widehat{\mathbf{v}}_i = \mathbf{u}_i, \quad JL\mathbf{u}_i=0, \quad \quad i=1,\dots,n.
	\end{equation}
	As in \cref{rem:contribute}, the Fredholm Alternative and the fact that $\mathfrak{m}_{\la_0}=0$ guarantee the existence of solutions to the first $n$ equations in \eqref{genker}, so the algebraic multiplicity of $\la=0$ is at least $2n$. Nondegeneracy of $\mathfrak{M}^{(2)}_{\la_0}$ then  implies the algebraic multiplicity is exactly $2n$. 
	
	The matrix $D$ in \cite[eq.(3.1)]{KKS04} is the $n\times n$ matrix with entries
	\begin{equation}\label{}
		D_{ij} = \langle \widehat{\mathbf{v}}_i, L \widehat{\mathbf{v}}_j\rangle=-\langle \widehat{\mathbf{v}}_i, J\mathbf{u}_j \rangle, 
	\end{equation}
	where the second equality follows since $ JL\widehat{\mathbf{v}}_i  = \mathbf{u}_i$ implies $L\widehat{\mathbf{v}}_i = J^{-1} \mathbf{u}_i = -J\mathbf{u}_i$. It is used to determine the number of negative eigenvalues of $L$ restricted to $\ran JL = \left [\ker (JL)^*\right ]^{\perp}$ (see \cite[Theorem 3.1]{KKS04}). Denoting $\dim\ker L_\pm = z_\pm \in\{0,1\}$ so that $z_+ + z_-=n$, notice that the off-diagonal structure of $JL$ implies that its eigenvectors and generalised eigenvectors may be written as
	\begin{align}\label{structure_evects}
		\mathbf{u}_i = \begin{cases}
			(u_i,0)^\top, \\
			(0,v_{i})^\top, 
		\end{cases} 
		\quad \widehat{\mathbf{v}}_i = \begin{cases}
			(0,\widehat{v}_i)^\top, \qquad &i=1,\dots,z_+, \\
			(\widehat{u}_{i},0)^\top, \qquad &i=z_++1,\dots,n, 
		\end{cases}
	\end{align}
	where, by \eqref{genker}, the functions $u_i, v_i, \widehat{u}_i, \widehat{v}_i$ satisfy
	\begin{align*}\label{}
		-L_-\widehat{v}_i& = u_i, \quad L_+u_i=0, \qquad  i=1,\dots,z_+,  \\
		L_+\widehat{u}_i& = v_i, \quad L_-v_i=0, \qquad \,i=z_++1,\dots,n. 
	\end{align*}
	The matrix $D$ thus has the block form (as in \cite[\S 3.3]{KKS04})
	\begin{equation*}\label{}
		D = \begin{pmatrix}
			D_- & 0 \\ 0 & D_+
		\end{pmatrix},
	\end{equation*}
	where
	\begin{gather}\label{D+D-}
		\begin{aligned}
			&   &&\left [D_-\right ]_{ij} = \langle \widehat{v}_i, L_-\widehat{v}_j\rangle = -\langle\widehat{v}_i, u_j \rangle,  &&i,j=1,\dots,z_+, & \\ 
			&   &&\left [D_+\right ]_{ij} =  \langle \widehat{u}_{z_++i}, L_+\widehat{u}_{z_++j}\rangle = \langle\widehat{u}_{z_++i}, v_{z_++j}\rangle,  &&i,j=1,\dots,z_-.&  
		\end{aligned}
	\end{gather}
	The matrices $D_+$ and $D_-$ are themselves used in constrained eigenvalue counts. Namely, if $D_+$ and $D_-$ are nondegenerate, then 
	\begin{equation}\label{constrained}
		n_-(\Pi L_+ \Pi) = P - n_-(D_+), \qquad  n_-(\Pi L_- \Pi) = Q - n_-(D_-), 
	\end{equation}
	where $\Pi$ is the orthogonal projection onto $ [\ker(L_-)\oplus \ker(L_+)]^\perp$ (see \cite[Lemma 3.1]{KKS04}).
	
	Now noticing that the entries of $\mathfrak{M}^{(2)}_{\la_0}$ are given by
	\begin{equation*}\label{}
		\left [\mathfrak{M}^{(2)}_{\la_0}\right ]_{ij} =  -2 \langle \widehat{\mathbf{v}}_i, S\mathbf{u}_{j} \rangle = \begin{cases}
			-2 \langle \widehat{v}_i, u_j \rangle, \qquad &i,j=1,\dots, z_+ \\
			-2 \langle\widehat{u}_i,v_j \rangle, \qquad &i,j=z_++1,\dots,n, \\
			0 \qquad& \text{elsewhere},
		\end{cases}
	\end{equation*}
	on account of \eqref{bigmfM2} and \eqref{structure_evects}, we are lead to the observation that 
	\begin{equation}\label{MD}
		\mathfrak{M}^{(2)}_{\la_0} = 2\begin{pmatrix}
			D_- & 0 \\ 0 & -D_+
		\end{pmatrix}.
	\end{equation}
	Clearly $\mathfrak{M}^{(2)}_{\la_0}$ is nonsingular if and only if $D_+$ and $D_-$ are nonsingular. Under this condition, in the notation of the current paper equation (3.9) from \cite{KKS04} reads
	\begin{equation}\label{Krein_lower_bound}
		n_+(N) \geq |n_-(\Pi L_+\Pi) - n_-(\Pi L_- \Pi)| = |P-Q-n_-(D_+) + n_-(D_-) |.
	\end{equation}
	Comparing \eqref{Krein_lower_bound} with \eqref{eq:bound_positive_evals}, we might na\"ively expect that $\mathfrak{c}= n_-(D_+) - n_-(D_-)$. We confirm this in the following proposition.
	
	\begin{prop}\label{prop:KreinMaslov}
		If $n\leq 2$ and $\mathfrak{M}^{(2)}_{\la_0}$ is nondegenerate, then
		\begin{equation}\label{Krein_Maslov}
			\mathfrak{c} = n_-(D_+) - n_-(D_-).
		\end{equation}
	\end{prop}
	That is, the contribution to the Maslov index from the crossing $(\la,s)=(0,1)$ is precisely the difference of the ``correction factors" counting the mismatch in negative dimensions between $L_\pm$ and their constrained counterparts (see \eqref{constrained}).
	\begin{proof}
		Recall the definition of $\mathfrak{b}$ given in the proof of \cref{prop:c}. By the same Proposition, if $n\leq 2$ we have 
		\begin{equation}\label{ceqn1}
			\mathfrak{b} 
			= -n_-(\mathfrak{M}^{(2)}_{\la_0}) = - \big ( n_-(D_-) + n_-(-D_+) \big ),
		\end{equation}
		where the last equality follows from \eqref{MD}. Notice that $D_+$ is a $z_-\times z_-$ matrix. Since $D_+$ is nondegenerate, it follows that
		\begin{equation}\label{}
			n_-(-D_+) = z_- - n_-(D_+).
		\end{equation}
		Thus, by \eqref{ceqn1},
		\begin{equation}\label{}
			\mathfrak{b} = -n_-(D_-) - (\dim\ker L_- - n_-(D_+)),
		\end{equation}
		and using \eqref{bc_dimkerL} and rearranging gives \eqref{Krein_Maslov}.
	\end{proof}
	
	A direct relationship between the matrices $D_\pm$ and the concavities of the eigenvalue curves follows from \cref{thm:sL},  \cref{lemma:second_order_lambda_form}, \cref{thm:touching} and equation \eqref{MD}. In particular, it is straightforward to show that:
	\begin{itemize}
		\item[(i)] If $0\in\spec(L_-) \backslash \spec(L_+)$ then $z_+=0$ and
		\begin{subequations}\label{eq:computingD+-}
			\begin{equation}
				\sgn \mathfrak{m}^{(2)}_{\la_0}(q) = - \sgn D_+  = -\sgn \ddot{s}(0).
			\end{equation}
			\item[(ii)] If $0\in\spec(L_+) \backslash \spec(L_-)$ then $z_-=0$ and
			\begin{equation}
				\sgn \mathfrak{m}^{(2)}_{\la_0}(q)  = \sgn D_- = \sgn \ddot{s}(0).
			\end{equation}
			\item[(iii)] If $0\in\spec(L_-) \cap \spec(L_+)$ then $z_-=z_+=1$ and 
			\begin{equation}\label{}
				\sgn \ddot s_1(0) = \sgn D_-, \qquad \sgn \ddot s_2(0) = \sgn D_+
			\end{equation}
		\end{subequations}
		(provided \eqref{uglycondition} holds so that $\sgn \ddot s_1(0) = - \sgn \langle \widehat{v}_1, u_{1} \rangle$ and $\sgn \ddot s_2(0) = \sgn \langle \widehat{u}_2, v_2 \rangle $).
	\end{itemize}
	
	We finish the present work with an application of our results to a formula relating the number of eigenvalues of $JL$ that are either unstable or susceptible to instability-inducing bifurcations, to the negative index of the constrained operator $L|_{X_c}$, $X_c\coloneqq \ran(JL)$, known as the \emph{Hamiltonian--Krein index theorem} (see \cite[Theorem 7.1.5]{KapProm} or \cite[Theorem 2.3]{LZ22}). For the eigenvalue problem \eqref{eq:N_EVP_full} -- \eqref{LplusLminus}, because $L$ is diagonal and the symplectic matrix $J$ is invertible, this formula reduces to that in \cite[Theorem 3.3]{KKS04}, which in the notation of the current paper reads 
	\begin{equation}\label{KKS}
		k_r + 2k_c +2k_i^- = P+Q - n_-(D_-)- n_-(D_+).
	\end{equation}
	Here, $k_r\coloneqq n_+(N)$, $k_c$ is the number eigenvalues lying in the open first quadrant, and $k_i^-$ is the number of eigenvalues on the positive imaginary axis with negative Krein signature (see \cite{KKS04}). Note that \eqref{KKS} holds provided $D_+$ and $D_-$ are nonsingular (and since $P,Q$ and $n$ are finite, where $\dim\ker(JL)= \f{1}{2}\dim\gker(JL) = n$; see \cite[\S7.1.3]{KapProm} or \cite{KKS04} for details). In light of our earlier results, this leads to the following.
	\begin{prop}\label{prop:HKM_fmls}
		Equation \eqref{KKS} may be written in one of the following equivalent forms:
		\vspace{-4mm}
		\begin{align}
			k_r + 2k_c +2k_i^- &= -\Mas(\Lambda,\cD;\Gamma_3^\e)  + 2P - 2n_-(D_+), \label{Krein_P} \\
			&= \Mas(\Lambda,\cD;\Gamma_3^\e)  + 2Q - 2n_-(D_-).\label{Krein_Q}
		\end{align}
	\end{prop}
	\begin{proof}
		Using \cref{prop:KreinMaslov} and \cref{lemma:Maslov_left} we can rearrange \eqref{KKS} to read
		\begin{equation}
			\quad  k_r + 2k_c +2k_i^- = \Mas(\Lambda,\cD;\Gamma_2^\e) + \mathfrak{c} + 2P - 2n_-(D_+). \label{Krein_P0}
		\end{equation}
		Then \eqref{Krein_P} follows from \eqref{Krein_P0} using \eqref{masleft=-mastop2}. A similar manipulation yields
		\begin{equation}
			k_r + 2k_c +2k_i^- = -\Mas(\Lambda,\cD;\Gamma_2^\e) - \mathfrak{c} + 2Q - 2n_-(D_-), \label{Krein_Q0}
		\end{equation}
		in which case \eqref{Krein_Q} follows from \eqref{Krein_Q0} via \eqref{masleft=-mastop2}.
	\end{proof}
	\begin{cor}\label{cor:PQ0}
		If $P=0$ or $Q=0$, then $k_c = k_i^-=0$.
		\vspace{-2.2mm}
	\end{cor}
	\begin{proof}
		If $P=0$, then by \cref{lemma:sign_definite}, we have $k_r=n_+(N) = -\Mas(\Lambda,\cD;\Gamma_3^\e) $. Furthermore, if $P=0$ then $L_+$ is a nonnegative operator in $L^2(0,\ell)$, and in particular $n_-(D_+)=0$. Cancelling terms on both sides of \eqref{Krein_P}, we get
		\begin{equation}\label{}
			2k_c +2k_i^- =0,
		\end{equation}
		as required. Note we could have argued that $k_c=0$ using \cref{lemma:imag_eigvals}. The case $Q=0$ is similar: $k_r=n_+(N) = \Mas(\Lambda,\cD;\Gamma_3^\e) $ by \cref{lemma:sign_definite}, and we have $L_-\geq 0$ in $L^2(0,\ell)$. Thus $n_-(D_-)=0$, and \eqref{Krein_Q} yields the result. 
	\end{proof}
	 {In the case that $L_\pm$ are invertible, the previous result agrees with that given in \cite[Corollary 2.26]{HK08}, where the dimension of intersecting cones is zero because $P=0$ or $Q=0$. The result for $Q=0$ is a special case of the formula in  \cite[Remark 3.1, Eq.(3.10)]{KKS04}.}
	 
	\begin{cor}\label{cor:kr=0}
		If either $k_r=0$ or the Maslov index of the path $\la \to \Lambda(\la,1), \la\in [\e,\la_\infty],\, 0<\e\ll 1$ is monotone in $\la$, then $k_c+k_i^- = Q-n_-(D_-) = P-n_-(D_+)$.
	\end{cor}
	\begin{proof}
		If $k_r=0$, the statement follows from \eqref{Krein_P} and \eqref{Krein_Q} upon noticing that $k_r=n_+(N)=0$ implies $\Mas(\Lambda,\cD;\Gamma_3^\e) =0$ by \eqref{eq:lowerbound}. 
		
		Monotonicity of the Lagrangian path stated means that the crossing form \eqref{maslov_horizontal} has the same sign at all crossings along $\Gamma_3$. In this case, $k_r=n_+(N) = \pm\Mas(\Lambda,\cD;\Gamma_3^\e) $ and the statement follows from \eqref{Krein_P} or \eqref{Krein_Q}.
	\end{proof}
	\begin{rem}
		Monotonicity in $\la$ is guaranteed if $P=0$ or $Q=0$. However, the Maslov index is in general not monotone when $P,Q\geq1$, and attempts to compute the terms $k_c$ and $k_i^-$ in these cases using the formulas above have so far been limited. 
	\end{rem}
	We finish with a numerical example to illustrate the scenario in \cref{cor:kr=0}.   In \cref{fig:complex} we have plotted the \textit{complex} eigenvalue curves for $s\in(0,1]$ under \cref{hypo:NLS}(i),   associated with a Jacobi cnoidal function $\phi_0$ (see \cref{fig:phaseA}) satisfying $\phi_0'(0) = \phi_0'(\ell)=0$.   Precisely, the blue curves represent real eigenvalues, the red curves represent imaginary eigenvalues, and the purple curves represent eigenvalues lying off the real and imaginary axes. It was computed that the minimum point of each blue connected component (for which $\la=0$) corresponds to a point of nontrivial kernel for $L_+^s$, while the maximum point of each such component corresponds to a point of nontrivial kernel for $L_-^s$.  Note that by a simple rescaling we can apply the formulas of the current section to the rescaled operators $N_s, L_\pm^s$ for \textit{any} $s\in(0,1]$. Consider then a horizontal plane at $s=s_*\approx 0.85$ in \cref{fig:complex}, which coincides with the maximum point of the top blue connected component. By the above considerations and \cref{lemma:morse_schrodinger} applied to the interval $(0,s^*)$ instead of $(0,1)$, we have $P=n_-(L_+^{s_*}) =3$ and $Q=n_-(L_-^{s_*})=2$. Since $0\in\spec(L_-^{s^*})\backslash\spec(L_+^{s^*})$, $D_-$ is null (see \eqref{D+D-}) and hence $n_-(D_-)=0$. \Cref{fig:complex} clearly shows $k_r=0$ for $s=s^*$, and by \cref{cor:kr=0} we deduce that $n_-(D_+)=1$ and $k_c+k_i^-=2$. (It was confirmed numerically that $k_c=2$.) A similar analysis can be done for any of the minima or maxima of the blue connected components in \cref{fig:complex}, or indeed for any horizontal plane which does not intersect the blue curves (for which $k_r=0$). 
	
	\begin{figure}[ht!]
		\hspace*{\fill}
		\subcaptionbox{ \label{} }{\includegraphics[width=0.45\textwidth,trim=0 -5em 0 0]{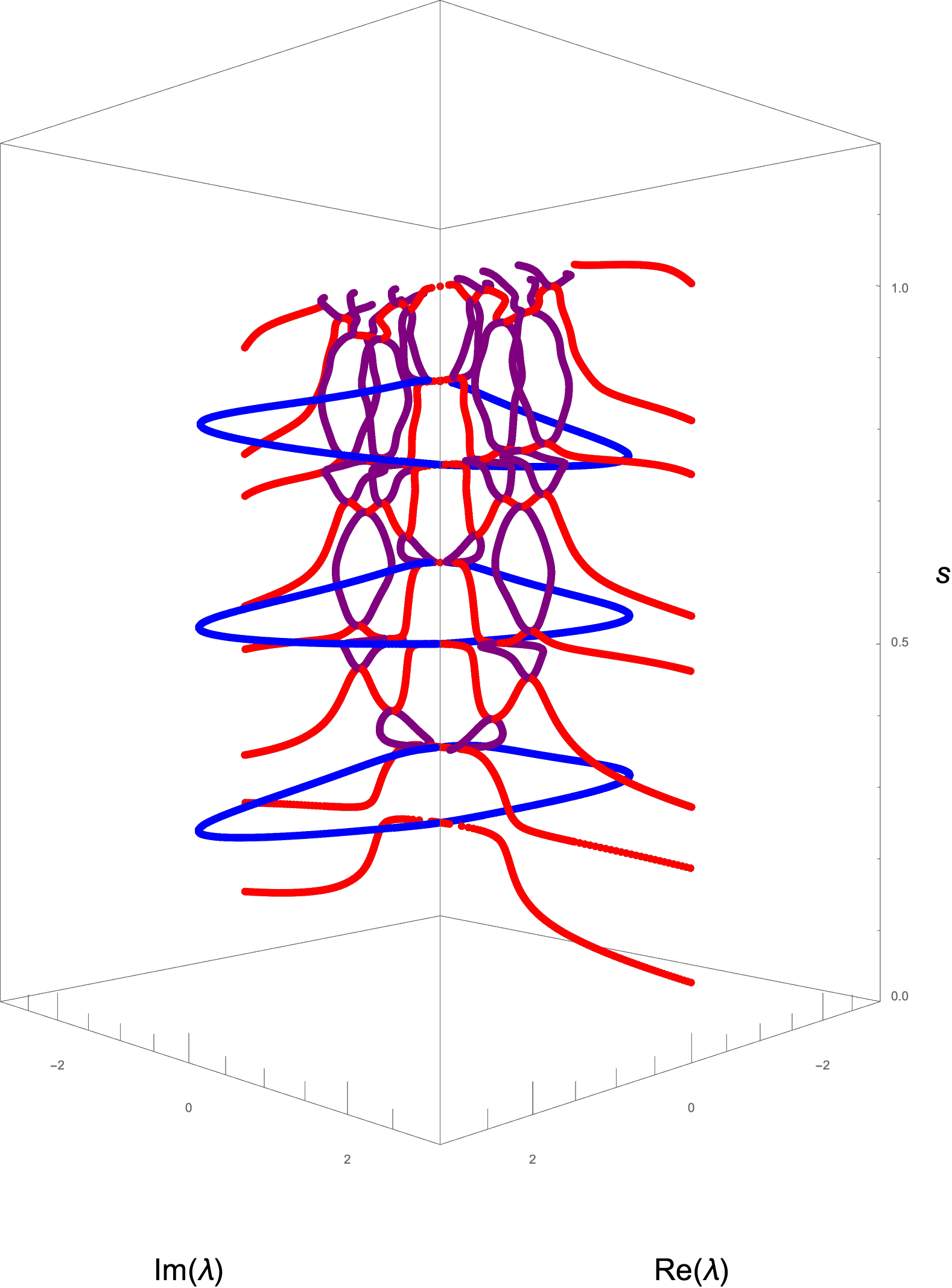}}
		\hfill
		\subcaptionbox{\label{} }{\includegraphics[width=0.45\textwidth]{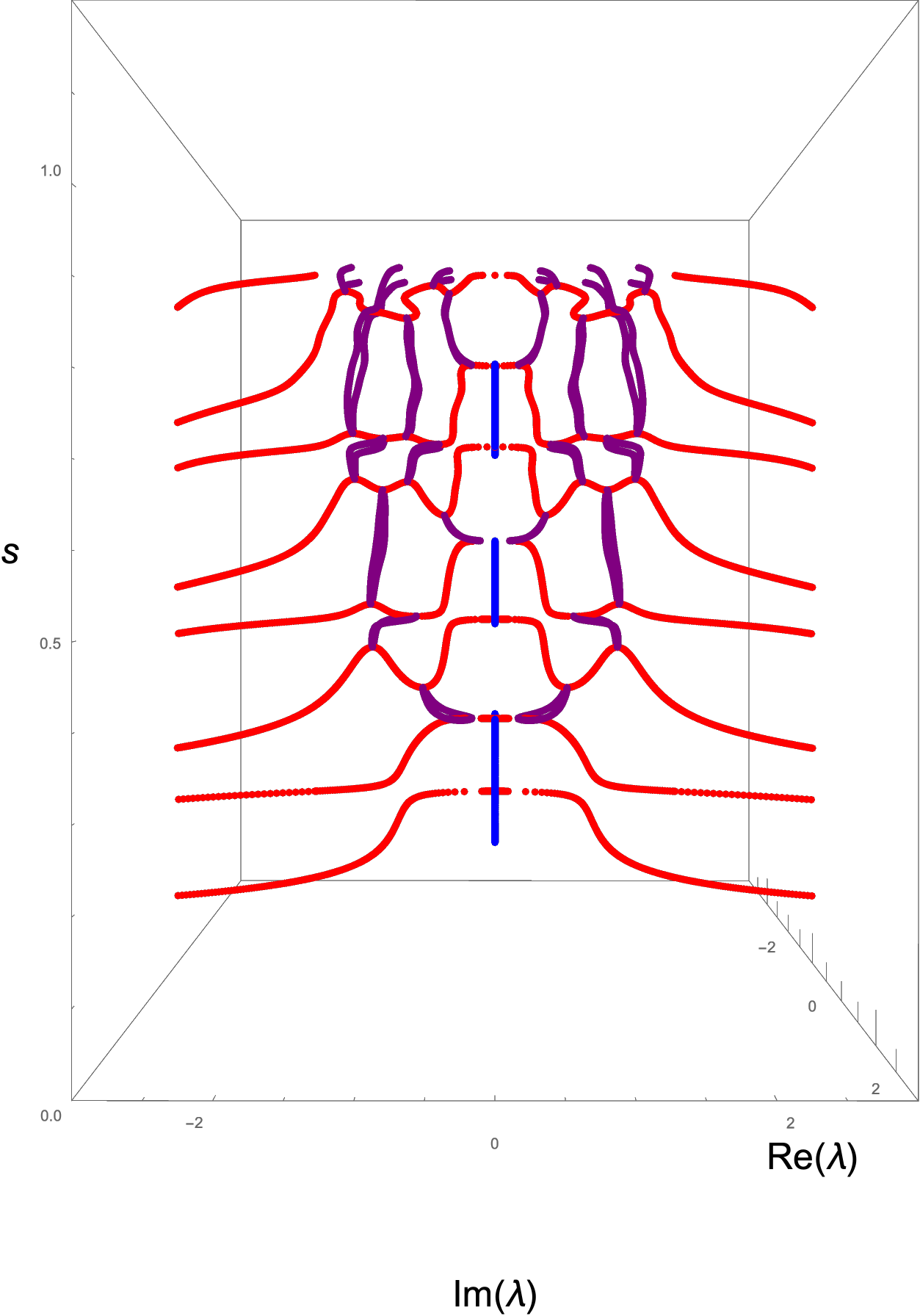}} 
		\hspace*{\fill}
		\caption{ Real (blue), imaginary (red) and complex (purple) eigenvalue curves $s^2\la \in \spec(N_s)\cap \C$, {  $\la\in [-3,3]\times [-3i,3i] \subset \C$, $s\in(0,1]$, under \cref{hypo:NLS}(i) for a $T$-periodic stationary state $\phi_0$ with $f(\phi^2)=\phi^2$ satisfying $\phi'_0(0)=\phi'_0(\ell)=0$, where $\ell=2T=13.3854$. Here, $\phi_0$ is a Jacobi cnoidal function corresponding to an orbit located outside the homoclinic orbit in \cref{fig:phaseA}. Figures (a) and (b) give two different viewpoints of the same curves.} The eigenvalues were computed using Mathematica's \texttt{NDEigenvalues} command.
		}
		\label{fig:complex}
	\end{figure}

	\subsection*{Acknowledgments}
	G.C. acknowledges the support of NSERC grant RGPIN-2017-04259. Y.L. was supported by the NSF grant DMS-2106157, and would like to thank the Courant Institute of Mathematical Sciences at New York University and especially Prof.\ Lai-Sang Young for the opportunity to visit CIMS. R.M. acknowledges the support of the Australian Research Council under grant DP210101102.

	\bibliographystyle{alpha}
	\bibliography{mybib}

\end{document}